\documentclass[12pt]{amsart}
\usepackage{mathrsfs,pstricks,ifpdf,tikz}
\usepackage{amsfonts}
\usepackage{enumerate}
\usepackage{makecell}
\newcommand{\bahao}{\fontsize{6pt}{\baselineskip}\selectfont}
\newcommand{\tabincell}[2]{\begin{tabular}{@{}#1@{}}#2\end{tabular}}
\usepackage{latexsym}
\usepackage{amsmath,amssymb}
\usepackage{amsthm}
\usepackage{ifthen,calc}
\usepackage{color}
\usepackage{graphicx}
\usepackage{latexsym}
\usepackage{multirow}
 \usepackage{diagbox}
 \usepackage{float}
\usepackage[format=hang, margin=10pt]{caption}

\newtheorem{theorem}{Theorem}[section]

\newtheorem{lemma}[theorem]{Lemma}

\newcounter{hours}
\newcounter{minutes}
\newcommand{\printtime}{
    \setcounter{hours}{\time/60}%
    \setcounter{minutes}{\time-\value{hours}*60}
    \ifthenelse{\value{hours}<10}{0}{}\thehours:%
    \ifthenelse{\value{minutes}<10}{0}{}\theminutes}

\begin{document}
\title{The thickness of $K_{1,n,n}$ and $K_{2,n,n}$}

\author{Xia Guo}
\address{Department of Mathematics, Tianjin University, 300072, Tianjin, China}
\email{guoxia@tju.edu.cn}
\author{Yan Yang}
\address{Department of Mathematics, Tianjin University, 300072, Tianjin, China}
\email{yanyang@tju.edu.cn    (Corresponding author: Yan YANG)}
\thanks{This work was supported by NNSF of China under Grant  No. 11401430}

\begin{abstract}
The thickness of a graph $G$ is the minimum number of planar subgraphs whose union is $G$. In this paper, we obtain the thickness of complete 3-partite graph $K_{1,n,n},K_{2,n,n}$ and complete 4-partite graph $K_{1,1,n,n}$.
\end{abstract}


\maketitle

\section{Introduction}

\noindent

The \textit{thickness} $\theta(G)$ of a graph $G$ is the minimum number of planar subgraphs whose union is $G$. It was first defined by W.T.Tutte \cite{Tut63} in 1963, then a few authors obtained  the thickness of hypercubes \cite{Kl67}, complete graphs \cite{AG76,BH65,Vas76} and complete bipartite graphs \cite{BHM64}. Naturally, people wonder about the thickness of the complete multipartite graphs.

A \textit{complete $k$-partite graph} is a graph whose vertex set can be partitioned into $k$ parts, such that every edge has its ends in different parts and every two vertices in different parts are adjacent.
Let $K_{p_1,p_2,\ldots,p_k}$ denote a complete $k$-partite graph in which the $i$th part contains $p_{i}$ $(1\leq i\leq k)$ vertices. For the complete 3-partite graph, Poranen proved $\theta(K_{n,n,n})\leq\big\lceil\frac{n}{2}\big\rceil$ in \cite{Por05}, then Yang \cite{Yan16} gave a new upper bound for $\theta(K_{n,n,n})$, i.e., $\theta(K_{n,n,n})\leq\big\lceil\frac{n+1}{3}\big\rceil+1$ and obtained $\theta(K_{n,n,n})=\big\lceil\frac{n+1}{3}\big\rceil$, when $n\equiv 3$ (mod $6$). And also Yang \cite{Yan14} gave the thickness number of $K_{l,m,n} (l\leq m\leq n)$ when $l+m\leq 5$ and showed that $\theta(K_{l,m,n})=\lceil\frac{l+m}{2}\rceil$ when $l+m$ is even and $n>\frac{1}{2}(l+m-2)^2$; or $l+m$ is odd and $n>(l+m-2)(l+m-1)$.

In this paper, we obtain the thickness of complete 3-partite graph $K_{1,n,n}$ and $K_{2,n,n}$, and we also deduce the thickness of complete 4-partite graph $K_{1,1,n,n}$ from that of $K_{2,n,n}$.


\section{The thickness of $K_{1,n,n}$ }

In \cite{BHM64}, Beineke, Harary and Moon gave the thickness of complete bipartite graphs $K_{m,n}$ for most value of $m$ and $n$, and their theorem implies the following result immediately.

\begin{lemma}\cite{BHM64}\label{2.1} The thickness of the complete bipartite graph $K_{n,n}$ is
$$\theta(K_{n,n})=\big\lceil\frac{n+2}{4}\big\rceil.$$ \end{lemma}

 In \cite{CYi16}, Chen and Yin gave a planar decomposition of the complete bipartite graph $K_{4p,4p}$ with $p+1$ planar subgraphs. Figure 1 shows their planar decomposition of $K_{4p,4p}$, in which $\{u_1,\ldots, u_{4p}\}=U$ and $\{v_1,\ldots, v_{4p}\}=V$ are the 2-partite vertex sets of it. Based on their decomposition, we give a planar decomposition of $K_{2,n,n}$ with $p+1$ subgraphs when $n\equiv 0$ or $3$ (mod $4$) and prove the following lemma.

\begin{figure}[htp]
\begin{center}
\begin{tikzpicture}                                                         
[inner sep=0pt]
\filldraw [black] (0,2.1) circle (1.4pt)
                  (2.1,0) circle (1.4pt)
                  (2.1,4.2) circle (1.4pt)
                  (4.2,2.1) circle (1.4pt);
\filldraw [black]     (1.4,2.1) circle (1.4pt)
                      (2.8,2.1) circle (1.4pt)
                      (2.1,1.4) circle (1.4pt)
                      (2.1,2.8) circle (1.4pt);
\filldraw [black] (0.7,2.5) circle (1.4pt)
                  (0.7,1.7) circle (1.4pt);
\filldraw [black]      (3.5,2.5) circle (1.4pt)
                       (3.5,1.7) circle (1.4pt);
\filldraw [black] (2.5,0.7) circle (1.4pt)
                  (1.7,0.7) circle (1.4pt);
\filldraw [black]      (2.5,3.5) circle (1.4pt)
                       (1.7,3.5) circle (1.4pt);
\filldraw [black] (0.7,2.2) circle (0.7pt)
                  (0.7,2.05) circle (0.7pt)
                  (0.7,1.9) circle (0.7pt);
\filldraw [black]      (3.5,2.2) circle (0.7pt)
                       (3.5,2.05) circle (0.7pt)
                       (3.5,1.9) circle (0.7pt);
\filldraw [black] (2.2,0.7) circle (0.7pt)
                  (2.05,0.7) circle (0.7pt)
                  (1.9,0.7) circle (0.7pt);
\filldraw [black] (2.2,3.5) circle (0.7pt)
                  (2.05,3.5) circle (0.7pt)
                  (1.9,3.5) circle (0.7pt);
\draw (0,2.1)--(2.1,0);\draw (2.1,0)--(4.2,2.1);\draw (2.1,4.2)--(0,2.1);\draw ((4.2,2.1)--(2.1,4.2);
\draw (1.4,2.1)--(2.1,1.4);\draw (2.1,1.4)--(2.8,2.1);\draw (2.8,2.1)--(2.1,2.8);\draw (1.4,2.1)--(2.1,2.8);
\draw (0,2.1)--(0.7,2.5);\draw (0,2.1)--(0.7,1.7);
\draw (1.4,2.1)--(0.7,2.5);\draw (1.4,2.1)--(0.7,1.7);
\draw (2.8,2.1)--(3.5,2.5);\draw (2.8,2.1)--(3.5,1.7) ;
\draw (4.2,2.1)--(3.5,2.5);\draw (4.2,2.1)--(3.5,1.7) ;
\draw (2.1,0)--(2.5,0.7);\draw (2.1,0)--(1.7,0.7);
\draw (2.1,1.4)--(2.5,0.7);\draw (2.1,1.4)--(1.7,0.7);
\draw (2.1,2.8)--(2.5,3.5);\draw (2.1,2.8)--(1.7,3.5);
\draw (2.1,4.2)--(2.5,3.5);\draw (2.1,4.2)--(1.7,3.5);
\draw (2.1,-0.2) node {\tiny $v_{4r}$}
      (2.1,4.35) node {\tiny $v_{4r-3}$}
      (-0.26,2.3) node {\tiny $u_{4r-1}$}
      (4.5,2.25) node {\tiny $u_{4r-2}$};
\draw    (2.55,1.35) node {\tiny $v_{4r-2}$}
          (1.68,2.83) node {\tiny $v_{4r-1}$}
           (1.4,1.85) node {\tiny $u_{4r}$}
             (2.35,2.1) node {\tiny $u_{4r-3}$};
\draw[-](2.1,4.2)..controls+(-1.2,-1.2)and+(-0.1,0.5)..(1.4,2.1);
\draw[-](0,2.1)..controls+(1.2,-1.2)and+(-0.5,-0.1)..(2.1,1.4);
\draw[-](2.1,0)..controls+(1.2,1.2)and+(0.1,-0.5)..(2.8,2.1);
\draw[-](4.2,2.1)..controls+(-1.2,1.2)and+(0.5,0.1)..(2.1,2.8);
\node(5) at(2.1,2.45)[circle,draw]{\scriptsize $5$};
\node(4) at(2.7,3.2)[circle,draw]{\scriptsize $4$};
\node(3) at(3.15,1.5)[circle,draw]{\scriptsize $3$};
\node(1) at(1,2.7)[circle,draw]{\scriptsize $1$};
\node(2) at(1.4,1)[circle,draw]{\scriptsize $2$};
\draw[-](0.7,2) to (-0.5,1.4);\draw (-1.2,1.2) node {\tiny $\bigcup\limits^p_{i=1,i\neq r}\{v_{4i-3},v_{4i-1}\}\triangleq V_1^r$};
\draw[-](3.6,2.1) to (4.2,3.2);\draw (4.8,3.5) node {\tiny $\bigcup\limits^p_{i=1,i\neq r}\{v_{4i-2},v_{4i}\}\triangleq V_2^r$};
\draw[-](2.00,0.8) to (3.5,1);\draw (5.00,1.0) node {\tiny $\bigcup\limits^p_{i=1,i\neq r}\{u_{4i-1},u_{4i}\}\triangleq U_2^r$};
\draw[-](2.05,3.5) to (0.7,3.9);\draw (-0.2,4.03) node {\tiny $\bigcup\limits^p_{i=1,i\neq r}\{u_{4i-3},u_{4i-2}\}\triangleq U_1^r$};
\end{tikzpicture}
\\ (a) The graph $G_r\ \ (1\leq r\leq p)$
\end{center}
\end{figure}
\begin{figure}[htp]
\begin{center}
\begin{tikzpicture}
[xscale=1]
\tikzstyle{every node}=[font=\tiny,scale=1]
[inner sep=0pt]\tiny

\filldraw [black] (0,1) circle (1.4pt)
                   (0,-0.5) circle (1.4pt)
                  (1,1) circle (1.4pt)
                   (1,-0.5) circle (1.4pt)
                   (3,1) circle (1.4pt)
                   (3,-0.5) circle (1.4pt)
                  (4,1) circle (1.4pt)
                  (4,-0.5) circle (1.4pt);
\draw(4,1)-- (4,-0.5);\draw(3,1)-- (3,-0.5);
\draw(1,1)-- (1,-0.5);\draw(0,1)-- (0,-0.5);
\filldraw [black] (1.8,0.3) circle (0.8pt)
                               (2,0.3) circle (0.8pt)
                               (2.2,0.3) circle (0.8pt);
\draw (-0.2,0.9) node {\tiny $u_{1}$}
      (-0.2,-0.4) node {\tiny $v_{1}$}
      (0.8,0.9) node {\tiny $u_{2}$}
      (0.8,-0.4) node {\tiny $v_{2}$};
\draw (3. 5,0.9) node {\tiny $u_{4p-1}$}
      (3.45,-0.4) node {\tiny $v_{4p-1}$}
      (4.35,0.9) node {\tiny $u_{4p}$}
      (4.3,-0.4) node {\tiny $v_{4p}$};
\end{tikzpicture}
\\ (b) The graph $G_{p+1}$
\caption{A planar decomposition of $K_{4p,4p}$}
\label{figure 1}
\end{center}
\end{figure}

\begin{lemma}\label{2.2} The thickness of the complete 3-partite graph $K_{1,n,n}$ and $K_{2,n,n}$ is
$$\theta(K_{1,n,n})=\theta(K_{2,n,n})=\big\lceil\frac{n+2}{4}\big\rceil,$$
when $n\equiv 0$ or $3$ (mod $4$).\end{lemma}

\begin{proof}~~Let the vertex partition of $K_{2,n,n}$ be $(X,U,V)$, where $X=\{x_1,x_2\}$, $U=\{u_1, \dots, u_{n}\}$ and $V=\{v_1, \dots,v_{n}\}$.

When $n\equiv 0$(mod $4$), let $n=4p ~(p\geq 1)$. Let $\{G_1,\dots,G_{p+1}\}$ be the planar decomposition of $K_{n,n}$ constructed by Chen and Yin in \cite{CYi16}. As shown in Figure 1, the graph $G_{p+1}$ consists of $n$ paths of length one. We put all the $n$ paths in a row, place vertex $x_1$ on one side of the row and the vertex $x_2$ on the other side of the row, join both $x_1$ and $x_2$ to all vertices in $G_{p+1}$. Then we get a planar graph, denote it by $\widehat{G}_{p+1}$.  It is easy to see that $\{G_1,\dots,G_p,\widehat{G}_{p+1}\}$ is a planar decomposition of $K_{2,n,n}$. Therefore, we have $\theta(K_{2,n,n})\leq p+1$. Since $K_{n,n}\subset K_{1,n,n}\subset K_{2,n,n}$, combining it with Lemma \ref{2.1}, we have
$$p+1=\theta(K_{n,n})\leq \theta(K_{1,n,n})\leq \theta(K_{2,n,n})\leq p+1,$$
that is,  $\theta(K_{1,n,n})=\theta(K_{2,n,n})=p+1$ when $n\equiv 0$(mod $4$).

When $n\equiv 3$(mod $4$), then $n=4p+3 ~(p\geq 0)$.  When $p=0$, from \cite{Yan14}, we have $\theta(K_{1,3,3})=\theta(K_{2,3,3})=2$.
When $p\geq 1$, since $K_{n,n}\subset K_{1,n,n}\subset K_{2,n,n}\subset K_{2,n+1,n+1}$, according to Lemma \ref{2.1} and $\theta(K_{2,4p,4p})= p+1$, we have
$$p+2=\theta(K_{n,n}) \leq \theta(K_{1,n,n})\leq \theta(K_{2,n,n})\leq \theta(K_{2,n+1,n+1})=p+2.$$
Then, we get $\theta(K_{1,n,n})=\theta(K_{2,n,n})=p+2$ when $n\equiv 3$(mod $4$).

Summarizing the above, the lemma is obtained.
\end{proof}


\begin{lemma}\label{2.3} There exists a planar decomposition of the complete 3-partite graph $K_{1,4p+2,4p+2}$ $(p\geq 0)$ with $p+1$ subgraphs.\end{lemma}

\begin{proof}~~Suppose the vertex partition of the complete 3-partite graph $K_{1,n,n}$ is $(X,U,V)$, where $X=\{x\}$, $U=\{u_1, \dots, u_{n}\}$ and $V=\{v_1, \dots,v_{n}\}$. When $n=4p+2$, we will construct a planar decomposition of $K_{1,4p+2,4p+2}$ with $p+1$ planar subgraphs to complete the proof. Our construction is based on the planar decomposition $\{G_1,$ $G_2,\dots,G_{p+1}\}$ of $K_{4p,4p}$ given in \cite{CYi16}, as shown in Figure 1 and the reader is referred to \cite{CYi16} for more details about this decomposition. For convenience, we denote the vertex set $\bigcup\limits^p_{i=1,i\neq r}\{u_{4i-3},u_{4i-2}\}$, $\bigcup\limits^p_{i=1,i\neq r}\{u_{4i-1},u_{4i}\}$,
$\bigcup\limits^p_{i=1,i\neq r}\{v_{4i-3},v_{4i-1}\}$ and $\bigcup\limits^p_{i=1,i\neq r}\{v_{4i-2},v_{4i}\}$ by $U_1^r$, $U_2^r$, $V_1^r$ and $V_2^r$ respectively.  We also label some faces of $G_r$ $(1\leq r\leq p)$, as indicated in Figure \ref{figure 1}, for example, the face $1$ is bounded by $v_{4r-3}u_{4r}v_{j}u_{4r-1}$ in which $v_{j}$ is some vertex from $V_1^r$.

 In the following, for $1\leq r\leq p+1$, by adding vertices $x,u_{4p+1},u_{4p+2},v_{4p+1},v_{4p+2}$ and some edges to $G_r$, and deleting some edges from $G_{r}$ such edges will be added to the graph $G_{p+1}$, we will get a new planar graph $\widehat{G}_r$ such that $\{\widehat{G}_1,\dots,\widehat{G}_{p+1}\}$ is a planar decomposition of $K_{1,4p+2,4p+2}$.  Because $v_{4r-3}$ and $v_{4r-1}$ in $G_{r}$ $(1\leq r\leq p)$ is joined by $2p-2$ edge-disjoint paths of length two that we call parallel paths, we can change the order of these parallel paths without changing the planarity of $G_{r}$.  For the same reason, we can do changes like this for parallel paths between $u_{4r-1}$ and $u_{4r}$, $v_{4r-2}$ and $v_{4r}$, $u_{4r-3}$ and $u_{4r-2}$.  We call this change by parallel paths modification for simplicity. All the subscripts of vertices are taken modulo $4p$, except that of $v_{4p+1}, v_{4p+2}, u_{4p+1}$ and $u_{4p+2}$ (the vertices we added to $G_{r}$).

{\bf Case 1.}~  When $p$ is even and $p> 2$.

\noindent {\bf (a)}~ The construction for $\widehat{G}_r$ , $1\leq r\leq p$, and $r$ is odd.

\noindent {\bf Step 1:}~~ Place the vertex $x$ in the face $1$ of $G_{r}$, delete edges $v_{4r-3}u_{4r}$ and $u_{4r}v_{4r-1}$ from $G_r$. Do parallel paths modification, such that $u_{4r+6}\in U_1^r$, $v_{4r+1}\in V_1^r$ and $u_{4r-3},u_{4r-1},u_{4r},v_{4r-3},v_{4r-2},v_{4r-1}$ are incident with a common face which  the vertex $x$ is in. Join $x$ to $u_{4r-3},u_{4r-1},u_{4r},v_{4r-3},v_{4r-2},v_{4r-1}$ and $u_{4r+6}$, $v_{4r+1}$.

\noindent {\bf Step 2:}~~ Do parallel paths modification, such that $u_{4r+11}, u_{4r+12}\in U_2^r $ are incident with a common face. Place the vertex $v_{4p+1}$ in the face, and join it to both $u_{4r+11}$ and $u_{4r+12}$.

\noindent {\bf Step 3:}~~ Do parallel paths modification, such that $u_{4r+7}, u_{4r+8}\in U_2^r $ are incident with a common face. Place the vertex $v_{4p+2}$ in the face, and join it to both $u_{4r+7}$ and $u_{4r+8}$.

\noindent {\bf Step 4:}~~ Do parallel paths modification, such that $v_{4r+10}, v_{4r+12}\in V_2^r $ are incident with a common face. Place the vertex $u_{4p+1}$ in the face, and join it to both $v_{4r+10}$ and $v_{4r+12}$.

\noindent {\bf Step 5:}~~ Do parallel paths modification, such that $v_{4r+6}, v_{4r+8}\in V_2^r $ are incident with a common face. Place the vertex $u_{4p+2}$ in the face, and join it to both $v_{4r+6}$ and $v_{4r+8}$.

\noindent {\bf (b)}~ The construction for $\widehat{G}_r$, $1\leq r\leq p$, and $r$ is even.

\noindent {\bf Step 1:}~~ Place the vertex $x$ in the face $3$ of $G_{r}$, delete edges $v_{4r}u_{4r-3}$ and $u_{4r-3}v_{4r-2}$ from $G_r$. Do parallel paths modification, such that $u_{4r+7}\in U_2^r$, $v_{4r+4}\in V_2^r$ and $u_{4r-3},u_{4r-2},u_{4r},v_{4r-2},v_{4r-1},v_{4r}$ are incident with a common face which  the vertex $x$ is in. Join $x$ to $u_{4r-3},u_{4r-2},u_{4r},v_{4r-2},v_{4r-1},v_{4r}$ and $u_{4r+7}$, $v_{4r+4}$.

\noindent {\bf Step 2:}~~ Do parallel paths modifications, such that $u_{4r+5}, u_{4r+6}\in U_1^r $, $u_{4r+1}, u_{4r+2}\in U_1^r $, $v_{4r+5}, v_{4r+7}\in V_1^r $,  $v_{4r+1}, v_{4r+3}\in V_1^r $ are incident with a common face, respectively. Join $v_{4p+1}$ to both $u_{4r+5}$ and $u_{4r+6}$, join $v_{4p+2}$ to both $u_{4r+1}$ and $u_{4r+2}$, join $u_{4p+1}$ to both $v_{4r+5}$ and $v_{4r+7}$, join $u_{4p+2}$ to both $v_{4r+1}$ and $v_{4r+3}$.

Table \ref{tab 1} shows how we add edges to $G_{r} (1\leq r\leq p)$ in Case 1. The first column lists the edges we added,  the second and third column lists the subscript of vertices, and we also indicate the vertex set which they belong to in brackets.

\begin{table}[htbp]
\centering\caption {The edges we add to $G_{r} (1\leq r \leq p)$ in Case 1}\label{tab 1}\tiny
\renewcommand{\arraystretch}{1.5}
\begin{tabular}{|*{10}{c|}}
\hline
\multicolumn{2}{|c|}{\diagbox {edge}{ subscript}{ case} } &  \multicolumn{4}{|c|}{$r$ is odd}  & \multicolumn{4}{|c|}{$r$ is even} \\
\hline
 \multicolumn{2}{|c|}{$xu_j$}     & \multicolumn{3}{|c|}{$4r-3,4r-1,4r$} & $4r+6~(U_1^r)$ & \multicolumn{3}{|c|}{$4r-3,4r-2,4r$} & $4r+7~(U_2^r)$    \\
\hline
 \multicolumn{2}{|c|}{$xv_j$} & \multicolumn{3}{|c|}{$4r-3,4r-2,4r-1$} & $4r+1~(V_1^r)$ & \multicolumn{3}{|c|}{$4r-2,4r-1,4r$} & $4r+4~(V_2^r)$  \\
\hline
 \multicolumn{2}{|c|}{$v_{4p+1}u_j$}  & \multicolumn{4}{|c|}{$4r+11,4r+12~(U_2^r)$}  & \multicolumn{4}{|c|}{$4r+5,4r+6~(U_1^r)$} \\
\hline
\multicolumn{2}{|c|}{$v_{4p+2}u_j$}
& \multicolumn{4}{|c|}{$4r+7,4r+8~(U_2^r)$}  & \multicolumn{4}{|c|}{$4r+1,4r+2~(U_1^r)$} \\
\hline
\multicolumn{2}{|c|}{$u_{4p+1}v_j$}   & \multicolumn{4}{|c|}{$4r+10,4r+12~(V_2^r)$}  & \multicolumn{4}{|c|}{$4r+5,4r+7~(V_1^r)$} \\
\hline
\multicolumn{2}{|c|}{$u_{4p+2}v_j$} & \multicolumn{4}{|c|}{$4r+6,4r+8~(V_2^r)$}  & \multicolumn{4}{|c|}{$4r+1,4r+3~(V_1^r)$} \\
\hline
\end{tabular}
\end{table}

\noindent{\bf (c)}~ The construction for $\widehat{G}_{p+1}$.

From the construction in (a) and (b), the subscript set of $u_{j}$ that $xu_{j}$ is an edge in $\widehat{G}_{r}$ for some $r\in \{1,\ldots, p\}$ is
$$\{4r-3, 4r-1, 4r, 4r+6 \mbox{(mod $4p$)}~|~1\leq r\leq p, ~\mbox{and $r$ is odd}\}$$
$$\cup\{4r-3, 4r-2, 4r, 4r+7 \mbox{(mod $4p$)}~|~1\leq r\leq p, ~\mbox{and $r$ is even}\}$$
$$=\{1,\ldots,p\}.$$
The subscript set of $u_{j}$ that $v_{4p+1}u_{j}$ is an edge in $\widehat{G}_{r}$ for some $r\in \{1,\ldots, p\}$ is
$$\{4r+11, 4r+12 \mbox{(mod $4p$)}~|~1\leq r\leq p, ~\mbox{and $r$ is odd}\}$$
$$\cup\{4r+5, 4r+6 \mbox{(mod $4p$)}~|~1\leq r\leq p, ~\mbox{and $r$ is even}\}$$
$$=\{4r-3, 4r-2, 4r-1, 4r~|~1\leq r\leq p, ~\mbox{and $r$ is even}\}.$$
Using the same procedure, we can list all the edges incident with $x$, $v_{4p+1}$, $v_{4p+2}$, $u_{4p+1}$ and $u_{4p+2}$ in $\widehat{G}_{r}$ $(1\leq r\leq p)$, so we can also list the edges that are incident with $x$, $v_{4p+1}$, $v_{4p+2}$, $u_{4p+1}$ in $K_{1,4p+2,4p+2}$ but not in any $\widehat{G}_{r}$ $(1\leq r\leq p)$.  Table \ref{tab 2} shows the edges that belong to $K_{1,4p+2,4p+2}$ but not to any $\widehat{G}_{r}$, $1\leq r\leq p$, in which the the fourth and fifth rows list the edges deleted form $G_r\ (1\leq r\leq p)$ in step one of (a) and (b),  and the sixth row lists the edges of $G_{p+1}$. The $\widehat{G}_{p+1}$ is the graph consists of the edges in Table \ref{tab 2}, Figure \ref{figure 2} shows $\widehat{G}_{p+1}$ is a planar graph.

\begin{table}[H]
\centering\caption {The edges of $\widehat{G}_{p+1}$ in Case 1} \label{tab 2}
\tiny
\renewcommand{\arraystretch}{1.6}
\begin{tabular}{|*{2}{c|}}
\hline
\Gape[8pt]edges & subscript\\
\hline
 $xv_{4p+1},xu_{4p+1},v_{4p+1}u_j,u_{4p+1}v_j$& $j=4r-3,4r-2,4r-1,4r,4p+2.$ ($r=1,3,\dots,p-1.$)\\
\hline
$xv_{4p+2},xu_{4p+2},v_{4p+2}u_j,u_{4p+2}v_j$ &$j=4r-3,4r-2,4r-1,4r,4p+1.$ ($r=2,4,\dots,p.$)\\
\hline
$v_{4r-3}u_{4r},u_{4r}v_{4r-1}$ &  $r=1,3,\dots,p-1.$ \\
\hline
$v_{4r}u_{4r-3},u_{4r-3}v_{4r-2}$ &  $r=2,4,\dots,p.$  \\
\hline
$u_{j}v_{j}$ &  $j=1,\dots,4p+2.$  \\
\hline
\end{tabular}
\end{table}

\vskip -0.4cm
\begin{figure}[H]
\begin{center}
\begin{tikzpicture}
[scale=0.95]
\tikzstyle{every node}=[font=\tiny,scale=0.95]
\tiny
[inner sep=0pt]
\filldraw [black] (0.7,1) circle (1.4pt)
                  (0.7,-0.5) circle (1.4pt)
                  (1.4,1) circle (1.4pt)
                  (1.4,-0.5) circle (1.4pt)
                   (2.1,1) circle (1.4pt)
                  (2.1,-0.5) circle (1.4pt)
                  (2.8,1) circle (1.4pt)
                  (2.8,-0.5) circle (1.4pt)
                  (4.2,1) circle (1.4pt)
                  (4.2,-0.5) circle (1.4pt)
                  (4.9,1) circle (1.4pt)
                  (4.9,-0.5) circle (1.4pt)
                  (5.6,1) circle (1.4pt)
                  (5.6,-0.5) circle (1.4pt)
                  (6.3,1) circle (1.4pt)
                  (6.3,-0.5) circle (1.4pt);
\draw (4.2,-0.5)--(4.9,1)-- (5.6,-0.5) ;\draw   (0.7,-0.5)--(1.4,1)-- (2.1,-0.5);
\draw (0.95,0.9)  node { $u_{6}$};\draw  (0.95,-0.4)node { $v_{6}$};
\draw (1.65,0.9)  node { $u_{5}$};\draw  (1.65,-0.4)node { $v_{5}$};
\draw (2.35,0.9)  node { $u_{8}$};\draw  (2.35,-0.4)node { $v_{8}$};
\draw (3,0.9)  node { $u_{7}$};\draw  (3,-0.4)node { $v_{7}$};
\draw [white](3.65,0.89) node [left]{}--(3.7,0.85)node [black,pos=0.02,above,sloped]{ $u_{4p-2}$}--(3.7,0.79) node [right]{};
\draw [white](4.35,0.89) node [left]{}--(4.4,0.86)node [black,pos=0.02,above,sloped]{ $u_{4p-3}$}--(4.4,0.79) node [right]{};
\draw [white](5.23,0.77) node [left]{}--(5.28,0.756)node [black,pos=0.02,above,sloped]{ $u_{4p}$}--(5.23,0.74) node [right]{};
\draw [white](5.8,0.857) node [left]{}--(5.85,0.84)node [black,pos=0.02,above,sloped]{ $ u_{4p-1}$}--(5.85,0.84) node [right]{};
\draw [white](3.75,-0.25) node [left]{}--(3.8,-0.34)node [black,pos=0.02,above,sloped]{ $v_{4p-2}$}--(3.8,-0.34) node [right]{};
\draw [white](4.49,-0.21) node [left]{}--(4.54,-0.34)node [black,pos=0.02,above,sloped]{$ v_{4p-3}$ }--(4.54,-0.40) node [right]{};
\draw [white](5.15,-0.4) node [left]{}--(5.17,-0.43)node [black,pos=0.02,above,sloped]{ $ v_{4p}$}--(5.17,-0.46) node [right]{};
\draw [white](5.82,-0.25) node [left]{}--(5.85,-0.30)node [black,pos=0.02,above,sloped]{ $ v_{4p-1}$}--(5.85,-0.37) node [right]{};
\filldraw [black] (-0.7,1) circle (1.4pt)    (-0.7,-0.5) circle (1.4pt)      (-1.4,1) circle (1.4pt)
                  (-1.4,-0.5) circle (1.4pt)            (-2.1,1) circle (1.4pt)            (-2.1,-0.5) circle (1.4pt)
                  (-2.8,1) circle (1.4pt)               (-2.8,-0.5) circle (1.4pt)        (-4.2,1) circle (1.4pt)
                  (-4.2,-0.5) circle (1.4pt)          (-4.9,1) circle (1.4pt)             (-4.9,-0.5) circle (1.4pt)
                  (-5.6,1) circle (1.4pt)               (-5.6,-0.5) circle (1.4pt)         (-6.3,1) circle (1.4pt)
                  (-6.3,-0.5) circle (1.4pt);
\draw (-6.3,1)--(-5.6,-0.5)-- (-4.9,1) ;\draw  (-2.8,1)-- (-2.1,-0.5)--(-1.4,1);
\draw (-6.45,0.8)  node { $v_{1}$};\draw  (-6.45,-0.3)node { $u_{1}$};
\draw (-5.85,0.9)  node { $v_{4}$};\draw  (-5.85,-0.4)node { $u_{4}$};
\draw (-5.15,0.9)  node { $v_{3}$};\draw  (-5.15,-0.4)node { $u_{3}$};
\draw (-4.4,0.9)  node { $v_{2}$};\draw  (-4.4,-0.4)node { $u_{2}$};
\draw [white](-3.35,0.87) node [left]{}--(-3.3,0.85)node [black,pos=0.02,above,sloped]{ $ v_{4p-7}$}--(-3.3,0.79) node [right]{};
\draw [white](-2.62,0.88) node [left]{}--(-2.57,0.86)node [black,pos=0.02,above,sloped]{ $ v_{4p-4}$}--(-2.57,0.79) node [right]{};
\draw [white](-1.92,0.825) node [left]{}--(-1.87,0.8)node [black,pos=0.02,above,sloped]{ $ v_{4p-5}$}--(-1.87,0.76) node [right]{};
\draw [white](-1.2,0.84) node [left]{}--(-1.15,0.82)node [black,pos=0.02,above,sloped]{ $ v_{4p-6}$}--(-1.15,0.82) node [right]{};

\draw [white](-3.32,-0.40) node [left]{}--(-3.27,-0.44)node [black,pos=0.02,above,sloped]{ $ u_{4p-7}$}--(-3.27,-0.44) node [right]{};
\draw [white](-2.62,-0.41) node [left]{}--(-2.57,-0.455)node [black,pos=0.02,above,sloped]{ $ u_{4p-4}$}--(-2.57,-0.46) node [right]{};
\draw [white](-1.85,-0.24) node [left]{}--(-1.83,-0.31)node [black,pos=0.02,above,sloped]{ $ u_{4p-5}$}--(-1.83,-0.31) node [right]{};
\draw [white](-1.2,-0.375) node [left]{}--(-1.17,-0.4)node [black,pos=0.02,above,sloped]{ $ u_{4p-6}$}--(-1.17,-0.4) node [right]{};
\filldraw [black] (0,0.25) circle (1.4pt);
\draw (0.2,0.25) node { $x$};
\draw[-] (0,0.25)..controls+(0.2,1.3)and+(-1,-0.1)..(3.5,1.9);
\draw[-] (0,0.25)..controls+(-0.2,1.3)and+(1,-0.1)..(-3.5,1.9);
\draw[-] (0,0.25)..controls+(0.2,-1.3)and+(-1,0.1)..(3.5,-1.4);
\draw[-] (0,0.25)..controls+(-0.2,-1.3)and+(1,0.1)..(-3.5,-1.4);
\filldraw [black] (-3.5,0.25) circle (1.0pt)
                  (-3.3,0.25) circle (1.0pt)
                  (-3.7,0.25) circle (1.0pt);
\filldraw [black] (3.5,0.25) circle (1.0pt)
                  (3.3,0.25) circle (1.0pt)
                  (3.7,0.25) circle (1.0pt)
                  (3.7,0.25) circle (1.0pt);
\filldraw [black] (-3.5,1.9) circle (1.4pt)
                  (-3.5,-1.4) circle (1.4pt)
                  (3.5,1.9) circle (1.4pt)
                  (3.5,-1.4) circle (1.4pt);
 \draw[-] (-3.5,1.9)..controls+(-4.5,-0.1)and+(-4.5,0.1)..  (-3.5,-1.4) ;
\draw[-] (3.5,1.9)..controls+(4.5,-0.1)and+(4.5,0.1)..(3.5,-1.4) ;
\draw  (-3.5,1.9)--(3.5,1.9); \draw   (-3.5,-1.4)--(3.5,-1.4);
\draw  (-3.5,2.1) node { $u_{4p+1}$};\draw (-3.5,-1.6)node { $v_{4p+1}$};
\draw(3.5,2.1) node { $v_{4p+2}$};\draw(3.5,-1.6)  node { $u_{4p+2}$};
\draw (-0.7,1) ..controls+(-0.1,0.2)and+(0.4,-0.1).. (-3.5,1.9);
\draw (-1.4,1)--(-3.5,1.9);\draw (-2.1,1)--(-3.5,1.9);
\draw (-2.8,1)--(-3.5,1.9);\draw (-4.2,1)--(-3.5,1.9);\draw (-4.9,1)--(-3.5,1.9);
\draw (-5.6,1)--(-3.5,1.9);\draw (-6.3,1)--(-3.5,1.9);
\draw (-3.5,-1.4)--(-0.7,-0.5);\draw (-3.5,-1.4)--(-1.4,-0.5);\draw (-3.5,-1.4)--(-2.1,-0.5);
\draw (-3.5,-1.4)--(-2.8,-0.5);\draw (-3.5,-1.4)--(-4.2,-0.5);\draw (-3.5,-1.4)--(-4.9,-0.5);
\draw (-3.5,-1.4)--(-5.6,-0.5);\draw (-3.5,-1.4)--(-6.3,-0.5);
\draw (0.7,1)--(3.5,1.9);\draw (1.4,1)--(3.5,1.9);\draw (2.1,1)--(3.5,1.9);
\draw (2.8,1)--(3.5,1.9);\draw (4.2,1)--(3.5,1.9);\draw (4.9,1)--(3.5,1.9);
\draw (5.6,1)--(3.5,1.9);\draw (6.3,1) ..controls+(-0.1,0.2)and+(0.4,-0.1).. (3.5,1.9);
\draw (3.5,-1.4)--(0.7,-0.5);\draw (3.5,-1.4)--(1.4,-0.5);\draw (3.5,-1.4)--(2.1,-0.5);
\draw (3.5,-1.4)--(2.8,-0.5);\draw (3.5,-1.4)--(4.2,-0.5);\draw (3.5,-1.4)--(4.9,-0.5);
\draw (3.5,-1.4)--(5.6,-0.5);\draw (3.5,-1.4)--(6.3,-0.5);
\draw (0.7,1)--(0.7,-0.5);\draw (1.4,1)--(1.4,-0.5);\draw (2.1,1)--(2.1,-0.5);
\draw (2.8,1)--(2.8,-0.5);\draw (4.2,1)--(4.2,-0.5);\draw (4.9,1)--(4.9,-0.5);
\draw (5.6,1)--(5.6,-0.5);\draw (6.3,1)--(6.3,-0.5);
\draw (-0.7,1)--(-0.7,-0.5);\draw (-1.4,1)--(-1.4,-0.5);\draw (-2.1,1)--(-2.1,-0.5);
\draw (-2.8,1)--(-2.8,-0.5);\draw (-4.2,1)--(-4.2,-0.5);\draw (-4.9,1)--(-4.9,-0.5);
\draw (-5.6,1)--(-5.6,-0.5);\draw (-6.3,1)--(-6.3,-0.5);
\end{tikzpicture}
\caption{The graph $\widehat{G}_{p+1}$ in Case 1}
\label{figure 2}
\end{center}
\end{figure}
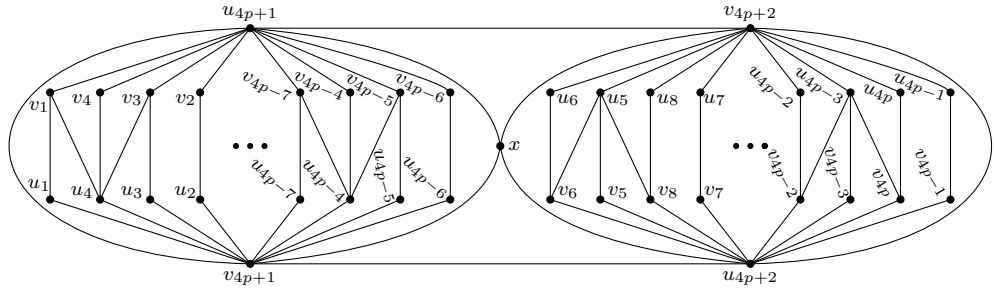

A planar decomposition $\{\widehat{G}_1,\dots,\widehat{G}_{p+1}\}$ of $K_{1,4p+2,4p+2}$ is obtained as above in this case.
In Figure \ref{figure 12}, we draw the planar decomposition of $K_{1,18,18}$, it is the smallest example for the Case $1$.  We denote vertex $u_i$ and $v_i$ by $i$ and $i'$ respectively in this figure.

\begin{center}
\begin{tikzpicture}
[scale=0.85]
\tikzstyle{every node}=[font=\tiny,scale=0.7]
\tiny
[inner sep=0pt]
\filldraw [black] (0.5,0) circle (1.3pt);\draw  (0.5,-0.18)  node { $1$};
\filldraw [black](3.6,0) circle (1.3pt);\draw  (3.6,-0.18)  node { $2$};
\filldraw [black](1.9,-1.35) circle (1.3pt);\draw  (1.95,-1.5)  node { $8'$};
\filldraw [black](1.9,-1) circle (1.3pt);\draw  (1.95,-1.16)  node { $6'$};
\filldraw [black](1.9,-0.6) circle (1.3pt);\draw  (1.95,-0.73)  node { $16'$};
\filldraw [black](1.9,-0.3) circle (1.3pt);\draw  (2.07,-0.29)  node { $17$};
\filldraw [black] (1.9,0) circle (1.3pt);\draw  (1.95,0.15)  node { $14'$};
\filldraw [black] (1.9,0.3) circle (1.3pt);\draw  (2.1,0.43)  node { $12'$};
\filldraw [black](1.9,0.6) circle (1.3pt);\draw  (1.7,0.6)  node { $18$};
\filldraw [black](1.9,0.9) circle (1.3pt);\draw  (1.65,0.9)  node { $10'$};

\filldraw [black] (-0.5,0) circle (1.3pt);\draw  (-0.54,-0.2)  node { $4$};
\filldraw [black](-3.6,0) circle (1.3pt);\draw  (-3.6,-0.2)  node { $3$};
 \filldraw [black](-1.9,0.9) circle (1.3pt);\draw  (-1.9,1.07)  node { $5'$};
 \filldraw [black] (-1.9,0.45) circle (1.3pt);\draw  (-1.9,0.61)  node { $7'$};
  \filldraw [black] (-1.9,0.1) circle (1.3pt);\draw  (-1.9,0.27)  node { $9'$};
 \filldraw [black](-1.9,-0.2) circle (1.3pt);\draw  (-1.9,-0.035)  node { $11'$};
 \filldraw [black] (-1.9,-0.55) circle (1.3pt);\draw  (-1.9,-0.39)  node { $13'$};
   \filldraw [black] (-1.9,-0.9) circle (1.3pt);\draw  (-1.9,-0.725)  node { $15'$};
\filldraw [black](0,-0.5) circle (1.3pt);\draw  (0.18,-0.50)  node { $2'$};
\filldraw [black](0,-3.6) circle (1.3pt);\draw  (0.2,-3.6)  node { $4'$};
\filldraw [black] (0.9,-1.9) circle (1.3pt);\draw  (0.97,-1.74)  node { $16$};
 \filldraw [black] (0.6,-1.9) circle (1.3pt);\draw  (0.6,-2.05)  node { $17'$};
\filldraw [black] (0.3,-1.9) circle (1.3pt);\draw  (0.41,-1.74)  node { $15$};
\filldraw [black] (0,-1.9) circle (1.3pt);\draw(-0.15,-2.04)  node { $12$};
\filldraw [black] (-0.3,-1.9) circle (1.3pt);\draw  (-0.3,-1.74)  node { $18'$};
 \filldraw [black] (-0.6,-1.9) circle (1.3pt);\draw  (-0.7,-1.74)  node { $11$};
\filldraw [black](-1,-1.9) circle (1.3pt);\draw  (-1.02,-1.74)  node { $8$};
\filldraw [black](-1.4,-1.9) circle (1.3pt);\draw  (-1.42,-1.74)  node { $7$};
\filldraw [black] (0,0.5) circle (1.3pt);\draw  (0.25,0.5)  node { $3'$};
\filldraw [black] (0,3.6) circle (1.3pt);\draw  (0.2, 3.65)  node { $1'$};
\filldraw [black](0.9, 1.9) circle (1.3pt);\draw  (1.07,1.9)  node { $14$};
\filldraw [black](0.45, 1.9) circle (1.3pt);\draw  (0.62,1.9)  node { $13$};
\filldraw [black](0.1, 1.9) circle (1.3pt);\draw  (0.24,1.9)  node { $9$};
\filldraw [black](-0.2,1.9) circle (1.3pt);\draw  (-0.06,1.9)  node { $6$};
\filldraw [black](-0.55, 1.9) circle (1.3pt);\draw  (-0.40,1.9)  node { $5$};
\filldraw [black](-0.9, 1.9) circle (1.3pt);\draw  (-0.74,1.9)  node { $10$};
\draw (3.6,0)--(0,-3.6)--(-3.6,0)--(0,3.6)--(3.6,0);
\draw (0,-0.5)-- (0.5,0)--(0,0.5);\draw (-0.5,0)--(0,-0.5);
\draw  (0,-0.5)..controls+(-1.7,-1.3)and+(0.5,-0.4)..(-3.6,0);
\draw  (0.5,0)..controls+(1.7,-1.3)and+(0.5,0.6)..(0, -3.6);
\draw  (0,0.5)..controls+(1.7,1.3)and+(-0.5,0.4)..(3.6,0);
\draw (0, -3.6)--(-1.4,-1.9)--(0,-0.5);\draw (0, -3.6)--(-1,-1.9)--(0,-0.5);
\draw (0, -3.6)--(-0.6,-1.9)--(0,-0.5);\draw (0, -3.6)--(0,-1.9)--(0,-0.5);
\draw (0, -3.6)--(0.3,-1.9)--(0,-0.5);\draw (0, -3.6)--(0.9,-1.9)--(0,-0.5);
\draw (-0.6,-1.9)--(0,-1.9);\draw (0.3,-1.9)--(0.9,-1.9);
\draw (0, 3.6)--(0.9,1.9)--(0,0.5);\draw (0, 3.6)--(0.45,1.9)--(0,0.5);
\draw (0, 3.6)--(0.1,1.9)--(0,0.5);\draw (0, 3.6)--(-0.2,1.9)--(0,0.5);
\draw (0, 3.6)--(-0.55,1.9)--(0,0.5);\draw (0, 3.6)--(-0.9,1.9)--(0,0.5);
\draw (-3.6 ,0)--(-1.9,0.9)--(-0.5,0);\draw (-3.6 ,0 )--(-1.9 ,0.45)--(-0.5, 0);
\draw (-3.6 ,0)--(-1.9,0.1)--(-0.5,0);\draw (-3.6 ,0 )--(-1.9 ,-0.2)--(-0.5, 0);
\draw (-3.6 ,0)--(-1.9,-0.55)--(-0.5,0);\draw (-3.6 ,0 )--(-1.9 ,-0.9)--(-0.5, 0);
\draw (3.6 ,0 )--(1.9 ,0.9)--(0.5, 0);\draw (3.6 ,0 )--(1.9 ,0.3)--(0.5, 0);
\draw (3.6 ,0)--(1.9,0)--(0.5,0);\draw (3.6 ,0)--(1.9,-0.6)--(0.5,0);
\draw (3.6 ,0 )--(1.9 ,-1)--(0.5, 0);\draw (3.6 ,0 )--(1.9 ,-1.35)--(0.5, 0);
\draw (1.9,0.9)--(1.9,0.3);\draw (1.9,0)--(1.9,-0.6);
\filldraw [black] (-1.2,1.1) circle (1.3pt);
\draw  (-1.28,0.87)  node { $x$};
\draw  (-1.2,1.1)--(-1.9,0.9);\draw  (-1.2,1.1)--(0 ,0.5);
\draw  (-1.2,1.1)--(-0.5, 0 );\draw  (-1.2,1.1)--(-0.9,1.9 );
\draw  (-1.2,1.1)..controls+(0.2,-0.3)and+(-0.1,0.5)..(0,-0.5);
\draw  (-1.2,1.1)..controls+(0.2,-0.3)and+(-0.5,0.1)..(0.5,0);
\draw  (-1.2,1.1)..controls+(-0.2,0.7)and+(-0.01,-0.01)..(0,3.6);
\draw  (-1.2,1.1)..controls+(-0.5,0.7)and+(0.4,0.3)..(-3.6,0);
\begin{scope}[xshift=8.2cm]
\filldraw [black] (-0.5,0) circle (1.3pt);\draw  (-0.5,0.18)  node { $8$};
\filldraw [black](-3.6,0) circle (1.3pt);\draw  (-3.58,0.18)  node { $7$};
\filldraw [black](-1.9,1.2) circle (1.3pt);\draw  (-1.98,1.3)  node { $1'$};
\filldraw [black](-1.9,0.9) circle (1.3pt);\draw  (-1.98,1)  node { $3'$};
\filldraw [black](-1.9,0.6) circle (1.3pt);\draw  (-2.05,0.67)  node { $13'$};
\filldraw [black](-1.9,0.3) circle (1.3pt);\draw  (-2.08,0.3)  node { $17$};
\filldraw [black] (-1.9,0) circle (1.3pt);\draw  (-1.7,0.13)  node { $15'$};
\filldraw [black] (-1.9,-0.3) circle (1.3pt);\draw  (-1.78,-0.15)  node { $9'$};
\filldraw [black](-1.9,-0.6) circle (1.3pt);\draw  (-2.09,-0.6)  node { $18$};
\filldraw [black](-1.9,-0.9) circle (1.3pt);\draw  (-1.66,-0.91)  node { $11'$};
\filldraw [black] (0.5,0) circle (1.3pt);\draw  (0.54,0.16)  node { $5$};
\filldraw [black](3.6,0) circle (1.3pt);\draw  (3.6,0.16)  node { $6$};
 \filldraw [black](1.9,-0.9) circle (1.3pt);\draw  (1.9,-1.05)  node { $12'$};
 \filldraw [black] (1.9,-0.45) circle (1.3pt);\draw  (1.9,-0.61)  node { $16'$};
  \filldraw [black] (1.9,-0.1) circle (1.3pt);\draw  (1.9,-0.25)  node { $14'$};
 \filldraw [black](1.9,0.2) circle (1.3pt);\draw  (1.9,0.06)  node { $10'$};
 \filldraw [black] (1.9,0.55) circle (1.3pt);\draw  (1.9,0.39)  node { $4'$};
   \filldraw [black] (1.9,0.9) circle (1.3pt);\draw  (1.9,0.76)  node { $2'$};
\filldraw [black](0,0.5) circle (1.3pt);\draw  (-0.18,0.53)  node { $7'$};
\filldraw [black](0,3.6) circle (1.3pt);\draw  (0.2,3.63)  node { $5'$};
\filldraw [black] (-0.9,1.9) circle (1.3pt);\draw  (-0.99,1.74)  node { $13$};
 \filldraw [black] (-0.6,1.9) circle (1.3pt);\draw  (-0.56,2.07)  node { $17'$};
\filldraw [black] (-0.3,1.9) circle (1.3pt);\draw  (-0.41,1.74)  node { $14$};
\filldraw [black] (-0,1.9) circle (1.3pt);\draw  (-0.10,1.74)  node { $9$};
\filldraw [black] (0.3,1.9) circle (1.3pt);\draw  (0.3,1.74)  node { $18'$};
 \filldraw [black] (0.6,1.9) circle (1.3pt);\draw  (0.77,1.9)  node { $10$};
\filldraw [black](1,1.9) circle (1.3pt);\draw  (1.03,1.74)  node { $1$};
\filldraw [black](1.4,1.9) circle (1.3pt);\draw  (1.43,1.74)  node { $2$};
\filldraw [black] (0,-0.5) circle (1.3pt);\draw  (-0.23,-0.5)  node { $6'$};
\filldraw [black] (0,-3.6) circle (1.3pt);\draw  (0.2, -3.6)  node { $8'$};
\filldraw [black](-1.2, -1.9) circle (1.3pt);\draw  (-1.35,-1.9)  node { $3$};
\filldraw [black](-0.8, -1.9) circle (1.3pt);\draw  (-0.93,-1.9)  node { $4$};
\filldraw [black](-0.4, -1.9) circle (1.3pt);\draw  (-0.56,-1.9)  node { $11$};
\filldraw [black](0,-1.9) circle (1.3pt);\draw  (-0.18,-1.9)  node { $12$};
\filldraw [black](0.45, -1.9) circle (1.3pt);\draw  (0.24,-1.9)  node { $16$};
\filldraw [black](0.9, -1.9) circle (1.3pt);\draw  (0.7,-1.9)  node { $15$};
\draw (-3.6,0)--(0,3.6)--(3.6,0)--(0,-3.6)--(-3.6,0);
\draw (0,0.5)-- (-0.5,0)--(0,-0.5);\draw (0.5,0)--(0,0.5);
\draw  (0,0.5)..controls+(1.7,1.3)and+(-0.5,0.4)..(3.6,0);
\draw  (-0.5,0)..controls+(-1.7,1.3)and+(-0.5,-0.6)..(0, 3.6);
\draw  (0,-0.5)..controls+(-1.9,-1.3)and+(0.5,-0.4)..(-3.6,0);
\draw (0, 3.6)--(1.4,1.9)--(0,0.5);\draw (0, 3.6)--(1,1.9)--(0,0.5);
\draw (0, 3.6)--(0.6,1.9)--(0,0.5);\draw (0, 3.6)--(0,1.9)--(0,0.5);
\draw (0, 3.6)--(-0.3,1.9)--(0,0.5);\draw (0, 3.6)--(-0.9,1.9)--(0,0.5);
\draw (0.6,1.9)--(0,1.9);\draw (-0.3,1.9)--(-0.9,1.9);
\draw (0, -3.6)--(-1.2,-1.9)--(0,-0.5);\draw (0, -3.6)--(-0.8,-1.9)--(0,-0.5);
\draw (0, -3.6)--(-0.4,-1.9)--(0,-0.5);\draw (0, -3.6)--(0,-1.9)--(0,-0.5);
\draw (0, -3.6)--(0.45,-1.9)--(0,-0.5);\draw (0, -3.6)--(0.9,-1.9)--(0,-0.5);
\draw (3.6 ,0)--(1.9,-0.9)--(0.5,0);\draw (3.6 ,0 )--(1.9 ,-0.45)--(0.5, 0);
\draw (3.6 ,0)--(1.9,-0.1)--(0.5,0);\draw (3.6 ,0 )--(1.9 ,0.2)--(0.5, 0);
\draw (3.6 ,0)--(1.9,0.55)--(0.5,0);\draw (3.6 ,0 )--(1.9 ,0.9)--(0.5, 0);
\draw (-3.6 ,0 )--(-1.9 ,-0.9)--(-0.5, 0);\draw (-3.6 ,0 )--(-1.9 ,-0.3)--(-0.5, 0);
\draw (-3.6 ,0)--(-1.9,0)--(-0.5,0);\draw (-3.6 ,0)--(-1.9,0.6)--(-0.5,0);
\draw (-3.6 ,0 )--(-1.9 ,0.9)--(-0.5, 0);\draw (-3.6 ,0 )--(-1.9 ,1.2)--(-0.5, 0);
\draw (-1.9,-0.9)--(-1.9,-0.3);\draw (-1.9,0)--(-1.9,0.6);
\filldraw [black] (1.2,-1.1) circle (1.3pt);
\draw  (1.28,-0.87)  node { $x$};
\draw  (1.2,-1.1)--(1.9,-0.9);\draw  (1.2,-1.1)--(0 ,-0.5);
\draw  (1.2,-1.1)--(0.5, 0 );\draw  (1.2,-1.1)--(0.9,-1.9 );
\draw  (1.2,-1.1)..controls+(-0.2,0.3)and+(0.1,-0.5)..(0,0.5);
\draw  (1.2,-1.1)..controls+(-0.2,0.3)and+(0.5,-0.1)..(-0.5,0);
\draw  (1.2,-1.1)..controls+(0.2,-0.7)and+(0.01,0.01)..(0,-3.6);
\draw  (1.2,-1.1)..controls+(0.5,-0.7)and+(-0.4,-0.3)..(3.6,0);
\end{scope}
\end{tikzpicture}
\vskip-0.03cm    (a)  The graph $\widehat{G}_1$\  \ \ \ \ \ \  \ \ \  \ \  \  \  \ \ \ \  \ \ \ \ \  \ \ \     (b)  The graph $\widehat{G}_2$
\end{center}

\begin{figure}[H]
\begin{center}
\begin{tikzpicture}
[scale=0.85]
\tikzstyle{every node}=[font=\tiny,scale=0.7]
\tiny
[inner sep=0pt]
\filldraw [black] (0.5,0) circle (1.3pt);\draw  (0.5,-0.18)  node { $9$};
\filldraw [black](3.6,0) circle (1.3pt);\draw  (3.62,-0.22)  node { $10$};
\filldraw [black](1.9,-1.4) circle (1.3pt);\draw  (1.85,-1.55)  node { $16'$};
\filldraw [black](1.9,-1) circle (1.3pt);\draw  (1.95,-1.15)  node { $14'$};
\filldraw [black](1.9,-0.6) circle (1.3pt);\draw  (1.95,-0.75)  node { $8'$};
\filldraw [black](1.9,-0.3) circle (1.3pt);\draw  (2.07,-0.3)  node { $17$};
\filldraw [black] (1.9,0) circle (1.3pt);\draw  (2.04,0.13)  node { $6'$};
\filldraw [black] (1.9,0.3) circle (1.3pt);\draw  (2.06,0.4)  node { $4'$};
\filldraw [black](1.9,0.6) circle (1.3pt);\draw  (1.71,0.6)  node { $18$};
\filldraw [black](1.9,0.9) circle (1.3pt);\draw  (1.7,0.9)  node { $2'$};

\filldraw [black] (-0.5,0) circle (1.3pt);\draw  (-0.54,-0.2)  node { $12$};
\filldraw [black](-3.6,0) circle (1.3pt);\draw  (-3.6,-0.25)  node { $11$};
 \filldraw [black](-1.9,0.9) circle (1.3pt);\draw  (-1.9,1.07)  node { $13'$};
 \filldraw [black] (-1.9,0.45) circle (1.3pt);\draw  (-1.9,0.64)  node { $1'$};
  \filldraw [black] (-1.9,0.1) circle (1.3pt);\draw  (-1.9,0.28)  node { $3'$};
 \filldraw [black](-1.9,-0.2) circle (1.3pt);\draw  (-1.9,-0.04)  node { $5'$};
 \filldraw [black] (-1.9,-0.55) circle (1.3pt);\draw  (-1.9,-0.39)  node { $7'$};
   \filldraw [black] (-1.9,-0.9) circle (1.3pt);\draw  (-1.85,-0.72)  node { $15'$};
\filldraw [black](0,-0.5) circle (1.3pt);\draw  (0.24,-0.51)  node { $10'$};
\filldraw [black](0,-3.6) circle (1.3pt);\draw  (0.24,-3.6)  node { $12'$};
\filldraw [black] (0.9,-1.9) circle (1.3pt);\draw  (0.93,-1.74)  node { $8$};
 \filldraw [black] (0.6,-1.9) circle (1.3pt);\draw  (0.6,-2.07)  node { $17'$};
\filldraw [black] (0.3,-1.9) circle (1.3pt);\draw  (0.38,-1.74)  node { $7$};
\filldraw [black] (0,-1.9) circle (1.3pt);\draw  (0.11,-1.74)  node { $4$};
\filldraw [black] (-0.3,-1.9) circle (1.3pt);\draw  (-0.3,-1.72)  node { $18'$};
 \filldraw [black] (-0.6,-1.9) circle (1.3pt);\draw  (-0.67,-1.74)  node { $3$};
\filldraw [black](-1,-1.9) circle (1.3pt);\draw  (-1.06,-1.74)  node { $16$};
\filldraw [black](-1.4,-1.9) circle (1.3pt);\draw  (-1.48,-1.74)  node { $15$};
\filldraw [black] (0,0.5) circle (1.3pt);\draw  (0.3,0.49)  node { $11'$};
\filldraw [black] (0,3.6) circle (1.3pt);\draw  (0.2, 3.65)  node { $9'$};
\filldraw [black](0.9, 1.9) circle (1.3pt);\draw  (1.08,1.9)  node { $14$};
\filldraw [black](0.45, 1.9) circle (1.3pt);\draw  (0.64,1.9)  node { $13$};
\filldraw [black](0.1, 1.9) circle (1.3pt);\draw  (0.25,1.9)  node { $6$};
\filldraw [black](-0.2,1.9) circle (1.3pt);\draw  (-0.06,1.9)  node { $5$};
\filldraw [black](-0.55, 1.9) circle (1.3pt);\draw  (-0.41,1.9)  node { $1$};
\filldraw [black](-0.9, 1.9) circle (1.3pt);\draw  (-0.74,1.9)  node { $2$};
\draw (3.6,0)--(0,-3.6)--(-3.6,0)--(0,3.6)--(3.6,0);
\draw (0,-0.5)-- (0.5,0)--(0,0.5);\draw (-0.5,0)--(0,-0.5);
\draw  (0,-0.5)..controls+(-1.7,-1.3)and+(0.5,-0.4)..(-3.6,0);
\draw  (0.5,0)..controls+(1.7,-1.3)and+(0.5,0.6)..(0, -3.6);
\draw  (0,0.5)..controls+(1.7,1.3)and+(-0.5,0.4)..(3.6,0);
\draw (0, -3.6)--(-1.4,-1.9)--(0,-0.5);\draw (0, -3.6)--(-1,-1.9)--(0,-0.5);
\draw (0, -3.6)--(-0.6,-1.9)--(0,-0.5);\draw (0, -3.6)--(0,-1.9)--(0,-0.5);
\draw (0, -3.6)--(0.3,-1.9)--(0,-0.5);\draw (0, -3.6)--(0.9,-1.9)--(0,-0.5);
\draw (-0.6,-1.9)--(0,-1.9);\draw (0.3,-1.9)--(0.9,-1.9);
\draw (0, 3.6)--(0.9,1.9)--(0,0.5);\draw (0, 3.6)--(0.45,1.9)--(0,0.5);
\draw (0, 3.6)--(0.1,1.9)--(0,0.5);\draw (0, 3.6)--(-0.2,1.9)--(0,0.5);
\draw (0, 3.6)--(-0.55,1.9)--(0,0.5);\draw (0, 3.6)--(-0.9,1.9)--(0,0.5);
\draw (-3.6 ,0)--(-1.9,0.9)--(-0.5,0);\draw (-3.6 ,0 )--(-1.9 ,0.45)--(-0.5, 0);
\draw (-3.6 ,0)--(-1.9,0.1)--(-0.5,0);\draw (-3.6 ,0 )--(-1.9 ,-0.2)--(-0.5, 0);
\draw (-3.6 ,0)--(-1.9,-0.55)--(-0.5,0);\draw (-3.6 ,0 )--(-1.9 ,-0.9)--(-0.5, 0);
\draw (3.6 ,0 )--(1.9 ,0.9)--(0.5, 0);\draw (3.6 ,0 )--(1.9 ,0.3)--(0.5, 0);
\draw (3.6 ,0)--(1.9,0)--(0.5,0);\draw (3.6 ,0)--(1.9,-0.6)--(0.5,0);
\draw (3.6 ,0 )--(1.9 ,-1)--(0.5, 0);\draw (3.6 ,0 )--(1.9 ,-1.4)--(0.5, 0);
\draw (1.9,0.9)--(1.9,0.3);\draw (1.9,0)--(1.9,-0.6);
\filldraw [black] (-1.2,1.1) circle (1.3pt);
\draw  (-1.28,0.87)  node { $x$};
\draw  (-1.2,1.1)--(-1.9,0.9);\draw  (-1.2,1.1)--(0 ,0.5);
\draw  (-1.2,1.1)--(-0.5, 0 );\draw  (-1.2,1.1)--(-0.9,1.9 );
\draw  (-1.2,1.1)..controls+(0.2,-0.3)and+(-0.1,0.5)..(0,-0.5);
\draw  (-1.2,1.1)..controls+(0.2,-0.3)and+(-0.5,0.1)..(0.5,0);
\draw  (-1.2,1.1)..controls+(-0.2,0.7)and+(-0.01,-0.01)..(0,3.6);
\draw  (-1.2,1.1)..controls+(-0.5,0.7)and+(0.4,0.3)..(-3.6,0);
\begin{scope}[xshift=8.2cm]
\filldraw [black] (-0.5,0) circle (1.3pt);\draw  (-0.48,0.22)  node { $16$};
\filldraw [black](-3.6,0) circle (1.3pt);\draw  (-3.6,0.22)  node { $15$};
\filldraw [black](-1.9,1.4) circle (1.3pt);\draw  (-1.72,1.5)  node { $9'$};
\filldraw [black](-1.9,1) circle (1.3pt);\draw  (-1.93,1.16)  node { $11'$};
\filldraw [black](-1.9,0.6) circle (1.3pt);\draw  (-1.95,0.75)  node { $7'$};
\filldraw [black](-1.9,0.3) circle (1.3pt);\draw  (-1.71,0.3)  node { $17$};
\filldraw [black] (-1.9,0) circle (1.3pt);\draw  (-2.05,0.13)  node { $5'$};
\filldraw [black] (-1.9,-0.3) circle (1.3pt);\draw  (-1.88,-0.15)  node { $3'$};
\filldraw [black](-1.9,-0.6) circle (1.3pt);\draw  (-2.11,-0.6)  node { $18$};
\filldraw [black](-1.9,-0.9) circle (1.3pt);\draw  (-1.66,-0.9)  node { $1'$};
\filldraw [black] (0.5,0) circle (1.3pt);\draw  (0.51,0.2)  node { $13$};
\filldraw [black](3.6,0) circle (1.3pt);\draw  (3.63,0.2)  node { $14$};
 \filldraw [black](1.9,-0.9) circle (1.3pt);\draw  (1.93,-1.05)  node { $4'$};
 \filldraw [black] (1.9,-0.55) circle (1.3pt);\draw  (1.93,-0.7)  node { $12'$};
  \filldraw [black] (1.9,-0.2) circle (1.3pt);\draw  (1.93,-0.35)  node { $10'$};
 \filldraw [black](1.9,0.15) circle (1.3pt);\draw  (1.93,-0.01)  node { $8'$};
 \filldraw [black] (1.9,0.55) circle (1.3pt);\draw  (1.93,0.39)  node { $6'$};
   \filldraw [black] (1.9,0.9) circle (1.3pt);\draw  (1.9,0.75)  node { $2'$};
\filldraw [black](0,0.5) circle (1.3pt);\draw  (-0.24,0.5)  node { $15'$};
\filldraw [black](0,3.6) circle (1.3pt);\draw  (0.22,3.65)  node { $13'$};
\filldraw [black] (-0.9,1.9) circle (1.3pt);\draw  (-0.96,1.74)  node { $5$};
 \filldraw [black] (-0.6,1.9) circle (1.3pt);\draw  (-0.58,2.08)  node { $17'$};
\filldraw [black] (-0.3,1.9) circle (1.3pt);\draw  (-0.4,1.74)  node { $6$};
\filldraw [black] (-0,1.9) circle (1.3pt);\draw  (-0.1,1.74)  node { $1$};
\filldraw [black] (0.3,1.9) circle (1.3pt);\draw  (0.3,1.74)  node { $18'$};
 \filldraw [black] (0.6,1.9) circle (1.3pt);\draw  (0.65,1.74)  node { $2$};
\filldraw [black](0.9,1.9) circle (1.3pt);\draw  (0.93,1.74)  node { $9$};
\filldraw [black](1.2,1.9) circle (1.3pt);\draw  (1.29,1.74)  node { $10$};
\filldraw [black] (0,-0.5) circle (1.3pt);\draw  (-0.3,-0.49)  node { $14'$};
\filldraw [black] (0,-3.6) circle (1.3pt);\draw  (0.23, -3.6)  node { $16'$};
\filldraw [black](-1.05, -1.9) circle (1.3pt);\draw  (-1.2,-1.9)  node { $3$};
\filldraw [black](-0.65, -1.9) circle (1.3pt);\draw  (-0.77,-1.9)  node { $4$};
\filldraw [black](-0.25, -1.9) circle (1.3pt);\draw  (-0.37,-1.9)  node { $8$};
\filldraw [black](0.15,-1.9) circle (1.3pt);\draw  (-0.03,-1.9)  node { $11$};
\filldraw [black](0.55, -1.9) circle (1.3pt);\draw  (0.35,-1.9)  node { $12$};
\filldraw [black](0.9, -1.9) circle (1.3pt);\draw  (0.74,-1.9)  node { $7$};
\draw (-3.6,0)--(0,3.6)--(3.6,0)--(0,-3.6)--(-3.6,0);
\draw (0,0.5)-- (-0.5,0)--(0,-0.5);\draw (0.5,0)--(0,0.5);
\draw  (0,0.5)..controls+(1.7,1.3)and+(-0.5,0.4)..(3.6,0);
\draw  (-0.5,0)..controls+(-1.5,1.8)and+(-0.2,-0.4)..(0, 3.6);
\draw  (0,-0.5)..controls+(-1.7,-1.3)and+(0.5,-0.4)..(-3.6,0);
\draw (0, 3.6)--(1.2,1.9)--(0,0.5);\draw (0, 3.6)--(0.9,1.9)--(0,0.5);
\draw (0, 3.6)--(0.6,1.9)--(0,0.5);\draw (0, 3.6)--(0,1.9)--(0,0.5);
\draw (0, 3.6)--(-0.3,1.9)--(0,0.5);\draw (0, 3.6)--(-0.9,1.9)--(0,0.5);
\draw (0.6,1.9)--(0,1.9);\draw (-0.3,1.9)--(-0.9,1.9);
\draw (0, -3.6)--(-1.05,-1.9)--(0,-0.5);\draw (0, -3.6)--(-0.65,-1.9)--(0,-0.5);
\draw (0, -3.6)--(-0.25,-1.9)--(0,-0.5);\draw (0, -3.6)--(0.15,-1.9)--(0,-0.5);
\draw (0, -3.6)--(0.55,-1.9)--(0,-0.5);\draw (0, -3.6)--(0.9,-1.9)--(0,-0.5);
\draw (3.6 ,0)--(1.9,-0.9)--(0.5,0);\draw (3.6 ,0 )--(1.9 ,-0.55)--(0.5, 0);
\draw (3.6 ,0)--(1.9,-0.2)--(0.5,0);\draw (3.6 ,0 )--(1.9 ,0.15)--(0.5, 0);
\draw (3.6 ,0)--(1.9,0.55)--(0.5,0);\draw (3.6 ,0 )--(1.9 ,0.9)--(0.5, 0);
\draw (-3.6 ,0 )--(-1.9 ,-0.9)--(-0.5, 0);\draw (-3.6 ,0 )--(-1.9 ,-0.3)--(-0.5, 0);
\draw (-3.6 ,0)--(-1.9,0)--(-0.5,0);\draw (-3.6 ,0)--(-1.9,0.6)--(-0.5,0);
\draw (-3.6 ,0 )--(-1.9 ,1)--(-0.5, 0);\draw (-3.6 ,0 )--(-1.9 ,1.4)--(-0.5, 0);
\draw (-1.9,-0.9)--(-1.9,-0.3);\draw (-1.9,0)--(-1.9,0.6);
\filldraw [black] (1.2,-1.1) circle (1.3pt);
\draw  (1.28,-0.87)  node { $x$};
\draw  (1.2,-1.1)--(1.9,-0.9);\draw  (1.2,-1.1)--(0 ,-0.5);
\draw  (1.2,-1.1)--(0.5, 0 );\draw  (1.2,-1.1)--(0.9,-1.9 );
\draw  (1.2,-1.1)..controls+(-0.2,0.3)and+(0.1,-0.5)..(0,0.5);
\draw  (1.2,-1.1)..controls+(-0.2,0.3)and+(0.5,-0.1)..(-0.5,0);
\draw  (1.2,-1.1)..controls+(0.2,-0.7)and+(0.01,0.01)..(0,-3.6);
\draw  (1.2,-1.1)..controls+(0.5,-0.7)and+(-0.4,-0.3)..(3.6,0);
\end{scope}
\end{tikzpicture}
\vskip-0.01cm    (c)  The graph $\widehat{G}_3$\  \ \ \ \ \ \  \  \ \ \ \ \  \ \ \ \ \ \ \ \ \ \  \  \ \ \ \  \ \ \ \ \  \ \ \    (d)  The graph $\widehat{G}_4$
\end{center}
\end{figure}

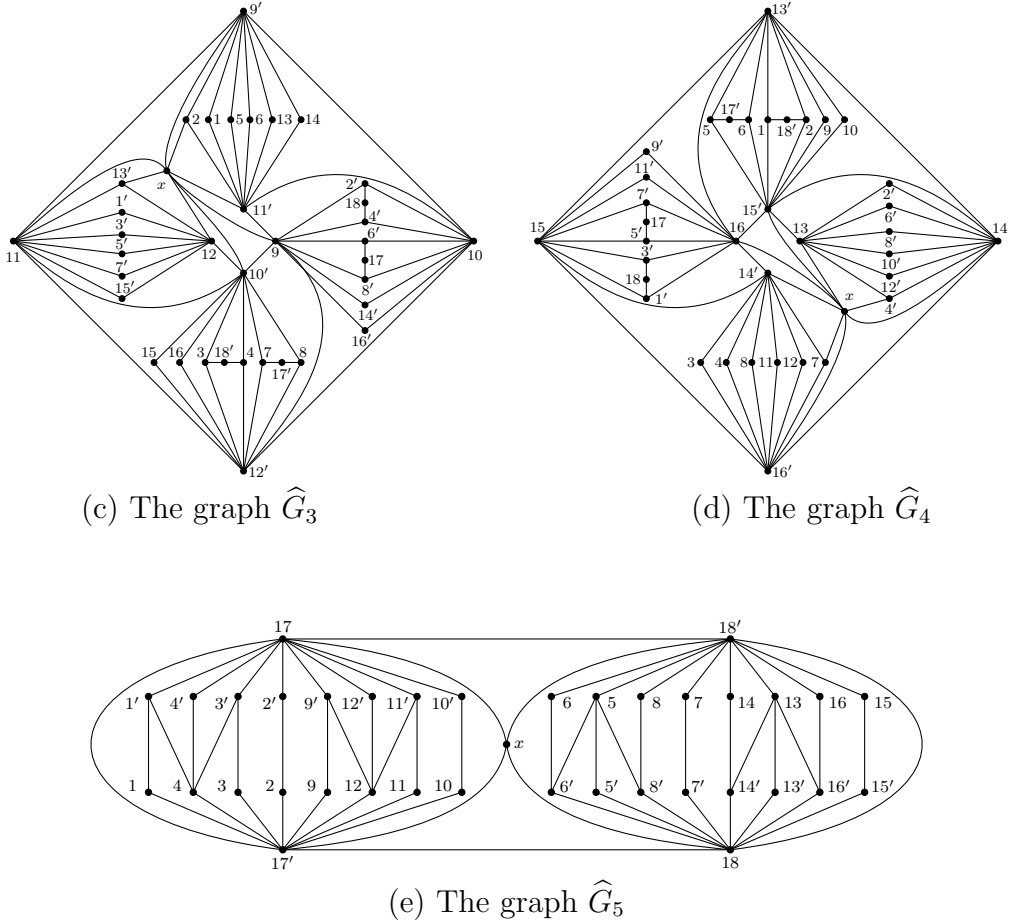
\begin{figure}[H]
\begin{center}
\begin{tikzpicture}
[scale=0.85]
\tikzstyle{every node}=[font=\tiny,scale=0.8]
\tiny
[inner sep=0pt]
\filldraw [black] (0.7,1) circle (1.4pt)
                  (0.7,-0.5) circle (1.4pt)
                  (1.4,1) circle (1.4pt)
                  (1.4,-0.5) circle (1.4pt)
                   (2.1,1) circle (1.4pt)
                  (2.1,-0.5) circle (1.4pt)
                  (2.8,1) circle (1.4pt)
                  (2.8,-0.5) circle (1.4pt)
                  (3.5,1) circle (1.4pt)
                  (3.5,-0.5) circle (1.4pt)
                  (4.2,1) circle (1.4pt)
                  (4.2,-0.5) circle (1.4pt)
                  (4.9,1) circle (1.4pt)
                  (4.9,-0.5) circle (1.4pt)
                  (5.6,1) circle (1.4pt)
                  (5.6,-0.5) circle (1.4pt);
\draw (3.5,-0.5)--(4.2,1)-- (4.9,-0.5) ;\draw   (0.7,-0.5)--(1.4,1)-- (2.1,-0.5);
\draw (0.95,0.9)  node { $6$};\draw  (0.95,-0.4)node { $6'$};
\draw (1.65,0.9)  node { $5$};\draw  (1.65,-0.4)node { $5'$};
\draw (2.35,0.9)  node { $8$};\draw  (2.35,-0.4)node { $8'$};
\draw (3,0.9)  node { $7$};\draw  (3,-0.4)node { $7'$};
\draw (3.75,0.9)  node { $14$};\draw  (3.8,-0.4)node { $14'$};
\draw (4.48,0.9)  node { $13$};\draw  (4.5,-0.4)node { $13'$};
\draw (5.18,0.9)  node { $16$};\draw  (5.2,-0.4)node { $16'$};
\draw (5.89,0.9)  node { $15$};\draw  (5.89,-0.4)node { $15'$};
\filldraw [black] (-0.7,1) circle (1.4pt)    (-0.7,-0.5) circle (1.4pt)      (-1.4,1) circle (1.4pt)
                  (-1.4,-0.5) circle (1.4pt)            (-2.1,1) circle (1.4pt)            (-2.1,-0.5) circle (1.4pt)
                  (-2.8,1) circle (1.4pt)               (-2.8,-0.5) circle (1.4pt)        (-3.5,1) circle (1.4pt)
                  (-3.5,-0.5) circle (1.4pt)          (-4.2,1) circle (1.4pt)             (-4.2,-0.5) circle (1.4pt)
                  (-5.6,1) circle (1.4pt)               (-5.6,-0.5) circle (1.4pt)         (-4.9,1) circle (1.4pt)
                  (-4.9,-0.5) circle (1.4pt);
\draw (-5.6,1)--(-4.9,-0.5)-- (-4.2,1) ;\draw  (-2.8,1)-- (-2.1,-0.5)--(-1.4,1);
\draw (-5.85,0.9)  node { $1'$};\draw  (-5.85,-0.4)node { $1$};
\draw (-5.15,0.9)  node { $4'$};\draw  (-5.15,-0.4)node { $4$};
\draw (-4.45,0.9)  node { $3'$};\draw  (-4.45,-0.4)node { $3$};
\draw (-3.7,0.9)  node { $2'$};\draw  (-3.7,-0.4)node { $2$};
\draw (-3.05,0.9)  node { $9'$};\draw  (-3.05,-0.4)node { $9$};
\draw (-2.4,0.9)  node { $12'$};\draw  (-2.4,-0.4)node { $12$};
\draw (-1.7,0.9)  node { $11'$};\draw  (-1.7,-0.4)node { $11$};
\draw (-0.99,0.9)  node { $10'$};\draw  (-0.99,-0.4)node { $10$};
\filldraw [black] (0,0.25) circle (1.4pt);
\draw (0.2,0.25) node { $x$};
\draw[-] (0,0.25)..controls+(0.2,1.3)and+(-1,-0.1)..(3.5,1.9);
\draw[-] (0,0.25)..controls+(-0.2,1.3)and+(1,-0.1)..(-3.5,1.9);
\draw[-] (0,0.25)..controls+(0.2,-1.3)and+(-1,0.1)..(3.5,-1.4);
\draw[-] (0,0.25)..controls+(-0.2,-1.3)and+(1,0.1)..(-3.5,-1.4);
\filldraw [black] (-3.5,1.9) circle (1.4pt)
                  (-3.5,-1.4) circle (1.4pt)
                  (3.5,1.9) circle (1.4pt)
                  (3.5,-1.4) circle (1.4pt);
 \draw[-] (-3.5,1.9)..controls+(-4,-0.5)and+(-4,0.5)..(-3.5,-1.4) ;
\draw[-] (3.5,1.9)..controls+(4,-0.5)and+(4,0.5)..(3.5,-1.4) ;
\draw  (-3.5,1.9)--(3.5,1.9); \draw   (-3.5,-1.4)--(3.5,-1.4);
\draw  (-3.5,2.1) node { $17$};\draw (-3.5,-1.6)node { $17'$};
\draw(3.5,2.1) node { $18'$};\draw(3.5,-1.6)  node { $18$};
\draw (-0.7,1)--(-3.5,1.9);\draw (-1.4,1)--(-3.5,1.9);\draw (-2.1,1)--(-3.5,1.9);
\draw (-2.8,1)--(-3.5,1.9);\draw (-4.2,1)--(-3.5,1.9);\draw (-4.9,1)--(-3.5,1.9);
\draw (-5.6,1)--(-3.5,1.9);\draw (-3.5,1)--(-3.5,1.9);
\draw (-3.5,-1.4)--(-0.7,-0.5);\draw (-3.5,-1.4)--(-1.4,-0.5);\draw (-3.5,-1.4)--(-2.1,-0.5);
\draw (-3.5,-1.4)--(-2.8,-0.5);\draw (-3.5,-1.4)--(-4.2,-0.5);\draw (-3.5,-1.4)--(-4.9,-0.5);
\draw (-3.5,-1.4)--(-5.6,-0.5);\draw (-3.5,-1.4)--(-3.5,-0.5);
\draw (0.7,1)--(3.5,1.9);\draw (1.4,1)--(3.5,1.9);\draw (2.1,1)--(3.5,1.9);
\draw (2.8,1)--(3.5,1.9);\draw (4.2,1)--(3.5,1.9);\draw (4.9,1)--(3.5,1.9);
\draw (5.6,1)--(3.5,1.9);\draw (3.5,1)--(3.5,1.9);
\draw (3.5,-1.4)--(0.7,-0.5);\draw (3.5,-1.4)--(1.4,-0.5);\draw (3.5,-1.4)--(2.1,-0.5);
\draw (3.5,-1.4)--(2.8,-0.5);\draw (3.5,-1.4)--(4.2,-0.5);\draw (3.5,-1.4)--(4.9,-0.5);
\draw (3.5,-1.4)--(5.6,-0.5);\draw (3.5,-1.4)--(3.5,-0.5);
\draw (0.7,1)--(0.7,-0.5);\draw (1.4,1)--(1.4,-0.5);\draw (2.1,1)--(2.1,-0.5);
\draw (2.8,1)--(2.8,-0.5);\draw (4.2,1)--(4.2,-0.5);\draw (4.9,1)--(4.9,-0.5);
\draw (5.6,1)--(5.6,-0.5);\draw (3.5,1)--(3.5,-0.5);
\draw (-0.7,1)--(-0.7,-0.5);\draw (-1.4,1)--(-1.4,-0.5);\draw (-2.1,1)--(-2.1,-0.5);
\draw (-2.8,1)--(-2.8,-0.5);\draw (-4.2,1)--(-4.2,-0.5);\draw (-4.9,1)--(-4.9,-0.5);
\draw (-5.6,1)--(-5.6,-0.5);\draw (-3.5,1)--(-3.5,-0.5);
\end{tikzpicture}
\vskip-0.01cm    (e)  The graph $\widehat{G}_5$
\caption{A planar decomposition of $K_{1,18,18}$}
\label{figure 12}
\end{center}
\end{figure}

{\bf Case 2.}~ When $p$ is odd and $p > 3$. The process is similar to that in Case 1.

\noindent{\bf (a)}~ The construction for  $\widehat{G}_r$, $1\leq r\leq p$, and $r$ is odd.

\noindent {\bf Step 1:}~~  Place the vertex $x$ in the face $1$ of $G_{r}$,  delete edges $v_{4r-3}u_{4r}$ and $u_{4r}v_{4r-1}$ from $G_r$, for $1\leq r\leq p$, and delete $v_2u_1$ from $G_1$ additionally.

For $1< r< p$,  do parallel paths modification to $G_{r}$, such that $u_{4r+6}\in U_1^r$, $v_{4r+1}\in V_1^r$ and $u_{4r-3},u_{4r-1},u_{4r},v_{4r-3},v_{4r-2},v_{4r-1}$ are incident with a common face which  the vertex $x$ is in. Join $x$ to $u_{4r-3},u_{4r-1},u_{4r},v_{4r-3},v_{4r-2},v_{4r-1}$ and $u_{4r+6}$, $v_{4r+1}$.

Similarly, in $G_1$, join $x$ to $u_{1},u_{3},u_{4},v_{1},v_{2},v_{3},v_4$ and $u_{10}\in U_1^1$, $v_{5}\in V_1^1$. In $G_{p}$, join $x$ to $u_{4p-3}$,$u_{4p-1}$,$u_{4p}$,$v_{4p-3}$,$v_{4p-2}$,$v_{4p-1}$ and $u_{2}\in U_1^p$.

\noindent {\bf Step 2:}~~  For $1\leq r< p$, do parallel paths modification to $G_{r}$, such that $u_{4r+11}$,  $u_{4r+12}\in U_2^r $, $u_{4r+7}, u_{4r+8}\in U_2^r $, $v_{4r+10}, v_{4r+12}\in V_2^r$ and $v_{4r+6}, v_{4r+8}\in V_2^r$ are incident with a common face, respectively. Join $v_{4p+1}$ to both $u_{4r+11}$ and $u_{4r+12}$, join $v_{4p+2}$ to both $u_{4r+7}$ and $u_{4r+8}$, join $u_{4p+1}$ to both $v_{4r+10}$ and $v_{4r+12}$, join $u_{4p+2}$ to both $v_{4r+6}$ and $v_{4r+8}$.

Similarly, in $G_p$, join $v_{4p+1}$ to $u_{5}, u_{6}\in U_1^p$, join $v_{4p+2}$ to $u_{7}, u_{8}\in U_2^p$, join $u_{4p+1}$ to $v_{6}, v_{8}\in V_2^p$, join $u_{4p+2}$ to $v_{5}, v_{7}\in V_1^p$.

\noindent{\bf (b)}~ The construction for $\widehat{G}_r$, $1\leq r\leq p$, and $r$ is even.

\noindent {\bf Step 1:}~~  Place the vertex $x$ in the face $3$ of $G_{r}$,  delete edges $v_{4r}u_{4r-3}$ and $u_{4r-3}v_{4r-2}$ from $G_r$, $1\leq r\leq p-1$.

Do parallel paths modification to $G_r, 1\leq r <p-1$, such that $u_{4r+7}\in U_2^r$, $v_{4r+4}\in V_2^r$ and $u_{4r-3},u_{4r-2},u_{4r},v_{4r-2},v_{4r-1},v_{4r}$ are incident with a common face which  the vertex $x$ is in. Join $x$ to $u_{4r-3},u_{4r-2},u_{4r},v_{4r-2},v_{4r-1},v_{4r}$ and $u_{4r+7}$, $v_{4r+4}$. Similarly, in $G_{p-1}$, join $x$ to $u_{4p-7},u_{4p-6},u_{4p-4},v_{4p-6},v_{4p-5},v_{4p-4}$ and $u_{7}\in U_2^{p-1}$, $v_{4p}\in V_2^{p-1}$.

\noindent {\bf Step 2:}~~  Do parallel paths modifications, such that $u_{4r+5}, u_{4r+6}\in U_1^r $, $u_{4r+1}, u_{4r+2}\in U_1^r $, $v_{4r+5}, v_{4r+7}\in V_1^r $, $v_{4r+1}, v_{4r+3}\in V_1^r $ are incident with a common face, respectively. Join $v_{4p+1}$ to both $u_{4r+5}$ and $u_{4r+6}$,  join $v_{4p+2}$ to both $u_{4r+1}$ and $u_{4r+2}$, join $u_{4p+1}$ to both $v_{4r+5}$ and $v_{4r+7}$, join $u_{4p+2}$ to both $v_{4r+1}$ and $v_{4r+3}$.

Table \ref{tab 3} shows how we add edges to $G_{r}(1\leq r\leq p)$ in Case 2.

\begin{table}[htbp]
\centering\caption {The edges we add to $G_{r} (1\leq r \leq p)$ in Case 2}\label{tab 3}\tiny
\renewcommand{\arraystretch}{1.1}
\begin{tabular}{|*{10}{c|}}
\hline
\multicolumn{2}{|c|}{\diagbox { \bahao edge}{\bahao subscript}{\bahao case} } &  \multicolumn{4}{|c|}{$r$ is odd}  & \multicolumn{4}{|c|}{$r$ is even} \\
\hline
\multicolumn{2}{|c|}{$xu_j$ } & \multicolumn{3}{|c|}{$4r-3,4r-1,4r$}
 & \tabincell{c}{$4r+6,r\neq p~(U_1^r)$ \\  $2,r=p~(U_1^r)$}
 & \multicolumn{3}{|c|}{$4r-3,4r-2,4r$}
 &  \tabincell{c}{ $4r+7,r\neq p-1~(U_2^r)$ \\ $7,r=p-1~(U_2^r)$}    \\
\hline
\multicolumn{2}{|c|}{$xv_j$ } & \multicolumn{3}{|c|}{$4r-3,4r-2,4r-1$} & \tabincell{c}{ $4,5,r=1$ \\ $4r+1,r\neq1, p~(V_1^r)$ }
& \multicolumn{3}{|c|}{$4r-2,4r-1,4r$}
&  $4r+4 ~(V_2^r)$
   \\
\hline
\multicolumn{2}{|c|}{$v_{4p+1}u_j$}
& \multicolumn{4}{|c|}{ \tabincell{c}{ $4r+11,4r+12,r\neq p~(U_2^r)$\\$5,6,r=p~(U_1^r)$}}
&\multicolumn{4}{|c|}{$4r+5,4r+6~(U_1^r)$}\\
\hline
\multicolumn{2}{|c|}{$v_{4p+2}u_j$}
& \multicolumn{4}{|c|}{  $\Gape[8pt]4r+7,4r+8~(U_2^r)$}
&\multicolumn{4}{|c|}{  $4r+1,4r+2~(U_1^r)$}\\
\hline
\multicolumn{2}{|c|}{$u_{4p+1}v_j$}
& \multicolumn{4}{|c|}{ \tabincell{c}{ $ 4r+10,4r+12, r\neq p~(V_2^r)$\\$6,8,r=p~(V_2^r)$}}
&\multicolumn{4}{|c|}{  $4r+5,4r+7~(V_1^r)$}\\
\hline
\multicolumn{2}{|c|}{$u_{4p+2}v_j$}
& \multicolumn{4}{|c|}{ \tabincell{c}{ $4r+6,4r+8,r\neq p~(V_2^r)$\\$5,7,r=p~(V_1^r)$}}
&\multicolumn{4}{|c|}{ $4r+1,4r+3~(V_1^r)$}\\
\hline
\end{tabular}
\end{table}

\noindent{\bf (c)}~ The construction for $\widehat{G}_{p+1}$.

With a similar argument to that in Case 1, we can list the edges that belong to $K_{1,4p+2,4p+2}$ but not to any $\widehat{G}_{r}$, $1\leq r\leq p$, in this case,   as shown in Table \ref{tab 4}.
Then $\widehat{G}_{p+1}$ is the graph that consists of the edges in Table \ref{tab 4}, Figure \ref{figure 3} shows $\widehat{G}_{p+1}$ is a planar graph.

Therefore, $\{\widehat{G}_1,\dots,\widehat{G}_{p+1}\}$ is a planar decomposition of $K_{1,4p+2,4p+2}$ in this case.

\begin{table}[htbp]
\centering\caption {The edges of $\widehat{G}_{p+1}$ in Case 2} \label{tab 4}
\scalebox{1}{
\tiny
\renewcommand{\arraystretch}{2}
\begin{tabular}{|c|c|}
\hline
\Gape[8pt]edges & subscript\\
\hline
$xv_{4p+1},v_{4p+1}u_j$ & {$j=4r-3,4r-2,4r-1,4r,7,8,4p+2.$ ($r=3,5,7,\dots,p.$)} \\
\hline
$xu_{4p+1},u_{4p+1}v_j$ & {$j=4r-3,4r-2,4r-1,4r,5,7,4p+2.$ ($r=3,5,7,\dots,p.$)} \\
\hline

$xv_{4p+2},v_{4p+2}u_j$ & {$j=4r-3,4r-2,4r-1,4r,5,6,4p+1.$ ($r=1,4,6,8\dots,p-1.$)} \\
\hline
$xu_{4p+2},u_{4p+2}v_j$ &{$j=4r-3,4r-2,4r-1,4r,6,8,4p+1.$ ($r=1,4,6,8\dots,p-1.$)}
\\
\hline
$u_1v_2,v_{4r-3}u_{4r},u_{4r}v_{4r-1}$ &  $r=1,3,\dots,p.$ \\
\hline
$v_{4r}u_{4r-3},u_{4r-3}v_{4r-2}$ &  $r=2,4,\dots,p-1.$  \\
\hline
$u_{j}v_{j}$ &  $j=1,\dots,4p+2.$  \\
\hline
\end{tabular}}
\end{table}

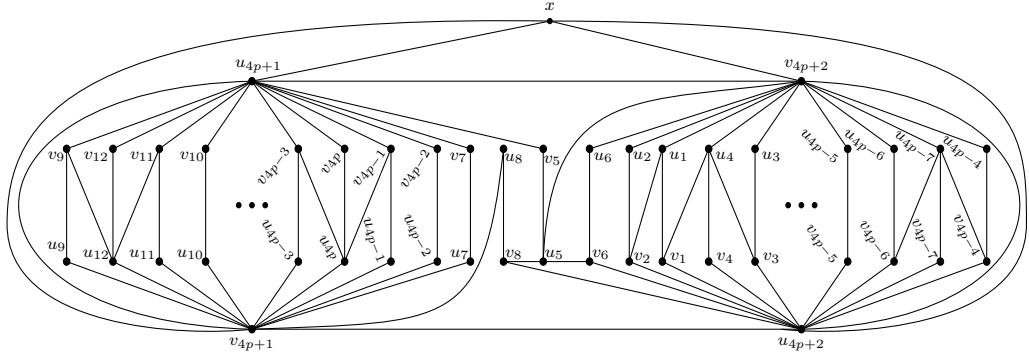
\begin{figure}[htp]
\begin{center}
\begin{tikzpicture}
[xscale=0.88]
\tikzstyle{every node}=[font=\tiny,scale=0.8]
[inner sep=0pt]\tiny
\draw (1,2.9)  node { $x$};\filldraw [black] (1,2.7) circle (1.0pt);
\draw(-3.5,1.9)--(1,2.7);
\draw[-](1,2.7)--(4.8,1.9);
\draw[-] (1,2.7) ..controls+(-7.2,0.1)and+(0,1.4)..(-7.2,0);
\draw[-] (-7.2,0) ..controls+(0,-1.4)and+(-1,-0.1).. (-3.5,-1.4);
\draw[-] (1,2.7) ..controls+(7.2,0)and+(0,1.4)..(8.2,0);
\draw[-] (8.2,0) ..controls+(0,-1.4)and+(1,-0.1).. (4.5,-1.4);
\filldraw [black] (-0.2,1) circle (1.4pt)         (-0.2,-0.5) circle (1.4pt)
                  (0.3,1) circle (1.4pt)                  (0.3,-0.5) circle (1.4pt)
                   (0.9,1) circle (1.4pt)                     (0.9,-0.5) circle (1.4pt)
                  (1.6,1) circle (1.4pt)                  (1.6,-0.5) circle (1.4pt);
 \draw (-0.2,1)--(-0.2,-0.5);\draw  (0.3,1)--(0.3,-0.5);\draw (0.9,1)--(0.9,-0.5);\draw (1.6,1)--(1.6,-0.5);
 \draw  (0.3,-0.5)--(1.6,-0.5);
 \draw(-0.38,0.89) node { $v_{7}$}; \draw(-0.36,-0.4) node { $u_{7}$};
 \draw(0.46,0.87) node { $u_{8}$};\draw(0.46,-0.4)  node { $v_{8}$};
  \draw(1.04,0.85) node { $v_{5}$}; \draw(1.07,-0.4) node { $u_{5}$};
   \draw(1.81,0.89) node { $u_{6}$}; \draw(1.77,-0.4) node { $v_{6}$};
\filldraw [black] (2.2,1) circle (1.4pt)
                  (2.2,-0.5) circle (1.4pt)
                  (2.7,1) circle (1.4pt)
                  (2.7,-0.5) circle (1.4pt)
                  (3.4,1) circle (1.4pt)
                  (3.4,-0.5) circle (1.4pt)
                   (4.1,1) circle (1.4pt)
                  (4.1,-0.5) circle (1.4pt)
                  (5.5,1) circle (1.4pt)
                  (5.5,-0.5) circle (1.4pt)
                  (6.2,1) circle (1.4pt)
                  (6.2,-0.5) circle (1.4pt)
                  (6.9,1) circle (1.4pt)
                  (6.9,-0.5) circle (1.4pt)
                  (7.6,1) circle (1.4pt)
                  (7.6,-0.5) circle (1.4pt);
\draw   (2.2,-0.5)--(2.2,1);\draw   (2.2,-0.5)--(2.7,1);
\draw (6.2,-0.5)--(6.9,1)-- (7.6,-0.5) ;\draw   (2.7,-0.5)--(3.4,1)-- (4.1,-0.5);
\draw (2.4,0.9)  node {$u_{2}$};\draw  (2.4,-0.45)node { $v_{2}$};
\draw (2.95,0.9)  node { $u_{1}$};\draw  (2.95,-0.45)node { $v_{1}$};
\draw (3.65,0.9)  node { $u_{4}$};\draw  (3.65,-0.45)node { $v_{4}$};
\draw (4.35,0.9)  node {$u_{3}$};\draw  (4.35,-0.45)node { $v_{3}$};
\draw [white](4.95,0.89) node [left]{}--(5,0.85)node [black,pos=0.02,above,sloped]{ $u_{4p-5}$}--(5,0.79) node [right]{};
\draw [white](5.65,0.89) node [left]{}--(5.7,0.86)node [black,pos=0.02,above,sloped]{ $u_{4p-6}$}--(5.7,0.79) node [right]{};
\draw [white](6.4,0.87) node [left]{}--(6.43,0.85)node [black,pos=0.02,above,sloped]{ $u_{4p-7}$}--(6.43,0.81) node [right]{};
\draw [white](7.12,0.80) node [left]{}--(7.15,0.78)node [black,pos=0.02,above,sloped]{ $ u_{4p-4}$}--(7.15,0.70) node [right]{};
\draw [white](4.95,-0.38) node [left]{}--(5,-0.44)node [black,pos=0.02,above,sloped]{ $v_{4p-5}$}--(5,-0.5) node [right]{};
\draw [white](5.74,-0.33) node [left]{}--(5.84,-0.46)node [black,pos=0.02,above,sloped]{$ v_{4p-6}$ }--(5.84,-0.46) node [right]{};
\draw [white](6.48,-0.3) node [left]{}--(6.497,-0.33)node [black,pos=0.02,above,sloped]{ $ v_{4p-7}$}--(6.597,-0.36) node [right]{};
\draw [white](7.12,-0.25) node [left]{}--(7.15,-0.30)node [black,pos=0.02,above,sloped]{ $ v_{4p-4}$}--(7.15,-0.37) node [right]{};
\filldraw [black] (-0.7,1) circle (1.4pt)    (-0.7,-0.5) circle (1.4pt)      (-1.4,1) circle (1.4pt)
                  (-1.4,-0.5) circle (1.4pt)            (-2.1,1) circle (1.4pt)            (-2.1,-0.5) circle (1.4pt)
                  (-2.8,1) circle (1.4pt)               (-2.8,-0.5) circle (1.4pt)        (-4.2,1) circle (1.4pt)
                  (-4.2,-0.5) circle (1.4pt)          (-4.9,1) circle (1.4pt)             (-4.9,-0.5) circle (1.4pt)
                  (-5.6,1) circle (1.4pt)               (-5.6,-0.5) circle (1.4pt)         (-6.3,1) circle (1.4pt)
                  (-6.3,-0.5) circle (1.4pt);
\draw (-6.3,1)--(-5.6,-0.5)-- (-4.9,1) ;\draw  (-2.8,1)-- (-2.1,-0.5)--(-1.4,1);
\draw (-6.45,0.9)  node { $v_{9}$};\draw  (-6.45,-0.3)node { $u_{9}$};
\draw (-5.85,0.9)  node { $v_{12}$};\draw  (-5.85,-0.4)node { $u_{12}$};
\draw (-5.15,0.9)  node { $v_{11}$};\draw  (-5.15,-0.4)node { $u_{11}$};
\draw (-4.4,0.9)  node { $v_{10}$};\draw  (-4.43,-0.40)node { $u_{10}$};
\draw [white](-0.85,0.65) node [left]{}--(-0.81,0.72)node [black,pos=0.02,above,sloped]{ $v_{4p-2}$}--(-0.81,0.72) node [right]{};
\draw [white](-1.55,0.65) node [left]{}--(-1.5,0.72)node [black,pos=0.02,above,sloped]{ $v_{4p-1}$}--(-1.5,0.72) node [right]{};
\draw [white](-2.12,0.72) node [left]{}--(-2.08,0.79)node [black,pos=0.02,above,sloped]{ $v_{4p}$}--(-2.08,0.79) node [right]{};
\draw [white](-3,0.67) node [left]{}--(-2.95,0.74)node [black,pos=0.02,above,sloped]{ $v_{4p-3}$}--(-2.95,0.74) node [right]{};
\draw [white](-3.27,-0.33) node [left]{}--(-3.22,-0.39)node [black,pos=0.02,above,sloped]{ $ u_{4p-3}$}--(-3.22,-0.39) node [right]{};
\draw [white](-2.5,-0.43) node [left]{}--(-2.45,-0.49)node [black,pos=0.02,above,sloped]{ $ u_{4p}$}--(-2.45,-0.55) node [right]{};
\draw [white](-1.85,-0.28) node [left]{}--(-1.83,-0.33)node [black,pos=0.02,above,sloped]{ $ u_{4p-1}$}--(-1.83,-0.36) node [right]{};
\draw [white](-1.18,-0.25) node [left]{}--(-1.15,-0.30)node [black,pos=0.02,above,sloped]{ $ u_{4p-2}$}--(-1.15,-0.37) node [right]{};
\filldraw [black] (-3.5,0.25) circle (1.0pt)
                  (-3.3,0.25) circle (1.0pt)
                  (-3.7,0.25) circle (1.0pt);
\filldraw [black] (4.8,0.25) circle (1.0pt)
                  (4.6,0.25) circle (1.0pt)
                  (5,0.25) circle (1.0pt);
\filldraw [black] (-3.5,1.9) circle (1.4pt)
                  (-3.5,-1.4) circle (1.4pt)
                  (4.8,1.9) circle (1.4pt)
                  (4.8,-1.4) circle (1.4pt);
 \draw[-] (-3.5,1.9)..controls+(-4.7,-0.1)and+(-4.7,0.1)..  (-3.5,-1.4) ;
\draw[-] (4.8,1.9)..controls+(4.4,-0.1)and+(4.4,0.1)..  (4.8,-1.4) ;
\draw  (-3.5,1.9)--(4.8,1.9); \draw   (-3.5,-1.4)--(4.8,-1.4);
\draw[-] (-3.5,-1.4)..controls+(3,0.1)and+(-0.3,-2.4)..  (0.3,1) ;
\draw  (-3.4,2.1) node { $u_{4p+1}$};\draw (-3.5,-1.6)node { $v_{4p+1}$};
\draw(4.88,2.1) node { $v_{4p+2}$};\draw(4.8,-1.6)  node { $u_{4p+2}$};
\draw[-] (4.8,1.9)..controls+(-2.4,-0.3)and+(0.1,2.4)..  (0.9,-0.5) ;
\draw (-0.7,1)--(-3.5,1.9);\draw (-1.4,1)--(-3.5,1.9);\draw (-2.1,1)--(-3.5,1.9);
\draw (-2.8,1)--(-3.5,1.9);\draw (-4.2,1)--(-3.5,1.9);\draw (-4.9,1)--(-3.5,1.9);
\draw (-5.6,1)--(-3.5,1.9);\draw (-0.2,1)--(-3.5,1.9);\draw (-6.3,1)--(-3.5,1.9);
\draw(0.9,1)--(-3.5,1.9);
\draw (-3.5,-1.4)--(-0.7,-0.5);\draw (-3.5,-1.4)--(-1.4,-0.5);\draw (-3.5,-1.4)--(-2.1,-0.5);
\draw (-3.5,-1.4)--(-2.8,-0.5);\draw (-3.5,-1.4)--(-4.2,-0.5);\draw (-3.5,-1.4)--(-4.9,-0.5);
\draw (-3.5,-1.4)--(-5.6,-0.5);\draw (-3.5,-1.4)--(-6.3,-0.5);\draw (-3.5,-1.4)--(-0.2,-0.5);
\draw (2.7,1)--(4.8,1.9);\draw (3.4,1)--(4.8,1.9);\draw (4.1,1)--(4.8,1.9);
\draw (5.5,1)--(4.8,1.9);\draw (6.2,1)--(4.8,1.9);
\draw (6.9,1)--(4.8,1.9);\draw (7.6,1)--(4.8,1.9);
\draw (2.2,1)--(4.8,1.9);\draw (1.6,1)--(4.8,1.9);
\draw (4.8,-1.4)--(2.7,-0.5);\draw (4.8,-1.4)--(3.4,-0.5);\draw (4.8,-1.4)--(4.1,-0.5);
\draw (4.8,-1.4)--(5.5,-0.5);\draw (4.8,-1.4)--(6.2,-0.5);
\draw (4.8,-1.4)--(6.9,-0.5);\draw (4.8,-1.4)--(7.6,-0.5);
\draw (4.8,-1.4)-- (2.2,-0.5);\draw (4.8,-1.4)--(1.6,-0.5);
\draw (4.8,-1.4)--(0.3,-0.5);
\draw (2.7,1)--(2.7,-0.5);\draw (3.4,1)--(3.4,-0.5);\draw (4.1,1)--(4.1,-0.5);
\draw (5.5,1)--(5.5,-0.5);\draw (6.2,1)--(6.2,-0.5);
\draw (6.9,1)--(6.9,-0.5);\draw (7.6,1)--(7.6,-0.5);
\draw (-0.7,1)--(-0.7,-0.5);\draw (-1.4,1)--(-1.4,-0.5);\draw (-2.1,1)--(-2.1,-0.5);
\draw (-2.8,1)--(-2.8,-0.5);\draw (-4.2,1)--(-4.2,-0.5);\draw (-4.9,1)--(-4.9,-0.5);
\draw (-5.6,1)--(-5.6,-0.5);\draw (-6.3,1)--(-6.3,-0.5);
\end{tikzpicture}
\caption{The graph $\widehat{G}_{p+1}$ in Case 2}
\label{figure 3}
\end{center}
\end{figure}

{\bf Case 3.}~ When $p\leq 3$.

When $p=0$, $K_{1,2,2}$ is a planar graph. When $p=1,2,3$, we give a planar decomposition for $K_{1,6,6}$, $K_{1,10,10}$ and $K_{1,14,14}$ with $2$, $3$ and $4$ subgraphs respectively, as shown in Figure \ref{figure 4}, Figure \ref{figure 5} and Figure \ref{figure 6}.

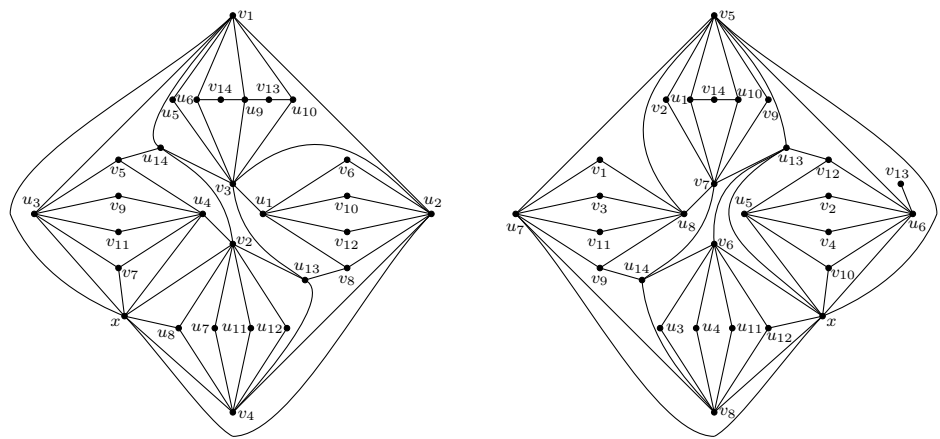
\begin{figure}[H]
\begin{center}
\begin{tikzpicture}
\tiny
[inner sep=0pt]
\filldraw [black] (0.5,0) circle (1.4pt)
                  (1.2,0) circle (1.4pt)
                  (1.9,0) circle (1.4pt)
                   (-0.5,0) circle (1.4pt)
                  (-1.2,0) circle (1.4pt)
                  (-1.9,0) circle (1.4pt)
                   (0,0.5) circle (1.4pt)
                  (0,1.2) circle (1.4pt)
                   (0,1.9) circle (1.4pt)
                  (0,-0.5) circle (1.4pt)
                  (0,-1.2) circle (1.4pt)
                   (0,-1.9) circle (1.4pt);
\draw  (0,-0.5)-- (0,-1.9) ;\draw(0,0.5)-- (0,1.9) ;
\draw  (-0.5,0)--  (-1.9,0) ;\draw   (0.5,0)--  (1.9,0);
\draw  (0,1.9) --  (-1.9,0)--(0,-1.9)-- (1.9,0)-- (0,1.9);
\draw  (0,0.5) --  (-0.5,0)--(0,-0.5)-- (0.5,0)-- (0,0.5);
\draw[-] (0,1.9)..controls+(-0.4,-0.5)and+(-0.1,0.5).. (-0.5,0);
\draw[-] (0,-1.9)..controls+(0.48,0.5)and+(0.1,-0.5).. (0.5,0);
\draw[-]  (-1.9,0)..controls+(0.4,-0.5)and+(-0.5,-0.1)..  (0,-0.5);
\draw[-]  (1.9,0)..controls+(-0.4,0.5)and+(0.5,0.1)..  (0,0.5);
\draw  (0.2,0.65) node { $v_3$};
\draw   (0.22,1.2)node { $u_5$};
\draw   (0.22,1.9)node { $v_1$};
\draw  (0.2,-0.65) node { $v_2$};
\draw   (0.22,-1.2)node { $u_6$};
\draw   (0.22,-2)node { $v_4$};
\draw   (-0.7,0.15) node { $u_4$};
\draw  (-1.2,0.15) node { $v_5$};
\draw   (-2,0.15)node { $u_3$};
\draw   (0.6,0.15) node { $u_1$};
\draw  (1.2,0.15) node { $v_6$};
\draw   (2,0.15)node { $u_2$};
\filldraw [black] (0,0) circle (1.4pt);
\draw  (0,0.2) node { $x$};
\begin{scope}[xshift=6cm]
\filldraw [black] (0,0) circle (1.4pt);
\draw  (0,0.3) node { $x$};
\draw  (-1.5,1.2)  parabola bend (0,1.7)  (1.5,1.2);
\draw  (-1.5,-1.2)  parabola bend (0,-1.7)  (1.5,-1.2);
\draw[-](-1.5,1.2)..controls+(-0.6,-0.8)and+(-0.6,0.8)..  (-1.5,-1.2);
\draw[-](1.5,-1.2)..controls+(0.6,0.8)and+(0.6,-0.8)..  (1.5,1.2);
\draw  (-1.5,1.4)  node { $u_{6}$};
\draw  (-0.4,1.4)  node { $v_{1}$};
\draw  (0.4,1.4)  node { $u_{1}$};
\draw  (1.5,1.4)  node { $v_{5}$};
\draw  (-1.5,-1.4)  node { $v_{6}$};
\draw  (-0.4,-1.4)  node { $u_{4}$};
\draw  (0.4,-1.4)  node { $v_{4}$};
\draw  (1.5,-1.4)  node { $u_{5}$};
\draw  (1.7,0.4)  node { $u_{2}$};
\draw  (1.7,-0.4)  node { $v_{2}$};
\draw  (-1.7,0.4)  node { $v_{3}$};
\draw  (-1.7,-0.4)  node { $u_{3}$};
\filldraw [black] (1.5,0.4) circle (1.4pt)
                  (1.5,1.2) circle (1.4pt)
                  (1.5,-0.4) circle (1.4pt)
                   (1.5,-1.2) circle (1.4pt)
                    (-1.5,0.4) circle (1.4pt)
                  (-1.5,1.2) circle (1.4pt)
                  (-1.5,-0.4) circle (1.4pt)
                   (-1.5,-1.2) circle (1.4pt)
                  (0.46,1.2) circle (1.4pt)
                  (-0.46,1.2) circle (1.4pt)
                  (0.46,-1.2) circle (1.4pt)
                  (-0.46,-1.2) circle (1.4pt);
\draw (0,0)--(-0.46,-1.2);\draw (0,0)--(0.46,-1.2);\draw (0,0)--(1.5,-1.2);
\draw (0,0)--(1.5,-0.4);\draw (0,0)--(1.5,1.2);\draw (0,0)--(0.46,1.2);
\draw (0,0)--(-0.46,1.2);\draw (0,0)--(-1.5,1.2);\draw (0,0)-- (1.5,0.4);
\draw (0,0)--(-1.5,-0.4);\draw (0,0)--(-1.5,-1.2);\draw (0,0)--(-1.5,0.4);
\draw (-0.46,-1.2)--(0.46,-1.2) --  (1.5,-1.2)  -- (1.5,-0.4) -- (1.5,0.4) -- (1.5,1.2) -- (0.46,1.2)  -- (-0.46,1.2)  --  (-1.5,1.2) -- (-1.5,-0.4)-- (-1.5,-1.2) -- (-0.46,-1.2) -- (-0.46,-1.2);
\end{scope}
\end{tikzpicture}
\caption{A planar decomposition of $K_{1,6,6}$}
\label{figure 4}
\end{center}
\end{figure}
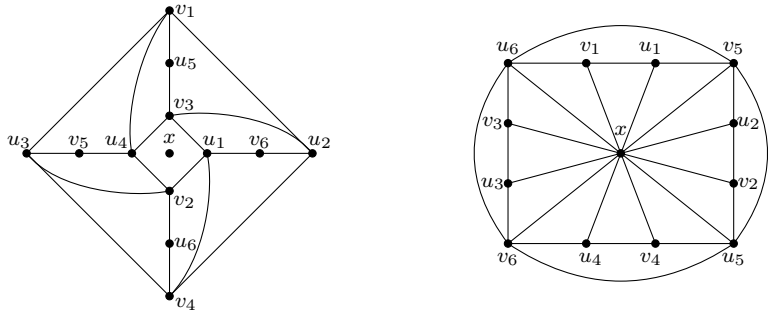

\begin{figure}[H]
\begin{center}
\begin{tikzpicture}
[scale=1.3]
\tikzstyle{every node}=[font=\tiny,scale=1]
\tiny
[inner sep=0pt]
\filldraw [black] (0.5,0) circle (1pt)
                  (1.9,0) circle (1pt)
                  (1.2,-0.16) circle (1pt)
                  (1.2,0.16) circle (1pt)
                   (-0.5,0) circle (1pt)
                  (-1.2,0.16) circle (1pt)
                  (-1.9,0) circle (1pt)
                  (-1.2,-0.16) circle (1pt)
                   (0,0.5) circle (1pt)
                  (0.36,1.2) circle (1pt)
                   (0,1.9) circle (1pt)
                    (-0, 1.2) circle (1pt)
                  (0,-0.5) circle (1pt)
                  (0.16,-1.2) circle (1pt)
                   (0,-1.9) circle (1pt)
                    (-0.16,-1.2) circle (1pt);
\filldraw [black] (-0.5,0.5) circle (1pt);\draw  (-0.7,0.55) node { $u_{10}$};
\draw   (-1.2,0.16)--(-0.5,0.5)-- (0,0.5) ;
\draw[-]  (0,1.9)..controls+(-0.4,-0.5)and+(-0.2,0.5).. (-0.5,0.5);
\draw[-] (0,-0.5)..controls+(-0.1,0.5)and+(0.2,-0.2).. (-0.5,0.5);
\filldraw [black] (0.5,-0.5) circle (1pt);\draw  (0.65,-0.58) node { $u_{9}$};
\draw   (1.2,-0.16)--(0.5,-0.5)-- (0,-0.5) ;
\draw[-]  (0,-1.9)..controls+(0.4,0.5)and+(0.1,-0.5).. (0.5,-0.5);
\draw[-] (0,0.5)..controls+(0.1,-0.5)and+(-0.2,0.2).. (0.5,-0.5);
\filldraw [black] (0.7,0.35) circle (1pt);\draw  (0.55,0.35) node { $v_{9}$};
\draw    (0.7,0.35)--(0.5,0);
\draw[-]  (0.7,0.35)..controls+(0.1,0.1)and+(-0.2,0.2).. (1.9,0);
\filldraw [black] (-0.2,0.9) circle (1pt);\draw  (-0.299,1.03) node { $v_{10}$};
\draw  (-0.5,0.5)-- (-0.2,0.9)-- (0,1.2) ;
\draw  (0,-0.5)--(-0.16,-1.2)-- (0,-1.9) ;\draw  (0,-0.5)--(0.16,-1.2)-- (0,-1.9) ;
\draw(0,0.5)--(0,1.2)-- (0,1.9) ;\draw(0,0.5)--(0.36,1.2)-- (0,1.9) ;
\draw  (-0.5,0)-- (-1.2,0.16)--  (-1.9,0) ;\draw  (-0.5,0)-- (-1.2,-0.16)--  (-1.9,0) ;
\draw   (0.5,0)-- (1.2,0.16)--  (1.9,0);\draw   (0.5,0)-- (1.2,-0.16)--  (1.9,0);
\draw  (0,1.9) --  (-1.9,0);\draw (0,-1.9)-- (1.9,0)-- (0,1.9);
\draw  (0,-0.5) --  (-0.5,0);\draw(0,0.5)-- (0.5,0);
\draw[-]  (1.9,0)..controls+(-0.4,0.38)and+(0.5,0.1)..  (0,0.5);
\draw  (0.22,0.62) node { $v_3$};
\draw   (0.55,1.18)node { $u_5$};
\draw   (0.16,1.2)node { $u_6$};
\draw   (0.17,1.9)node { $v_1$};
\draw  (0.11,-0.4) node { $v_2$};
\draw   (0.28,-1.11)node { $u_7$};
\draw   (-0.281,-1.07)node { $u_8$};
\draw   (0.17,-1.95)node { $v_4$};
\draw   (-0.4,0.10) node { $u_4$};
\draw  (-1.35,0.25) node { $v_5$};
\draw  (-1.35,-0.27) node { $v_7$};
\draw   (-2,0.12)node { $u_3$};
\draw   (0.6,-0.12) node { $u_1$};
\draw  (1.2,0.03) node { $v_6$};
\draw  (1.2,-0.3) node { $v_8$};
\draw   (2,0.09)node { $u_2$};
\filldraw [black] (-1,-1) circle (1pt);
\draw  (-1.1,-1.1) node { $x$};
\draw   (-1,-1)  --  (-0.5,0);\draw   (-1,-1)  --  (0,-0.5);
\draw   (-1,-1)  --  (-1.2,-0.16);\draw   (-1,-1)  --  (-0.16,-1.2);
\draw   (-1,-1)  --  (-1.9,0);\draw   (-1,-1)  --  (0,-1.9);
\draw[-]   (-1,-1)..controls+(-0.5,0.1)and+(0.1,-0.3)..  (-2.2,0);
\draw[-]   (0,1.9)..controls+(-0.5,-0.1)and+(-0.01, 0.8)..  (-2.2,0);
\draw[-]   (-1,-1)..controls+(0.1,-0.1)and+(-0.5,0.1)..  (0,-2.2);
\draw[-]   (1.9,0)..controls+(-0.1,-0.1)and+(0.8,0.01)..  (0,-2.2);
\begin{scope}[xshift=5cm]
\filldraw [black] (0.5,0) circle (1pt)
                  (1.9,0) circle (1pt)
                  (1.2,-0.16) circle (1pt)
                  (1.2,0.16) circle (1pt)
                   (-0.5,0) circle (1pt)
                  (-1.2,0.4) circle (1pt)
                  (-1.9,0) circle (1pt)
                  (-1.2,-0.35) circle (1pt)
                   (0,0.5) circle (1pt)
                  (0.16,1.2) circle (1pt)
                   (0,1.9) circle (1pt)
                    (-0.16, 1.2) circle (1pt)
                  (0,-0.5) circle (1pt)
                  (0.41,-1.2) circle (1pt)
                   (0,-1.9) circle (1pt)
                    (-0.46,-1.2) circle (1pt);
\filldraw [black] (-0.5,0.5) circle (1pt);\draw  (-0.65,0.58) node { $u_{9}$};
\draw   (-1.2,0.4)--(-0.5,0.5)-- (0,0.5) ;
\draw[-]  (0,1.9)..controls+(-0.4,-0.5)and+(-0.1,0.5).. (-0.5,0.5);
\draw[-] (0,-0.5)..controls+(-0.1,0.5)and+(0.2,-0.2).. (-0.5,0.5);
\filldraw [black] (0.5,-0.5) circle (1pt);\draw  (0.74,-0.56) node { $u_{10}$};
\draw   (1.2,-0.16)--(0.5,-0.5)-- (0,-0.5) ;
\draw[-]  (0,-1.9)..controls+(0.4,0.5)and+(0.4,-0.5).. (0.5,-0.5);
\draw[-] (0,0.5)..controls+(0.1,-0.5)and+(-0.2,0.2).. (0.5,-0.5);
 \filldraw [black]  (0,-1.2) circle (1pt);\draw  (0,-1.07) node { $v_{9}$};
\draw (-0.46,-1.2)--(0,-1.2)-- (0.41,-1.2) ;
\filldraw [black] (-1.2,0) circle (1pt);\draw  (-1.2,0.14) node { $v_{10}$};
\draw  (-0.5 ,0)--(-1.9, 0) ;
\draw  (0,-0.5)--(-0.46,-1.2)-- (0,-1.9) ;\draw  (0,-0.5)--(0.41,-1.2)-- (0,-1.9) ;
\draw(0,0.5)--(-0.16,1.2)-- (0,1.9) ;\draw(0,0.5)--(0.16,1.2)-- (0,1.9) ;
\draw  (-0.5,0)-- (-1.2,0.4)--  (-1.9,0) ;\draw  (-0.5,0)-- (-1.2,-0.35)--  (-1.9,0) ;
\draw   (0.5,0)-- (1.2,0.16)--  (1.9,0);\draw   (0.5,0)-- (1.2,-0.16)--  (1.9,0);
\draw  (0,1.9) --  (-1.9,0);\draw (-1.9,0)-- (0,-1.9)-- (1.9,0);
\draw  (0,-0.5) --  (-0.5,0);\draw(0,0.5)-- (0.5,0);
\draw[-]  (-1.9,0)..controls+(1.2,-1)and+(-0.5,-0.2)..  (0,-0.5);
\draw  (0.17,0.5) node { $v_7$};
\draw   (0.3,1.07)node { $u_1$};
\draw   (-0.31,1.14)node { $u_2$};
\draw   (0.17,1.9)node { $v_5$};
\draw  (0.12,-0.4) node { $v_6$};
\draw   (0.47,-1.07)node { $u_4$};
\draw   (-0.53,-1.07)node { $u_3$};
\draw   (0.13,-2)node { $v_8$};
\draw   (-0.4,0.09) node { $u_8$};
\draw  (-1.15,0.53) node { $v_1$};
\draw  (-1.2,-0.21) node { $v_3$};
\draw   (-2,0.12)node { $u_7$};
\draw   (0.57,-0.15) node { $u_5$};
\draw  (1.2,0.29) node { $v_2$};
\draw  (1.2,-0.29) node { $v_4$};
\draw   (2,0.1)node { $u_6$};
\filldraw [black] (1,1) circle (1pt);
\draw  (1.03,0.8) node { $x$};
\draw   (1,1)  --  (0.5,0);\draw   (1,1)  --  (0.16,1.2);
\draw   (1,1)  --  (1.9,0);\draw   (1,1)  --  (0,1.9);
\draw[-]   (1,1)..controls+(0.5,-0.1)and+(-0.1,0.4)..  (2.3,0);
\draw[-]   (0,-1.9)..controls+(0.5,0.1)and+(0.01, -0.8)..  (2.3,0);
\draw[-]   (1,1)..controls+(-0.1,0.1)and+(0.5,-0.1)..  (0,2.2);
\draw[-]   (-1.9,0)..controls+(0.1,0.1)and+(-0.8,-0.01)..  (0,2.2);
\end{scope}
\end{tikzpicture}
\end{center}
\end{figure}

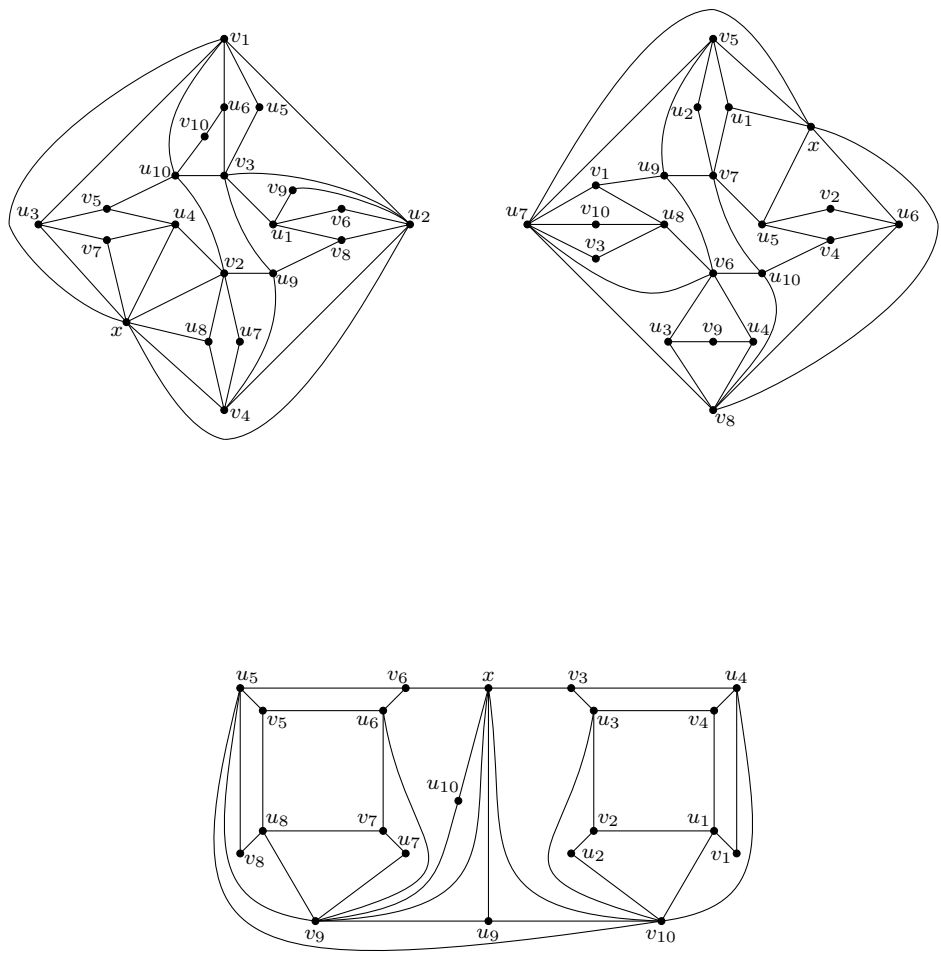
\begin{figure}[thp]
\begin{center}
\begin{tikzpicture}
[inner sep=0pt]
\filldraw [black]     (-3.3,-2) circle (1.3pt);
\draw (-3.3,-2.2) node {\tiny $v_{9}$};
\filldraw [black]     (-1.4,-0.4) circle (1.3pt);
\draw (-1.6,-0.2) node {\tiny $u_{10}$};
\draw[-](-3.3,-2)..controls+(1.7,0.2)and+(-0.3,-1)..(-1.4,-0.4);
\draw[-](-3.3,-2)--(-2.1,-1.1);\draw[-](-3.3,-2)--(-4,-0.8);
\draw[-](-3.3,-2)..controls+(2.7,0.1)and+(-0.3,-2)..(-1,1.1);
\draw[-](-3.3,-2)..controls+(2.5,0.7)and+(0.3,-2)..(-2.4,0.8);
\draw[-](-3.3,-2)..controls+(-1.5,0.2)and+(-0.3,-2)..(-4.3,1.1);
\filldraw [black]     (1.3,-2) circle (1.3pt);
\draw (1.3,-2.2) node {\tiny $v_{10}$};
\draw[-](1.3,-2)--(0.1,-1.1);\draw[-](1.3,-2)--(2,-0.8);
\draw[-](1.3,-2)..controls+(-2.7,0.1)and+(0.3,-2)..(-1,1.1);
\draw[-](1.3,-2)..controls+(-2.5,0.7)and+(-0.3,-2)..(0.4,0.8);
\draw[-](1.3,-2)..controls+(1.5,0.2)and+(0.3,-2)..(2.3,1.1);
\draw[-](1.3,-2)..controls+(-6,-0.9)and+(-0.9,-3.5)..(-4.3,1.1);
\filldraw [black]     (-1,1.1) circle (1.3pt);
\draw (-1,1.25) node {\tiny $x$};
\draw[-](0.1,1.1)--(-2.1,1.1);\draw[-] (-1,1.1)--(-1,-2);\draw[-] (-1,1.1)--(-1.4,-0.4);
\filldraw [black]     (-1,-2) circle (1.3pt);
\draw (-1,-2.18) node {\tiny $u_9$};
\draw[-](-3.3,-2)--(1.3,-2);
\filldraw [black]     (-4,0.8) circle (1.3pt);
\draw (-3.8,0.65) node {\tiny $v_{5}$};
\filldraw [black]     (-4,-0.8) circle (1.3pt);
\draw (-3.8,-0.66) node {\tiny $u_{8}$};
\filldraw [black]     (-2.4,0.8) circle (1.3pt);
\draw (-2.6,0.645) node {\tiny $u_{6}$};
\filldraw [black]     (-2.4,-0.8) circle (1.3pt);
\draw (-2.6,-0.66) node {\tiny $v_{7}$};
\draw[-](-4,0.8)--(-4,-0.8)--(-2.4,-0.8)--(-2.4,0.8)--(-4,0.8);
\filldraw [black]     (-4.3,1.1) circle (1.3pt);\draw[-](-4.3,1.1)--(-4,0.8);
\draw (-4.2,1.25) node {\tiny $u_{5}$};\draw[-](-2.1,1.1)--(-4.3,1.1)--(-4.3,-1.1);
\filldraw [black]     (-2.1,1.1) circle (1.3pt);\draw[-](-2.1,1.1)--(-2.4,0.8);
\draw (-2.2,1.25) node {\tiny $v_{6}$};
\filldraw [black]     (-4.3,-1.1) circle (1.3pt);\draw[-](-4,-0.8)--(-4.3,-1.1);
\draw (-4.1,-1.2) node {\tiny $v_{8}$};
\filldraw [black]     (-2.1,-1.1) circle (1.3pt);\draw[-](-2.1,-1.1) --(-2.4,-0.8);
\draw (-2.05,-0.95) node {\tiny $u_{7}$};
\filldraw [black]     (2,0.8) circle (1.3pt);
\draw (1.8,0.65) node {\tiny $v_{4}$};
\filldraw [black]     (2,-0.8) circle (1.3pt);
\draw (1.8,-0.66) node {\tiny $u_{1}$};
\filldraw [black]     (0.4,0.8) circle (1.3pt);
\draw (0.6,0.645) node {\tiny $u_{3}$};
\filldraw [black]     (0.4,-0.8) circle (1.3pt);
\draw (0.6,-0.66) node {\tiny $v_{2}$};
\draw[-](2,0.8)--(2,-0.8)--(0.4,-0.8)--(0.4,0.8)--(2,0.8);
\filldraw [black]     (2.3,1.1) circle (1.3pt);\draw[-](2.3,1.1)--(2,0.8);
\draw (2.3,1.25) node {\tiny $u_{4}$};\draw[-](0.1,1.1)--(2.3,1.1)--(2.3,-1.1);
\filldraw [black]     (0.1,1.1) circle (1.3pt);\draw[-](0.1,1.1)--(0.4,0.8);
\draw (0.2,1.25) node {\tiny $v_{3}$};
\filldraw [black]     (2.3,-1.1) circle (1.3pt);\draw[-](2,-0.8)--(2.3,-1.1);
\draw (2.1,-1.15) node {\tiny $v_{1}$};
\filldraw [black]     (0.1,-1.1) circle (1.3pt);\draw[-](0.1,-1.1) --(0.4,-0.8);
\draw (0.4,-1.12) node {\tiny $u_{2}$};
\end{tikzpicture}
\caption{A planar decomposition of $K_{1,10,10}$}
\label{figure 5}
\end{center}
\end{figure}

\begin{figure}[H]
\begin{center}
\begin{tikzpicture}
[scale=0.8]
\tikzstyle{every node}=[font=\tiny,scale=0.8]
\tiny
[inner sep=0pt]
\filldraw [black] (0.5,0) circle (1.3pt);\draw  (0.53,0.2)  node { $u_{1}$};
\filldraw [black](3.3,0) circle (1.3pt);\draw  (3.33,0.2)  node { $u_{2}$};
 \filldraw [black] (1.9,-0.3) circle (1.3pt);\draw  (1.9,-0.5)  node { $v_{12}$};
\filldraw [black](1.9,0.3) circle (1.3pt);\draw  (1.9,0.1)  node { $v_{10}$};
\filldraw [black](1.9,-0.9) circle (1.3pt);\draw  (1.9,-1.1)  node { $v_{8}$};
\filldraw [black](1.9,0.9) circle (1.3pt);\draw  (1.9,0.7)  node { $v_{6}$};
\filldraw [black] (-0.5,0) circle (1.3pt);\draw  (-0.5,0.2)  node { $u_{4}$};
\filldraw [black](-1.9,0.3) circle (1.3pt);\draw  (-1.9,0.1)  node { $v_{9}$};
\filldraw [black](-3.3,0) circle (1.3pt);\draw  (-3.34,0.2)  node { $u_{3}$};
\filldraw [black] (-1.9,-0.3) circle (1.3pt);\draw  (-1.9,-0.5)  node { $v_{11}$};
\filldraw [black] (-1.9,0.9) circle (1.3pt);\draw  (-1.9,0.7)  node { $v_{5}$};
 \filldraw [black](-1.9,-0.9) circle (1.3pt);\draw  (-1.7,-1)  node { $v_{7}$};
\filldraw [black] (0,0.5) circle (1.3pt);\draw  (-0.15,0.4)  node { $v_{3}$};
\filldraw [black](0.2,1.9) circle (1.3pt);\draw  (0.37,1.70)  node { $u_{9}$};
\filldraw [black] (0,3.3) circle (1.3pt);\draw  (0.23, 3.3)  node { $v_{1}$};
 \filldraw [black](-0.2, 1.9) circle (1.3pt);\draw  (-0.2,2.09)  node { $v_{14}$};
\filldraw [black](-0.6, 1.9) circle (1.3pt);\draw  (-0.79,1.9)  node { $u_{6}$};
\filldraw [black](0.6, 1.9) circle (1.3pt);\draw  (0.6,2.09)  node { $v_{13}$};
\filldraw [black](-1, 1.9) circle (1.3pt);\draw  (-1.03,1.65)  node { $u_{5}$};
\filldraw [black](1, 1.9) circle (1.3pt);\draw  (1.2,1.70)  node { $u_{10}$};
\filldraw [black](0,-0.5) circle (1.3pt);\draw  (0.2,-0.46)  node { $v_{2}$};
 \filldraw [black] (0.3,-1.9) circle (1.3pt);\draw  (0.04,-1.9)  node { $u_{11}$};
\filldraw [black](0,-3.3) circle (1.3pt);\draw  (0.23,-3.35)  node { $v_{4}$};
\filldraw [black] (-0.3,-1.9) circle (1.3pt);\draw  (-0.53,-1.9)  node { $u_{7}$};
\filldraw [black](0.9,-1.9) circle (1.3pt);\draw  (0.62,-1.9)  node { $u_{12}$};
\filldraw [black] (-0.9,-1.9) circle (1.3pt);\draw  (-1.1,-2)  node { $u_{8}$};
\draw (-3.3,0)--(0,3.3)--(3.3,0)--(0,-3.3);
\draw (-0.5,0)--(0,-0.5);\draw (0.5,0)--(0,0.5);
\draw  (0,0.5)..controls+(1.7,1.6)and+(-0.5,0.3)..(3.3,0);
\draw (0, 3.3)--(-1,1.9)--(0,0.5);\draw (0, 3.3)--(-0.6,1.9)--(0,0.5);
\draw (-0.6,1.9)--(-0.2,1.9)--(0.2,1.9)--(0.6,1.9)--(1,1.9);
\draw (0, 3.3)--(0.2,1.9)--(0,0.5);\draw (0, 3.3)--(1,1.9)--(0,0.5);
\draw (0, -3.3)--(0.3,-1.9)--(0,-0.5);\draw (0, -3.3)--(0.9,-1.9)--(0,-0.5);
\draw (0, -3.3)--(-0.3,-1.9)--(0,-0.5);\draw (0, -3.3)--(-0.9,-1.9)--(0,-0.5);
\draw (-3.3 ,0)--(-1.9,0.3)--(-0.5,0);\draw (-3.3 ,0 )--(-1.9 ,0.9)--(-0.5, 0);
\draw (-3.3 ,0)--(-1.9,-0.3)--(-0.5,0);\draw (-3.3 ,0 )--(-1.9 ,-0.9)--(-0.5, 0);
\draw (3.3 ,0)--(1.9,0.3)--(0.5,0);\draw (3.3 ,0 )--(1.9 ,0.9)--(0.5, 0);
\draw (3.3 ,0)--(1.9,-0.3)--(0.5,0);\draw (3.3 ,0 )--(1.9 ,-0.9)--(0.5, 0);
\filldraw [black] (-1.2,1.1) circle (1.3pt);\draw  (-1.28,0.9)  node { $u_{14}$};
\draw  (-1.2,1.1)--(-1.9,0.9);\draw  (-1.2,1.1)--(0 ,0.5);
\draw  (-1.2,1.1)..controls+(0.4,-0.3)and+(-0.1,1)..(0,-0.5);
\draw  (-1.2,1.1)..controls+(-0.6,0.3)and+(-0.01,-0.01)..(0,3.3);
\filldraw [black] (1.2,-1.1) circle (1.3pt);\draw  (1.24,-0.9)  node { $u_{13}$};
\draw  (1.2,-1.1)--(1.9,-0.9);\draw  (1.2,-1.1)--(0 ,-0.5);
\draw  (1.2,-1.1)..controls+(-0.4,0.3)and+(0.1,-1)..(0,0.5);
\draw  (1.2,-1.1)..controls+(0.6,-0.3)and+(0.01,0.01)..(0,-3.3);
\filldraw [black] (-1.8,-1.7) circle (1.3pt);
\draw  (-1.95,-1.8)  node { $x$};
\draw  (-3.3,0)--(-1.8,-1.7)--(0,-3.3);
\draw  (-1.9,-0.9)--(-1.8,-1.7)--(-0.9,-1.9);
\draw  (-0.5,0)--(-1.8,-1.7)--(0,-0.5);
\draw   (-1.8,-1.7)..controls+(-0.5,0.2)and+(0.4,-1)..(-3.7,0);
\draw  (-3.7,0)..controls+(0.01,1)and+(-1,-1)..(0,3.3);
\draw   (-1.8,-1.7)..controls+(0.5,-0.5)and+(-0.4,0.2)..(0,-3.7);
\draw (0,-3.7)..controls+(1, 0.01)and+(-1,-1)..(3.3, 0);
\begin{scope}[xshift=8cm]
\filldraw [black] (0.5,0) circle (1.3pt);\draw  (0.5,0.2)  node { $u_{5}$};
 \filldraw [black] (3.3,0) circle (1.3pt);\draw  (3.38,-0.2)  node { $u_{6}$};
 \filldraw [black] (1.9,-0.3) circle (1.3pt);\draw  (1.9,-0.5)  node { $v_{4}$};
 \filldraw [black] (1.9,0.3) circle (1.3pt);\draw  (1.9,0.1)  node { $v_{2}$};
\filldraw [black](1.9,-0.9) circle (1.3pt);\draw  (2.15,-1)  node { $v_{10}$};
\filldraw [black] (1.9,0.9) circle (1.3pt);\draw  (1.9,0.68)  node { $v_{12}$};
\filldraw [black] (-0.5,0) circle (1.3pt);\draw  (-0.45,-0.2)  node { $u_{8}$};
\filldraw [black] (-1.9,0.3) circle (1.3pt);\draw  (-1.9,0.1)  node { $v_{3}$};
\filldraw [black](-3.3,0) circle (1.3pt);\draw  (-3.3,-0.25)  node { $u_{7}$};
\filldraw [black] (-1.9,-0.3) circle (1.3pt);\draw  (-1.9,-0.5)  node { $v_{11}$};
\filldraw [black] (-1.9,0.9) circle (1.3pt);\draw  (-1.9,0.7)  node { $v_{1}$};
\filldraw [black](-1.9,-0.9) circle (1.3pt);\draw  (-1.9,-1.1)  node { $v_{9}$};
\filldraw [black] (0,0.5) circle (1.3pt);\draw  (-0.2,0.5)  node { $v_{7}$};
\filldraw [black] (0.4,1.9) circle (1.3pt);\draw  (0.6,2)  node { $u_{10}$};
\filldraw [black]  (0,3.3) circle (1.3pt);\draw  (0.23,3.3)  node { $v_{5}$};
\filldraw [black] (-0.4, 1.9) circle (1.3pt);\draw  (-0.57,1.90)  node { $u_{1}$};
\filldraw [black] (-0.8, 1.9) circle (1.3pt);\draw  (-0.9,1.74)  node { $v_{2}$};
\filldraw [black] (0, 1.9) circle (1.3pt);\draw  (0,2.05)  node { $v_{14}$};
\filldraw [black] (0.9, 1.9) circle (1.3pt);\draw  (0.95,1.64)  node { $v_{9}$};
\filldraw [black](0,-0.5) circle (1.3pt);\draw  (0.2,-0.5)  node { $v_{6}$};
\filldraw [black](0.3,-1.9) circle (1.3pt);\draw  (0.6,-1.9)  node { $u_{11}$};
\filldraw [black] (0,-3.3) circle (1.3pt);\draw  (0.23,-3.31)  node { $v_{8}$};
\filldraw [black] (-0.3,-1.9) circle (1.3pt);\draw  (-0.04,-1.9)  node { $u_{4}$};
\filldraw [black](0.9,-1.9) circle (1.3pt);\draw  (1.09,-2.05)  node { $u_{12}$};
 \filldraw [black](-0.9,-1.9) circle (1.3pt);\draw  (-0.64,-1.9)  node { $u_{3}$};
\draw (3.3,0)--(0,3.3)--(-3.3,0)--(0,-3.3);
\draw (-0.5,0)--(0,0.5);
\draw  (-0.5, 0)..controls+(-1.7, 1.8)and+(-0.2,-0.5)..(0, 3.3);
\draw (0, 3.3)--(-0.8,1.9)--(0,0.5);\draw (0, 3.3)--(-0.4,1.9)--(0,0.5);
\draw (-0.4,1.9)--(0,1.9)--(0.4,1.9);
\draw (0, 3.3)--(0.4,1.9)--(0,0.5);\draw (0, 3.3)--(0.9,1.9)--(0,0.5);
\draw (0, -3.3)--(0.3,-1.9)--(0,-0.5);\draw (0, -3.3)--(0.9,-1.9)--(0,-0.5);
\draw (0, -3.3)--(-0.3,-1.9)--(0,-0.5);\draw (0, -3.3)--(-0.9,-1.9)--(0,-0.5);
\draw (-3.3 ,0)--(-1.9,0.3)--(-0.5,0);\draw (-3.3 ,0 )--(-1.9 ,0.9)--(-0.5, 0);
\draw (-3.3 ,0)--(-1.9,-0.3)--(-0.5,0);\draw (-3.3 ,0 )--(-1.9 ,-0.9)--(-0.5, 0);
\draw (3.3,0)--(3.1,0.5);\filldraw [black](3.1,0.5) circle (1.3pt);
\draw  (3.05,0.68)  node { $v_{13}$};
\draw (3.3 ,0)--(1.9,0.3)--(0.5,0);\draw (3.3 ,0 )--(1.9 ,0.9)--(0.5, 0);
\draw (3.3 ,0)--(1.9,-0.3)--(0.5,0);\draw (3.3 ,0 )--(1.9 ,-0.9)--(0.5, 0);
\filldraw [black] (-1.2,-1.1) circle (1.3pt);\draw  (-1.28,-0.9)  node { $u_{14}$};
\draw  (-1.2,-1.1)--(-1.9,-0.9);\draw  (-1.2,-1.1)--(0 ,-0.5);
\draw  (-1.2,-1.1)..controls+(0.4,0.3)and+(-0.1,-1)..(0,0.5);
\draw  (-1.2,-1.1)..controls+(0.1, -1.2)and+(-0.1,0.2)..(0,-3.3);
\filldraw [black] (1.2,1.1) circle (1.3pt);\draw  (1.28,0.9)  node { $u_{13}$};
\draw  (1.2,1.1)--(1.9,0.9);\draw  (1.2,1.1)--(0 ,0.5);
\draw  (1.2,1.1)..controls+(-0.4,-0.3)and+(-0.1,1.2)..(0,-0.5);
\draw  (1.2,1.1)..controls+(-0.1, 1.2)and+(0.1,-0.2)..(0,3.3);
\draw (1.8,-1.7)..controls+(-1, 1)and+(-2, -1.2)..(1.2,1.1);
\filldraw [black] (1.8,-1.7) circle (1.3pt);
\draw  (1.95,-1.8)  node { $x$};
\draw  (3.3,0)--(1.8,-1.7)--(0,-3.3);
\draw  (1.9,-0.9)--(1.8,-1.7)--(0.9,-1.9);
\draw  (0.5,0)--(1.8,-1.7)--(0,-0.5);
\draw   (1.8,-1.7)..controls+(0.5,0.2)and+(-0.4,-1)..(3.7,0);
\draw  (3.7,0)..controls+(-0.01,1)and+(1,-1)..(0,3.3);
\draw   (1.8,-1.7)..controls+(-0.5,-0.5)and+(0.4,0.2)..(0,-3.7);
\draw (0,-3.7)..controls+(-1, 0.01)and+(1,-1)..(-3.3, 0);
\end{scope}
\end{tikzpicture}
\end{center}
\end{figure}

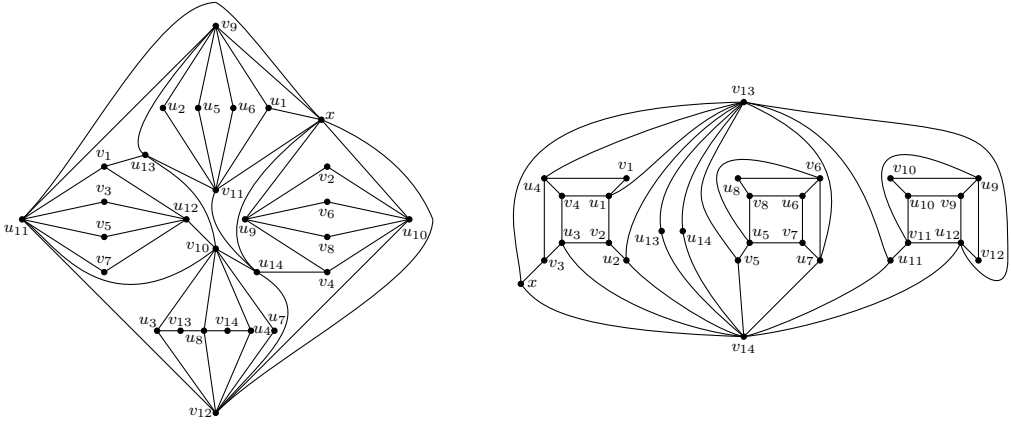
\begin{figure}[H]
\begin{center}
\begin{tikzpicture}
[scale=0.78]
\tikzstyle{every node}=[font=\tiny,scale=0.78]
\tiny
[inner sep=0pt]
\filldraw [black] (0.5,0) circle (1.3pt);\draw  (0.55,-0.2)  node { $u_{9}$};
\filldraw [black](3.3,0) circle (1.3pt);\draw  (3.4, -0.25)  node { $u_{10}$};
\filldraw [black](1.9,-0.3) circle (1.3pt);\draw  (1.9,-0.5)  node { $v_{8}$};
\filldraw [black](1.9,0.3) circle (1.3pt);\draw  (1.9,0.1)  node { $v_{6}$};
\filldraw [black](1.9,-0.9) circle (1.3pt);\draw  (1.9,-1.1)  node { $v_{4}$};
\filldraw [black](1.9,0.9) circle (1.3pt);\draw  (1.9,0.7)  node { $v_{2}$};
\filldraw [black](-0.5,0) circle (1.3pt);\draw  (-0.5,0.2)  node { $u_{12}$};
\filldraw [black](-1.9,0.3) circle (1.3pt);\draw  (-1.9,0.5)  node { $v_{3}$};
\filldraw [black](-3.3,0) circle (1.3pt);\draw  (-3.39,-0.18)  node { $u_{11}$};
\filldraw [black](-1.9,-0.3) circle (1.3pt);\draw  (-1.9,-0.1)  node { $v_{5}$};
\filldraw [black](-1.9,0.9) circle (1.3pt);\draw  (-1.9,1.1)  node { $v_{1}$};
\filldraw [black](-1.9,-0.9) circle (1.3pt);\draw  (-1.9,-0.7)  node { $v_{7}$};
\filldraw [black](0,-0.5) circle (1.3pt);\draw  (-0.3,-0.47)  node { $v_{10}$};
\filldraw [black] (-0.2,-1.9) circle (1.3pt);\draw  (-0.36,-2.05)  node { $u_{8}$};
\filldraw [black](0,-3.3) circle (1.3pt);\draw  (-0.25,-3.3)  node { $v_{12}$};
\filldraw [black](0.2, -1.9) circle (1.3pt);\draw  (0.2,-1.74)  node { $v_{14}$};
\filldraw [black](0.6, -1.9) circle (1.3pt);\draw  (0.8,-1.9)  node { $u_{4}$};
\filldraw [black](-0.6,- 1.9) circle (1.3pt);\draw  (-0.6,-1.74)  node { $v_{13}$};
\filldraw [black](1, -1.9) circle (1.3pt);\draw  (1.05,-1.7)  node { $u_{7}$};
\filldraw [black](-1, -1.9) circle (1.3pt);\draw  (-1.15,-1.75)  node { $u_{3}$};
\filldraw [black](0,0.5) circle (1.3pt);\draw  (0.3,0.45)  node { $v_{11}$};
\filldraw [black] (-0.3,1.9) circle (1.3pt);\draw  (-0.06,1.9)  node { $u_{5}$};
\filldraw [black](0,3.3) circle (1.3pt);\draw  (0.24,3.3)  node { $v_{9}$};
\filldraw [black] (0.3,1.9) circle (1.3pt);\draw  (0.54,1.9)  node { $u_{6}$};
\filldraw [black] (-0.9,1.9) circle (1.3pt);\draw  (-0.66,1.9)  node { $u_{2}$};
\filldraw [black] (0.9,1.9) circle (1.3pt);\draw  (1.08,2)  node { $u_{1}$};
\draw (3.3,0)--(0,-3.3)--(-3.3,0)--(0,3.3);
\draw (-0.5,0)--(0,-0.5);
\draw  (0,-0.5)..controls+(-1.9,-1.5)and+(0.5,-0.3)..(-3.3,0);
\draw (0, -3.3)--(1,-1.9)--(0,-0.5);\draw (0, -3.3)--(0.6,-1.9)--(0,-0.5);
\draw (0.6,-1.9)--(0.2,-1.9)--(-0.2,-1.9)--(-0.6,-1.9)--(-1,-1.9);
\draw (0, -3.3)--(-0.2,-1.9)--(0,-0.5);\draw (0, -3.3)--(-1,-1.9)--(0,-0.5);
\draw (0, 3.3)--(-0.3,1.9)--(0,0.5);\draw (0, 3.3)--(-0.9,1.9)--(0,0.5);
\draw (0, 3.3)--(0.3,1.9)--(0,0.5);\draw (0, 3.3)--(0.9,1.9)--(0,0.5);
\draw (3.3 ,0)--(1.9,-0.3)--(0.5,0);\draw (3.3 ,0 )--(1.9 ,-0.9)--(0.5, 0);
\draw (3.3 ,0)--(1.9,0.3)--(0.5,0);\draw (3.3 ,0 )--(1.9 ,0.9)--(0.5, 0);
\draw (-3.3 ,0)--(-1.9,-0.3)--(-0.5,0);\draw (-3.3 ,0 )--(-1.9 ,-0.9)--(-0.5, 0);
\draw (-3.3 ,0)--(-1.9,0.3)--(-0.5,0);\draw (-3.3 ,0 )--(-1.9 ,0.9)--(-0.5, 0);
\filldraw [black] (0.7,-0.9) circle (1.3pt);\draw  (0.95,-0.77)  node { $u_{14}$};
\draw  (0.7,-0.9)--(1.9,-0.9);\draw   (0.7,-0.9)--(0 ,-0.5);
\draw   (0.7,-0.9)..controls+(-0.7,0.6)and+(0.2,-0,01)..(0,0.5);
\draw   (0.7,-0.9)..controls+(1.2,-0.7)and+(0.8,0.94)..(0,-3.3);
\draw   (0.7,-0.9)..controls+(-0.35,0.3)and+(-2,-1.5)..(1.8,1.7);
\filldraw [black] (-1.2,1.1) circle (1.3pt);\draw  (-1.23,0.9)  node { $u_{13}$};
\draw  (-1.2,1.1)--(-1.9,0.9);\draw  (-1.2,1.1)--(0,0.5);
\draw  (-1.2,1.1)..controls+(0.4,-0.3)and+(-0.1,1)..(0,-0.5);
\draw  (-1.2,1.1)..controls+(-0.6,0.3)and+(-0.01,-0.01)..(0,3.3);
\filldraw [black] (1.8,1.7) circle (1.3pt);
\draw  (1.95,1.78)   node { $x$};
\draw  (3.3,0)--(1.8,1.7)--(0,3.3);
\draw  (1.8,1.7)--(0.9,1.9);
\draw  (0.5,0)--(1.8,1.7)--(0,0.5);
\draw   (1.8,1.7)..controls+(0.5,-0.2)and+(-0.4,1)..(3.7,0);
\draw  (3.7,0)..controls+(-0.01,-1)and+(1,1)..(0,-3.3);
\draw   (1.8,1.7)..controls+(-0.5,0.5)and+(0.4,-0.2)..(0,3.7);
\draw (0,3.7)..controls+(-1, -0.01)and+(1,1)..(-3.3, 0);
\begin{scope}[xshift=8cm]
\filldraw [black]     (1,2) circle (1.3pt);
\draw (1,2.2) node {\tiny $v_{13}$};
\filldraw [black]     (1,-2) circle (1.3pt);
\draw (1,-2.2) node {\tiny $v_{14}$};
\draw[-](1,2)..controls+(-1,-0.2)and+(0.5,0.4)..(-2.4,0.7);
\draw[-](1,2)..controls+(-1,-0.3)and+(0.6,0.2)..(-1.3,0.4);
\draw[-](1,2)..controls+(-1,-0.5)and+(-0.1,0.8)..(-1,-0.7);
\draw[-](1,2)..controls+(-0.9,-1.6)and+(-0.9,1.2)..(0.9,-0.7);
\draw[-](1,2)..controls+(1.8,-0.8)and+(0.3,1.0)..(2.3,-0.7);
\draw[-](1,2)..controls+(2.2,-0.4)and+(-0.2,0.01)..(3.5,-0.7);
\draw[-](4.7,-0.4)..controls+(0.01,-0.3)and+(-0.01,-0.8)..(5.5,-0.7);
\draw[-](1,2)..controls+(4.1,-0.7)and+(0.01,2.7)..(5.5,-0.7);
\draw[-](1,-2)..controls+(-1,0.1)and+(0.41,-0.9)..(-2.1,-0.4);
\draw[-](1,-2)..controls+(-0.5,0.1)and+(0.3,-0.4)..(-1,-0.7);
\draw[-](1,-2)--(0.9,-0.7);
\draw[-](1,-2)--(2.3,-0.7);
\draw[-](1,-2)..controls+(0.5,0.1)and+(-0.2,-0.4)..(3.5,-0.7);
\draw[-](1,-2)..controls+(1,0.1)and+(-0.41,-0.9)..(4.7,-0.4);
\filldraw [black]     (-2.8,-1.1) circle (1.3pt);
\draw (-2.6,-1.1) node {\tiny $x$};
\draw[-](1,2)..controls+(-5,-0.1)and+(0,0.9)..(-2.8,-1.1);
\draw[-](-2.8,-1.1)--(-2.4,-0.7);
\draw[-](1,-2)..controls+(-3,0.1)and+(0.3,-0.4)..(-2.8,-1.1);
\filldraw [black]     (-0.4,-0.2) circle (1.3pt);
\draw (-0.64,-0.35) node {\tiny $u_{13}$};
\filldraw [black]     (-0.04,-0.2) circle (1.3pt);
\draw (0.25,-0.35) node {\tiny $u_{14}$};
\draw[-](1,2)..controls+(-1.9,-1.6)and+(-1.9,1.6)..(1,-2);
\draw[-](1,2)..controls+(-1.4,-1.6)and+(-1.4,1.6)..(1,-2);
\filldraw [black]     (-2.1,0.4) circle (1.3pt);
\draw (-1.92,0.25) node {\tiny $v_{4}$};
\filldraw [black]     (-2.1,-0.4) circle (1.3pt);
\draw (-1.92,-0.25) node {\tiny $u_{3}$};
\filldraw [black]     (-1.3,0.4) circle (1.3pt);
\draw (-1.48,0.25) node {\tiny $u_{1}$};
\filldraw [black]     (-1.3,-0.4) circle (1.3pt);
\draw (-1.48,-0.25) node {\tiny $v_{2}$};
\draw[-](-2.1,0.4)--(-2.1,-0.4)--(-1.3,-0.4)--(-1.3,0.4)--(-2.1,0.4);
\filldraw [black]     (-2.4,0.7) circle (1.3pt);
\draw[-](-2.4,0.7)--(-2.1,0.4);
\draw (-2.6,0.55) node {\tiny $u_{4}$};
\draw[-](-1,0.7)--(-2.4,0.7)--(-2.4,-0.7);
\filldraw [black]     (-1,0.7) circle (1.3pt);
\draw[-](-1,0.7)--(-1.3,0.4);
\draw (-1,0.9) node {\tiny $v_{1}$};
\filldraw [black]     (-2.4,-0.7) circle (1.3pt);
\draw[-](-2.1,-0.4)--(-2.4,-0.7);
\draw (-2.18,-0.78) node {\tiny $v_{3}$};
\filldraw [black]     (-1.0,-0.7) circle (1.3pt);
\draw[-](-1.0,-0.7) --(-1.3,-0.4);
\draw (-1.25,-0.7) node {\tiny $u_{2}$};
\filldraw [black]     (2,0.4) circle (1.3pt);
\draw (1.8,0.25) node {\tiny $u_{6}$};
\filldraw [black]     (2,-0.4) circle (1.3pt);
\draw (1.8,-0.25) node {\tiny $v_{7}$};
\filldraw [black]     (1.1,0.4) circle (1.3pt);
\draw (1.3,0.25) node {\tiny $v_{8}$};
\filldraw [black]     (1.1,-0.4) circle (1.3pt);
\draw (1.3,-0.25) node {\tiny $u_{5}$};
\draw[-](2,0.4)--(2,-0.4)--(1.1,-0.4)--(1.1,0.4)--(2,0.4);
\filldraw [black]     (2.3,0.7) circle (1.3pt);
\draw[-](2.3,0.7)--(2,0.4);
\draw (2.2,0.9) node {\tiny $v_{6}$};
\draw[-](0.9,0.7)--(2.3,0.7)--(2.3,-0.7);
\filldraw [black]     (0.9,0.7) circle (1.3pt);
\draw[-](0.9,0.7)--(1.1,0.4);
\draw (0.8,0.5) node {\tiny $u_{8}$};
\filldraw [black]     (2.3,-0.7) circle (1.3pt);
\draw[-](2,-0.4)--(2.3,-0.7);
\draw (2.05,-0.7) node {\tiny $u_{7}$};
\filldraw [black]     (0.9,-0.7) circle (1.3pt);
\draw[-](0.9,-0.7) --(1.1,-0.4);
\draw (1.15,-0.7) node {\tiny $v_{5}$};
\draw[-](1.1,-0.4)..controls+(-0.4,0.5)and+(-2.7,1)..(2.3,0.7);
\filldraw [black]     (4.7,0.4) circle (1.3pt);
\draw (4.5,0.25) node {\tiny $v_{9}$};
\filldraw [black]     (4.7,-0.4) circle (1.3pt);
\draw (4.48,-0.25) node {\tiny $u_{12}$};
\filldraw [black]     (3.8,0.4) circle (1.3pt);
\draw (4.05,0.25) node {\tiny $u_{10}$};
\filldraw [black]     (3.8,-0.4) circle (1.3pt);
\draw (4.02,-0.28) node {\tiny $v_{11}$};
\draw[-](4.7,0.4)--(4.7,-0.4)--(3.8,-0.4)--(3.8,0.4)--(4.7,0.4);
\filldraw [black]     (5,0.7) circle (1.3pt);
\draw[-](5,0.7)--(4.7,0.4);
\draw (5.2,0.56) node {\tiny $u_{9}$};
\draw[-](3.5,0.7)--(5,0.7)--(5,-0.7);
\filldraw [black]     (3.5,0.7) circle (1.3pt);
\draw[-](3.5,0.7)--(3.8,0.4);
\draw (3.75,0.85) node {\tiny $v_{10}$};
\filldraw [black]     (5,-0.7) circle (1.3pt);
\draw[-](4.7,-0.4)--(5,-0.7);
\draw (5.25,-0.6) node {\tiny $v_{12}$};
\filldraw [black]     (3.5,-0.7) circle (1.3pt);
\draw[-](3.8,-0.4) --(3.5,-0.7);
\draw (3.85,-0.7) node {\tiny $u_{11}$};
\draw[-](3.8,-0.4)..controls+(-0.3,0.5)and+(-2.5,1.2)..(5,0.7);
\end{scope}
\end{tikzpicture}
\caption{A planar decomposition of $K_{1,14,14}$}
\label{figure 6}
\end{center}
\end{figure}

Lemma follows from Cases 1, 2 and 3.\end{proof}

\begin{theorem}\label{main1} The thickness of the complete 3-partite graph $K_{1,n,n}$ is
$$\theta(K_{1,n,n})=\big\lceil\frac{n+2}{4}\big\rceil.$$
\end{theorem}

\begin{proof}
When $n=4p,4p+3$, the theorem follows from Lemma \ref{2.2}.

When $n=4p+1,n=4p+2$, from Lemma \ref{2.3},  we have $\theta(K_{1,4p+2,4p+2})\leq p+1$. Since $\theta(K_{4p,4p})=p+1$ and $K_{4p,4p}\subset K_{1,4p+1,4p+1}\subset K_{1,4p+2,4p+2}$, we obtain $$p+1\leq \theta(K_{1,4p+1,4p+1})\leq \theta(K_{1,4p+2,4p+2})\leq p+1.$$
Therefore, $\theta(K_{1,4p+1,4p+1})=\theta(K_{1,4p+2,4p+2})=p+1$.

Summarizing the above, the theorem is obtained.\end{proof}

\section{The thickness of $K_{2,n,n}$ }

\begin{lemma}\label{2.4} There exists a planar decomposition of the complete 3-partite graph $K_{2,4p+1,4p+1}$ $(p\geq 0)$ with $p+1$ subgraphs.\end{lemma}

\begin{proof}~~Let  $(X,U,V)$ be the vertex partition of the complete 3-partite graph $K_{2,n,n}$, in which $X=\{x_1,x_2\}$, $U=\{u_1, \dots, u_{n}\}$ and $V=\{v_1, \dots,v_{n}\}$. When $n=4p+1$, we will construct a planar decomposition of $K_{2,4p+1,4p+1}$ with $p+1$ planar subgraphs.

The construction is analogous to that in Lemma \ref{2.3}. Let $\{G_1,$ $G_2,\dots,G_{p+1}\}$ be a planar decomposition of $K_{4p,4p}$ given in \cite{CYi16}. In the following, for $1\leq r\leq p+1$, by adding vertices $x_1,x_2,u_{4p+1},v_{4p+1}$ to $G_r$, deleting some edges from $G_{r}$ and adding some edges to $G_{r}$, we will get a new planar graph $\widehat{G}_r$ such that $\{\widehat{G}_1,\dots,\widehat{G}_{p+1}\}$ is a planar decomposition of $K_{2,4p+1,4p+1}$. All the subscripts of vertices are taken modulo $4p$, except that of $u_{4p+1} $ and $v_{4p+1}$ (the vertices we added to $G_{r}$).

{\bf Case 1.}~  When $p$ is even and $p> 2$.

\noindent{\bf (a)}~ The construction for $\widehat{G}_r$ , $1\leq r\leq p$.

\noindent {\bf Step 1:}~~ When $r$ is odd, place the vertex $x_1$,$x_2$ and $u_{4p+1}$ in the face $1$,$2$ and $5$ of $G_{r}$ respectively. Delete edges $v_{4r-3}u_{4r}$ and $u_{4r-1}v_{4r-2}$ from $G_r$.

When $r$ is even, place the vertex $x_1$,$x_2$ and $u_{4p+1}$ in the face $3$,$4$ and $5$ of $G_{r}$, respectively. Delete edge $v_{4r}u_{4r-3}$ and $u_{4r-2}v_{4r-1}$ from $G_r$.

\noindent {\bf Step 2:}~~  Do parallel paths modifications, then join $x_1$, $x_2$, $u_{4p+1}$ and $v_{4p+1}$ to some $u_{j}$ and $v_{j}$, as shown in Table \ref{tab 5}.

\begin{table}[H]
\centering\caption {The edges we add to $G_{r} (1\leq r\leq p)$ in Case 1}\label{tab 5}\tiny
\renewcommand{\arraystretch}{1.5}
\begin{tabular}{|*{10}{c|}}
\hline
\multicolumn{2}{|c|}{\diagbox { edge}{ subscript}{ case} } &  \multicolumn{4}{|c|}{$r$ is odd}  & \multicolumn{4}{|c|}{$r$ is even} \\
\hline
\multicolumn{2}{|c|}{$x_1u_j$}   & \multicolumn{3}{|c|}{$4r-1,4r$} & $4r+5~(U_1^r)$ & \multicolumn{3}{|c|}{$4r-3,4r-2$} & $4r+8~(U_2^r)$    \\
\hline
\multicolumn{2}{|c|}{$x_1v_j$}  & \multicolumn{3}{|c|}{$4r-3,4r-1$} & $4r+1~(V_1^r)$ & \multicolumn{3}{|c|}{$4r-2,4r$} & $4r+4~(V_2^r)$  \\
\hline
\multicolumn{2}{|c|}{$x_2u_j$}     & \multicolumn{3}{|c|}{$4r-1,4r$} & $4r+3~(U_2^r)$ & \multicolumn{3}{|c|}{$4r-3,4r-2$} & $4r+2~(U_1^r)$    \\
\hline
\multicolumn{2}{|c|}{$x_2v_j$}   & \multicolumn{3}{|c|}{$4 r-2,4r$} & $4r+7~(V_1^r)$ & \multicolumn{3}{|c|}{$4r-3,4r-1$} & $4r+6~(V_2^r)$  \\
\hline
\multicolumn{2}{|c|}{ $u_{4p+1}v_j$}   & \multicolumn{8}{|c|}{$4r-2,4r-1$}  \\
\hline
\multicolumn{2}{|c|}{ $v_{4p+1}u_j$ }   & \multicolumn{4}{|c|}{$4r+4,4r+8~(U_2^r)$}  & \multicolumn{4}{|c|}{$4r-11,4r-7~(U_1^r)$} \\
\hline
\end{tabular}
\end{table}

\noindent{\bf (b)}~ The construction for $\widehat{G}_{p+1}$.

We list the edges that belong to $K_{2,4p+1,4p+1}$ but not to any $\widehat{G}_{r}$, $1\leq r\leq p$, as shown in Table \ref{tab 6}.
Then $\widehat{G}_{p+1}$ is the graph that consists of the edges in Table \ref{tab 6}, Figure \ref{figure 7} shows $\widehat{G}_{p+1}$ is a planar graph.

\begin{table}[htp]
\centering\caption {The edges of $\widehat{G}_{p+1}$ in Case 1} \label{tab 6}
\tiny
\renewcommand{\arraystretch}{1.6}
\begin{tabular}{|*{2}{c|}}
\hline
\Gape[8pt]edges & subscript\\
\hline
$x_1u_j$ & \multirow{2}{*}{$j=4r-2,4r+3,4p+1.$ ($r=1,3,\dots,p-1.$)} \\
\cline{1-1}
$x_1v_j$ &  \\
\hline
$x_2u_j$ & \multirow{2}{*}{$j=4r-7,4r,4p+1.$ ($r=2,4,\dots,p.$)} \\
\cline{1-1}
$x_2v_j$ &  \\
\hline
$u_{4p+1}v_j$ &$j=4r-3,4r.$ ($r=1,2,\dots,p.$) \\
\cline{1-1}\cline{2-2}
$v_{4p+1}u_j$ & $j=4r-2,4r-1.$ ($r=1,2,\dots,p.$) \\
\hline
$v_{4r-3}u_{4r},v_{4r-2}u_{4r-1} $ &  $r=1,3,\dots,p-1.$ \\
\hline
$u_{4r-3}v_{4r},u_{4r-2}v_{4r-1} $ &  $r=2,4,\dots,p.$  \\
\hline
$u_{j}v_{j}$ &  $j=1,\dots,4p+1.$  \\
\hline
\end{tabular}
\end{table}
\vskip-0.3cm
\begin{figure}[H]
\begin{center}
\begin{tikzpicture}
\tiny
[inner sep=0pt]
\filldraw [black] (2.6,1.8) circle (1.4pt)
                  (2.6,-1.3) circle (1.4pt)
                  (-2.6,1.8) circle (1.4pt)
                  (-2.6,-1.3) circle (1.4pt);
\draw (2.6,1.8)-- (-2.6,1.8) ;\draw (2.6,-1.3)-- (-2.6,-1.3) ;
\draw[-] (2.6,1.8)..controls+(4,0.1)and+(4,0.3).. (2.6,-1.3);
\draw[-] (-2.6,1.8)..controls+(-4,-0.3)and+(-4,-0.1).. (-2.6,-1.3);
\draw[-] (-2.6,1.8)..controls+(5.4,-0.2)and+(-5,0.2).. (2.6,-1.3);
\draw  (2.6,2)  node { $x_2$};\draw  (-2.6,2) node { $v_{4p+1}$};
\draw  (2.6,-1.5)  node { $u_{4p+1}$};\draw  (-2.6,-1.5) node { $x_1$};
\draw (-2.6,-1.3)-- (-4.4,-0.5) ;\draw (-2.6,-1.3)--(-2,-0.5)  ;
\draw[-] (-2.6,-1.3)..controls+(-2,0.1)and+(-1,-1.8)..(-4.4,1);
\draw[-] (-2.6,-1.3)..controls+(0.2,0.6)and+(-0.4,-0.4)..(-2,1);
\draw (2.6,1.8)-- (2,1) ;\draw (2.6,1.8)-- (4.4,1)  ;
\draw[-] (2.6,1.8)..controls+(-0.3,-0.1)and+(-0.6,1.7)..(2,-0.5);
\draw[-] (2.6,1.8)..controls+(2,-0.1)and+(0.96,1.8)..(4.4,-0.5);
\draw (-2.6,1.8)--  (-1,1);\draw (-2.6,1.8)--(-2,1) ;
\draw (-2.6,1.8)-- (-3.4,1)  ;\draw (-2.6,1.8)-- (-4.4,1) ;
\draw (2.6,-1.3)--  (1,-0.5);\draw (2.6,-1.3)--(2,-0.5)  ;
\draw (2.6,-1.3)-- (3.4,-0.5)  ;\draw (2.6,-1.3)-- (4.4,-0.5)  ;
\filldraw [black] (-1,1) circle (1.4pt)
                  (-1,-0.5) circle (1.4pt)
                  (-2,1) circle (1.4pt)
                  (-2,-0.5) circle (1.4pt);
\filldraw [black] (-3.4,1) circle (1.4pt)
                  (-3.4,-0.5) circle (1.4pt)
                  (-4.4,1) circle (1.4pt)
                  (-4.4,-0.5) circle (1.4pt);
\draw  (-3.4,1)--(-3.4,-0.5) ;\draw  (-4.4,1) --(-4.4,-0.5);
\draw (-4.4,-0.5)--(-3.4,1);
\draw  (-4.62,-0.5)  node { $v_{2}$};\draw  (-4.64,1) node { $u_{2}$};
\draw (-3.64,-0.5)  node { $v_{3}$};\draw   (-3.64,1)node { $u_{3}$};
\draw  (-1,1)--(-1,-0.5) ;\draw  (-2,1) --(-2,-0.5);
\draw (-2,-0.5)--(-1,1);
\draw  (-0.53,-0.5)  node { $v_{4p-2}$};\draw  (-0.53,1) node { $u_{4p-2}$};
\draw (-1.48,-0.5)  node { $v_{4p-1}$};\draw   (-1.53,1)node { $u_{4p-1}$};
\filldraw [black] (1,1) circle (1.4pt)
                  (1,-0.5) circle (1.4pt)
                  (2,1) circle (1.4pt)
                  (2,-0.5) circle (1.4pt);
\filldraw [black] (3.4,1) circle (1.4pt)
                  (3.4,-0.5) circle (1.4pt)
                  (4.4,1) circle (1.4pt)
                  (4.4,-0.5) circle (1.4pt);
\draw  (3.4,1)--(3.4,-0.5) ;\draw  (4.4,1) --(4.4,-0.5);
\draw (4.4,-0.5)--(3.4,1);
\draw  (4.64,-0.5)  node { $v_{1}$};\draw  (4.62,1) node { $u_{1}$};
\draw (3.64,-0.5)  node { $v_{4}$};\draw   (3.64,1)node { $u_{4}$};
\draw  (1,1)--(1,-0.5) ;\draw  (2,1) --(2,-0.5);
\draw (2,-0.5)--(1,1);
\draw  (2.34,-0.5)  node { $v_{4p}$};\draw  (2.34,1) node { $u_{4p}$};
\draw (1.44,-0.43)  node { $v_{4p-3}$};\draw   (1.45,0.97)node { $u_{4p-3}$};
\filldraw [black] (2.5,0.25) circle (1pt)
                  (2.7,0.25) circle (1pt)
                  (2.9,0.25) circle (1pt);
\filldraw [black] (-2.5,0.25) circle (1pt)
                  (-2.7,0.25) circle (1pt)
                  (-2.9,0.25) circle (1pt);
\end{tikzpicture}
\vskip-0.3cm
\caption{The graph $\widehat{G}_{p+1}$ in Case 1}
\label{figure 7}
\end{center}
\end{figure}
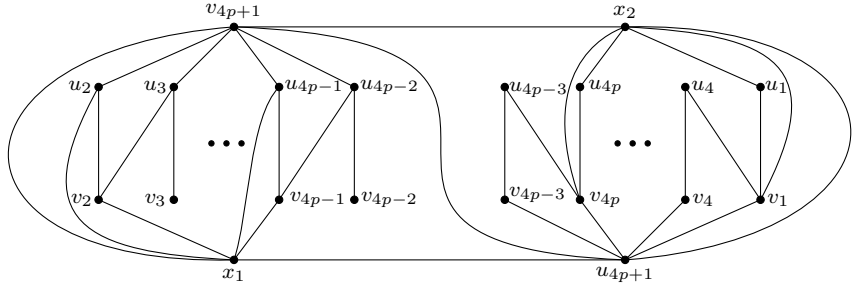

Therefore, $\{\widehat{G}_1,\dots,\widehat{G}_{p+1}\}$ is a planar decomposition of $K_{2,4p+1,4p+1}$ in this case. In Figure \ref{figure 13}, we draw the planar decomposition of $K_{2,17,17}$ it is the smallest example for the Case $1$.  We denote vertex $u_i$ and $v_i$ by $i$ and $i'$ respectively in this figure.

\begin{figure}[H]
\begin{center}
\begin{tikzpicture}
[scale=0.85]
\tikzstyle{every node}=[font=\tiny,scale=0.7]
\tiny
[inner sep=0pt]
\filldraw [black] (0.5,0) circle (1.3pt);\draw  (0.5,-0.18)  node { $1$};
\filldraw [black](3.6,0) circle (1.3pt);\draw  (3.58,-0.18)  node { $2$};
\filldraw [black](1.9,-1.1) circle (1.3pt);\draw  (1.9,-1.28)  node { $16'$};
\filldraw [black](1.9,-0.7) circle (1.3pt);\draw  (1.9,-0.88)  node { $14'$};
\filldraw [black](1.9,-0.3) circle (1.3pt);\draw  (1.9,-0.48)  node { $12'$};
\filldraw [black] (1.9,0.1) circle (1.3pt);\draw  (1.9,-0.08)  node { $10'$};
\filldraw [black] (1.9,0.5) circle (1.3pt);\draw  (1.9,0.32)  node { $8'$};
\filldraw [black](1.9,0.9) circle (1.3pt);\draw  (1.9,0.72)  node { $6'$};
\filldraw [black] (-0.5,0) circle (1.3pt);\draw  (-0.54,-0.2)  node { $4$};
\filldraw [black](-3.6,0) circle (1.3pt);\draw  (-3.6,-0.2)  node { $3$};
 \filldraw [black](-1.9,0.9) circle (1.3pt);\draw  (-1.94,1.05)  node { $5'$};
 \filldraw [black] (-1.9,0.5) circle (1.3pt);\draw  (-1.9,0.68)  node { $7'$};
  \filldraw [black] (-1.9,0.1) circle (1.3pt);\draw  (-1.9,0.28)  node { $9'$};
 \filldraw [black](-1.9,-0.3) circle (1.3pt);\draw  (-1.88,-0.12)  node { $13'$};
 \filldraw [black] (-1.9,-0.7) circle (1.3pt);\draw  (-1.88,-0.52)  node { $15'$};
   \filldraw [black] (-1.9,-1.1) circle (1.3pt);\draw  (-1.88,-0.92)  node { $11'$};
\filldraw [black](0,-0.5) circle (1.3pt);\draw  (0.22,-0.51)  node { $2'$};
\filldraw [black](0,-3.6) circle (1.3pt);\draw  (0.2,-3.6)  node { $4'$};
\filldraw [black] (1,-1.9) circle (1.3pt);\draw  (1.03,-1.74)  node { $12$};
\filldraw [black] (0.7,-1.9) circle (1.3pt);\draw  (0.65,-1.74)  node { $17'$};
\filldraw [black] (0.4,-1.9) circle (1.3pt);\draw  (0.25,-1.9)  node { $8$};
\filldraw [black] (0.05,-1.9) circle (1.3pt);\draw  (-0.14,-1.9)  node { $16$};
 \filldraw [black] (-0.33,-1.9) circle (1.3pt);\draw  (-0.52,-1.9)  node { $15$};
\filldraw [black](-0.7,-1.9) circle (1.3pt);\draw  (-0.88,-1.9)  node { $11$};
\filldraw [black](-1.1,-1.9) circle (1.3pt);\draw  (-1.1,-1.74)  node { $7$};

\filldraw [black] (0,0.5) circle (1.3pt);\draw  (0.24,0.49)  node { $3'$};
\filldraw [black] (0,3.6) circle (1.3pt);\draw  (0.2, 3.64)  node { $1'$};
\filldraw [black](0.9, 1.9) circle (1.3pt);\draw  (1.09,1.9)  node { $14$};
\filldraw [black](0.5, 1.9) circle (1.3pt);\draw  (0.67,1.9)  node { $13$};
\filldraw [black](0.1, 1.9) circle (1.3pt);\draw  (0.28,1.9)  node { $10$};
\filldraw [black](-0.3,1.9) circle (1.3pt);\draw  (-0.15,1.9)  node { $6$};
\filldraw [black](-0.7, 1.9) circle (1.3pt);\draw  (-0.55,1.9)  node { $5$};
\filldraw [black](-1.1, 1.9) circle (1.3pt);\draw  (-0.95,1.9)  node { $9$};
\draw (3.6,0)--(0,-3.6)--(-3.6,0)--(0,3.6)--(3.6,0);
\draw (0,-0.5)-- (0.5,0)--(0,0.5)--(-0.5,0)--(0,-0.5);
\draw  (0.5,0)..controls+(1.7,-1.3)and+(0.5,0.6)..(0, -3.6);
\draw  (0,0.5)..controls+(1.7,1.3)and+(-0.5,0.5)..(3.6,0);
\draw (0, -3.6)--(-1.1,-1.9)--(0,-0.5);\draw (0, -3.6)--(-0.7,-1.9)--(0,-0.5);
\draw (0, -3.6)--(-0.35,-1.9)--(0,-0.5);\draw (0, -3.6)--(0.05,-1.9)--(0,-0.5);
\draw (0, -3.6)--(0.4,-1.9)--(0,-0.5);\draw (0, -3.6)--(1,-1.9)--(0,-0.5);
\draw (0.4,-1.9)--(1,-1.9);
\draw (0, 3.6)--(0.9,1.9)--(0,0.5);\draw (0, 3.6)--(0.5,1.9)--(0,0.5);
\draw (0, 3.6)--(0.1,1.9)--(0,0.5);\draw (0, 3.6)--(-0.3,1.9)--(0,0.5);
\draw (0, 3.6)--(-0.7,1.9)--(0,0.5);\draw (0, 3.6)--(-1.1,1.9)--(0,0.5);
\draw (-3.6 ,0)--(-1.9,0.9)--(-0.5,0);\draw (-3.6 ,0 )--(-1.9 ,0.5)--(-0.5, 0);
\draw (-3.6 ,0)--(-1.9,0.1)--(-0.5,0);\draw (-3.6 ,0 )--(-1.9 ,-0.3)--(-0.5, 0);
\draw (-3.6 ,0)--(-1.9,-0.7)--(-0.5,0);\draw (-3.6 ,0 )--(-1.9 ,-1.1)--(-0.5, 0);
\draw (3.6 ,0 )--(1.9 ,0.9)--(0.5, 0);\draw (3.6 ,0 )--(1.9 ,0.5)--(0.5, 0);
\draw (3.6 ,0)--(1.9,0.1)--(0.5,0);\draw (3.6 ,0)--(1.9,-0.3)--(0.5,0);
\draw (3.6 ,0 )--(1.9 ,-0.7)--(0.5, 0);\draw (3.6 ,0 )--(1.9 ,-1.1)--(0.5, 0);
\filldraw [black] (-1.7,1.5) circle (1.3pt);
\draw  (-1.45,1.5)  node { $x_1$};
\draw  (-1.7,1.5)--(-1.9,0.9);\draw  (-1.7,1.5)--(0 ,0.5);
\draw   (-1.7,1.5)--(-0.5, 0 );\draw   (-1.7,1.5)--(-1.1,1.9 );
\draw   (-1.7,1.5)--(0,3.6);
\draw   (-1.7,1.5)--(-3.6,0);
\filldraw [black] (-1.7,-1.5) circle (1.3pt);
\draw  (-1.45,-1.52)  node { $x_2$};
\draw  (-1.7,-1.5)--(-1.9,-1.1);\draw  (-1.7,-1.5)--(0 ,-0.5);
\draw   (-1.7,-1.5)--(-0.5, 0 );\draw   (-1.7,-1.5)--(-1.1,-1.9 );
\draw   (-1.7,-1.5)--(0,-3.6);
\draw   (-1.7,-1.5)--(-3.6,0);
\filldraw [black] (0,0) circle (1.3pt);
\draw  (0.18,0)  node { $17$};
\draw  (0,0.5)-- (0,0)--(0 ,-0.5);

\begin{scope}[xshift=8.3cm]
\filldraw [black] (-0.5,0) circle (1.3pt);\draw  (-0.5,0.18)  node { $8$};
\filldraw [black](-3.6,0) circle (1.3pt);\draw  (-3.58,0.18)  node { $7$};
\filldraw [black](-1.9,1.1) circle (1.3pt);\draw  (-1.9,1.28)  node { $1'$};
\filldraw [black](-1.9,0.7) circle (1.3pt);\draw  (-1.9,0.88)  node { $3'$};
\filldraw [black](-1.9,0.3) circle (1.3pt);\draw  (-1.9,0.48)  node { $9'$};
\filldraw [black] (-1.9,-0.1) circle (1.3pt);\draw  (-1.87,0.08)  node { $11'$};
\filldraw [black] (-1.9,-0.5) circle (1.3pt);\draw  (-1.87,-0.32)  node { $13'$};
\filldraw [black](-1.9,-0.9) circle (1.3pt);\draw  (-1.87,-0.72)  node { $15'$};
\filldraw [black] (0.5,0) circle (1.3pt);\draw  (0.51,0.2)  node { $5$};
\filldraw [black](3.6,0) circle (1.3pt);\draw  (3.6,0.2)  node { $6$};
 \filldraw [black](1.9,-0.9) circle (1.3pt);\draw  (2.05,-1.04)  node { $12'$};
 \filldraw [black] (1.9,-0.5) circle (1.3pt);\draw  (1.9,-0.65)  node { $16'$};
  \filldraw [black] (1.9,-0.1) circle (1.3pt);\draw  (1.9,-0.27)  node { $10'$};
 \filldraw [black](1.9,0.3) circle (1.3pt);\draw  (1.9,0.13)  node { $4'$};
 \filldraw [black] (1.9,0.7) circle (1.3pt);\draw  (1.9,0.53)  node { $2'$};
   \filldraw [black] (1.9,1.1) circle (1.3pt);\draw  (1.9,0.9)  node { $14'$};
\filldraw [black](0,0.5) circle (1.3pt);\draw  (-0.18,0.53)  node { $7'$};
\filldraw [black](0,3.6) circle (1.3pt);\draw  (0.18,3.65)  node { $5'$};
\filldraw [black] (-1,1.9) circle (1.3pt);\draw  (-1.03,1.74)  node { $1$};
\filldraw [black] (-0.65,1.9) circle (1.3pt);\draw  (-0.65,1.74)  node { $17'$};
\filldraw [black] (-0.3,1.9) circle (1.3pt);\draw  (-0.45,2.05)  node { $13$};
\filldraw [black] (0,1.9) circle (1.3pt);\draw  (0.13,1.9)  node { $2$};
 \filldraw [black] (0.35,1.9) circle (1.3pt);\draw  (0.49,1.9)  node { $9$};
\filldraw [black](0.7,1.9) circle (1.3pt);\draw  (0.87,1.9)  node { $14$};
\filldraw [black](1.1,1.9) circle (1.3pt);\draw  (1.13,1.69)  node { $10$};
\filldraw [black] (0,-0.5) circle (1.3pt);\draw  (-0.27,-0.5)  node { $6'$};
\filldraw [black] (0,-3.6) circle (1.3pt);\draw  (0.2, -3.6)  node { $8'$};
\filldraw [black](-0.9, -1.9) circle (1.3pt);\draw  (-1.05,-1.9)  node { $3$};
\filldraw [black](-0.5, -1.9) circle (1.3pt);\draw  (-0.64,-1.9)  node { $4$};
\filldraw [black](-0.1, -1.9) circle (1.3pt);\draw  (-0.27,-1.9)  node { $11$};
\filldraw [black](0.3,-1.9) circle (1.3pt);\draw  (0.11,-1.9)  node { $12$};
\filldraw [black](0.7, -1.9) circle (1.3pt);\draw  (0.51,-1.9)  node { $15$};
\filldraw [black](1.1, -1.9) circle (1.3pt);\draw  (0.91,-1.9)  node { $16$};
\draw (3.6,0)--(0,-3.6)--(-3.6,0)--(0,3.6)--(3.6,0);
\draw (0,-0.5)-- (0.5,0)--(0,0.5)--(-0.5,0)--(0,-0.5);
\draw  (-0.5,0)..controls+(-1.7,1.3)and+(-0.5,-0.6)..(0, 3.6);
\draw  (0,-0.5)..controls+(-1.7,-1.3)and+(0.5,-0.4)..(-3.6,0);
\draw (0, 3.6)--(1.1,1.9)--(0,0.5);\draw (0, 3.6)--(0.7,1.9)--(0,0.5);
\draw (0, 3.6)--(0.35,1.9)--(0,0.5);\draw (0, 3.6)--(0,1.9)--(0,0.5);
\draw (0, 3.6)--(-0.3,1.9)--(0,0.5);\draw (0, 3.6)--(-1,1.9)--(0,0.5);
\draw (-0.3,1.9)--(-1,1.9);
\draw (0, -3.6)--(-0.9,-1.9)--(0,-0.5);\draw (0, -3.6)--(-0.5,-1.9)--(0,-0.5);
\draw (0, -3.6)--(-0.1,-1.9)--(0,-0.5);\draw (0,- 3.6)--(0.3,-1.9)--(0,-0.5);
\draw (0, -3.6)--(0.7,-1.9)--(0,-0.5);\draw (0, -3.6)--(1.1,-1.9)--(0,-0.5);
\draw (3.6 ,0)--(1.9,-0.9)--(0.5,0);\draw (3.6 ,0 )--(1.9 ,-0.5)--(0.5, 0);
\draw (3.6 ,0)--(1.9,-0.1)--(0.5,0);\draw (3.6 ,0 )--(1.9 ,0.3)--(0.5, 0);
\draw (3.6 ,0)--(1.9,0.7)--(0.5,0);\draw (3.6 ,0 )--(1.9 ,1.1)--(0.5, 0);
\draw (-3.6 ,0 )--(-1.9 ,-0.9)--(-0.5, 0);\draw (-3.6 ,0 )--(-1.9 ,-0.5)--(-0.5, 0);
\draw (-3.6 ,0)--(-1.9,-0.1)--(-0.5,0);\draw (-3.6 ,0)--(-1.9,0.3)--(-0.5,0);
\draw (-3.6 ,0 )--(-1.9 ,0.7)--(-0.5, 0);\draw (-3.6 ,0 )--(-1.9 ,1.1)--(-0.5, 0);
\filldraw [black] (1.7,-1.5) circle (1.3pt);
\draw  (1.45,-1.5)  node { $x_1$};
\draw  (1.7,-1.5)--(1.9,-0.9);\draw  (1.7,-1.5)--(0 ,-0.5);
\draw   (1.7,-1.5)--(0.5, 0 );\draw   (1.7,-1.5)--(1.1,-1.9 );
\draw   (1.7,-1.5)--(0,-3.6);
\draw   (1.7,-1.5)--(3.6,0);
\filldraw [black] (1.7,1.5) circle (1.3pt);
\draw  (1.45,1.52)  node { $x_2$};
\draw  (1.7,1.5)--(1.9,1.1);\draw  (1.7,1.5)--(0 ,0.5);
\draw   (1.7,1.5)--(0.5, 0 );\draw   (1.7,1.5)--(1.1,1.9 );
\draw   (1.7,1.5)--(0,3.6);
\draw   (1.7,1.5)--(3.6,0);
\filldraw [black] (0,0) circle (1.3pt);
\draw  (-0.18,0)  node { $17$};
\draw  (0,-0.5)-- (0,0)--(0 ,0.5);
\end{scope}
\end{tikzpicture}
\vskip-0.05cm    (a)  The graph $\widehat{G}_1$\  \ \ \ \ \ \  \  \ \ \ \ \  \ \ \ \ \ \ \ \ \ \  \  \ \ \ \  \ \ \ \ \  \ \ \     (b)  The graph $\widehat{G}_2$
\end{center}
\end{figure}
\vskip-0.6cm

\begin{figure}[H]
\begin{center}
\begin{tikzpicture}
[scale=0.85]
\tikzstyle{every node}=[font=\tiny,scale=0.7]
\tiny
[inner sep=0pt]
\filldraw [black] (0.5,0) circle (1.3pt);\draw  (0.5,-0.2)  node { $9$};
\filldraw [black](3.6,0) circle (1.3pt);\draw  (3.63,-0.2)  node { $10$};
\filldraw [black](1.9,-1.1) circle (1.3pt);\draw  (1.9,-1.25)  node { $16'$};
\filldraw [black](1.9,-0.7) circle (1.3pt);\draw  (1.9,-0.85)  node { $14'$};
\filldraw [black](1.9,-0.3) circle (1.3pt);\draw  (1.9,-0.45)  node { $8'$};
\filldraw [black] (1.9,0.1) circle (1.3pt);\draw  (1.9,-0.05)  node { $6'$};
\filldraw [black] (1.9,0.5) circle (1.3pt);\draw  (1.9,0.35)  node { $4'$};
\filldraw [black](1.9,0.9) circle (1.3pt);\draw  (1.9,0.75)  node { $2'$};
\filldraw [black] (-0.5,0) circle (1.3pt);\draw  (-0.5,-0.24)  node { $12$};
\filldraw [black](-3.6,0) circle (1.3pt);\draw  (-3.6,-0.24)  node { $11$};
 \filldraw [black](-1.9,0.9) circle (1.3pt);\draw  (-2.02,1.02)  node { $13'$};
 \filldraw [black] (-1.9,0.5) circle (1.3pt);\draw  (-1.9,0.68)  node { $1'$};
  \filldraw [black] (-1.9,0.1) circle (1.3pt);\draw  (-1.9,0.28)  node { $5'$};
 \filldraw [black](-1.9,-0.3) circle (1.3pt);\draw  (-1.9,-0.12)  node { $7'$};
 \filldraw [black] (-1.9,-0.7) circle (1.3pt);\draw  (-1.87,-0.52)  node { $15'$};
   \filldraw [black] (-1.9,-1.1) circle (1.3pt);\draw  (-1.87,-0.92)  node { $3'$};
\filldraw [black](0,-0.5) circle (1.3pt);\draw  (0.25,-0.51)  node { $10'$};
\filldraw [black](0,-3.6) circle (1.3pt);\draw  (0.25,-3.6)  node { $12'$};
\filldraw [black] (1,-1.9) circle (1.3pt);\draw  (1.05,-1.70)  node { $16$};
\filldraw [black] (0.65,-1.9) circle (1.3pt);\draw  (0.61,-1.70)  node { $17'$};
\filldraw [black] (0.3,-1.9) circle (1.3pt);\draw  (0.18,-1.9)  node { $4$};
\filldraw [black] (-0.05,-1.9) circle (1.3pt);\draw  (-0.18,-1.9)  node { $8$};
 \filldraw [black] (-0.4,-1.9) circle (1.3pt);\draw  (-0.54,-1.9)  node { $7$};
\filldraw [black](-0.75,-1.9) circle (1.3pt);\draw  (-0.89,-1.9)  node { $3$};
\filldraw [black](-1.1,-1.9) circle (1.3pt);\draw  (-1.12,-1.7)  node { $15$};

\filldraw [black] (0,0.5) circle (1.3pt);\draw  (0.3,0.5)  node { $11'$};
\filldraw [black] (0,3.6) circle (1.3pt);\draw  (0.2, 3.6)  node { $9'$};
\filldraw [black](0.9, 1.9) circle (1.3pt);\draw  (1.09,1.9)  node { $14$};
\filldraw [black](0.5, 1.9) circle (1.3pt);\draw  (0.67,1.9)  node { $13$};
\filldraw [black](0.1, 1.9) circle (1.3pt);\draw  (0.23,1.9)  node { $6$};
\filldraw [black](-0.3,1.9) circle (1.3pt);\draw  (-0.15,1.9)  node { $5$};
\filldraw [black](-0.7, 1.9) circle (1.3pt);\draw  (-0.55,1.9)  node { $2$};
\filldraw [black](-1.1, 1.9) circle (1.3pt);\draw  (-0.95,1.9)  node { $1$};
\draw (3.6,0)--(0,-3.6)--(-3.6,0)--(0,3.6)--(3.6,0);
\draw (0,-0.5)-- (0.5,0)--(0,0.5)--(-0.5,0)--(0,-0.5);
\draw  (0.5,0)..controls+(1.7,-1.3)and+(0.5,0.6)..(0, -3.6);
\draw  (0,0.5)..controls+(1.7,1.3)and+(-0.5,0.5)..(3.6,0);
\draw (0, -3.6)--(-1.1,-1.9)--(0,-0.5);\draw (0, -3.6)--(-0.75,-1.9)--(0,-0.5);
\draw (0, -3.6)--(-0.4,-1.9)--(0,-0.5);\draw (0, -3.6)--(-0.05,-1.9)--(0,-0.5);
\draw (0, -3.6)--(0.3,-1.9)--(0,-0.5);\draw (0, -3.6)--(1,-1.9)--(0,-0.5);
\draw (0.3,-1.9)--(1,-1.9);
\draw (0, 3.6)--(0.9,1.9)--(0,0.5);\draw (0, 3.6)--(0.5,1.9)--(0,0.5);
\draw (0, 3.6)--(0.1,1.9)--(0,0.5);\draw (0, 3.6)--(-0.3,1.9)--(0,0.5);
\draw (0, 3.6)--(-0.7,1.9)--(0,0.5);\draw (0, 3.6)--(-1.1,1.9)--(0,0.5);
\draw (-3.6 ,0)--(-1.9,0.9)--(-0.5,0);\draw (-3.6 ,0 )--(-1.9 ,0.5)--(-0.5, 0);
\draw (-3.6 ,0)--(-1.9,0.1)--(-0.5,0);\draw (-3.6 ,0 )--(-1.9 ,-0.3)--(-0.5, 0);
\draw (-3.6 ,0)--(-1.9,-0.7)--(-0.5,0);\draw (-3.6 ,0 )--(-1.9 ,-1.1)--(-0.5, 0);
\draw (3.6 ,0 )--(1.9 ,0.9)--(0.5, 0);\draw (3.6 ,0 )--(1.9 ,0.5)--(0.5, 0);
\draw (3.6 ,0)--(1.9,0.1)--(0.5,0);\draw (3.6 ,0)--(1.9,-0.3)--(0.5,0);
\draw (3.6 ,0 )--(1.9 ,-0.7)--(0.5, 0);\draw (3.6 ,0 )--(1.9 ,-1.1)--(0.5, 0);
\filldraw [black] (-1.7,1.5) circle (1.3pt);
\draw  (-1.45,1.5)  node { $x_1$};
\draw  (-1.7,1.5)--(-1.9,0.9);\draw  (-1.7,1.5)--(0 ,0.5);
\draw   (-1.7,1.5)--(-0.5, 0 );\draw   (-1.7,1.5)--(-1.1,1.9 );
\draw   (-1.7,1.5)--(0,3.6);
\draw   (-1.7,1.5)--(-3.6,0);
\filldraw [black] (-1.7,-1.5) circle (1.3pt);
\draw  (-1.45,-1.52)  node { $x_2$};
\draw  (-1.7,-1.5)--(-1.9,-1.1);\draw  (-1.7,-1.5)--(0 ,-0.5);
\draw   (-1.7,-1.5)--(-0.5, 0 );\draw   (-1.7,-1.5)--(-1.1,-1.9 );
\draw   (-1.7,-1.5)--(0,-3.6);
\draw   (-1.7,-1.5)--(-3.6,0);
\filldraw [black] (0,0) circle (1.3pt);
\draw  (0.18,0)  node { $17$};
\draw  (0,0.5)-- (0,0)--(0 ,-0.5);
\begin{scope}[xshift=8.3cm]
\filldraw [black] (-0.5,0) circle (1.3pt);\draw  (-0.52,0.22)  node { $16$};
\filldraw [black](-3.6,0) circle (1.3pt);\draw  (-3.62,0.22)  node { $15$};
\filldraw [black](-1.9,1.1) circle (1.3pt);\draw  (-1.9,1.27)  node { $1'$};
\filldraw [black](-1.9,0.7) circle (1.3pt);\draw  (-1.9,0.87)  node { $3'$};
\filldraw [black](-1.9,0.3) circle (1.3pt);\draw  (-1.9,0.47)  node { $5'$};
\filldraw [black] (-1.9,-0.1) circle (1.3pt);\draw  (-1.9,0.07)  node { $7'$};
\filldraw [black] (-1.9,-0.5) circle (1.3pt);\draw  (-1.94,-0.32)  node { $9'$};
\filldraw [black](-1.9,-0.9) circle (1.3pt);\draw  (-1.87,-0.70)  node { $11'$};
\filldraw [black] (0.5,0) circle (1.3pt);\draw  (0.5,0.23)  node { $13$};
\filldraw [black](3.6,0) circle (1.3pt);\draw  (3.63,0.23)  node { $14$};
 \filldraw [black](1.9,-0.9) circle (1.3pt);\draw  (2,-1.02)  node { $4'$};
 \filldraw [black] (1.9,-0.5) circle (1.3pt);\draw  (1.9,-0.65)  node { $12'$};
  \filldraw [black] (1.9,-0.1) circle (1.3pt);\draw  (1.9,-0.27)  node { $10'$};
 \filldraw [black](1.9,0.3) circle (1.3pt);\draw  (1.9,0.12)  node { $8'$};
 \filldraw [black] (1.9,0.7) circle (1.3pt);\draw  (1.9,0.52)  node { $2'$};
   \filldraw [black] (1.9,1.1) circle (1.3pt);\draw  (1.9,0.92)  node { $6'$};
\filldraw [black](0,0.5) circle (1.3pt);\draw  (-0.26,0.53)  node { $15'$};
\filldraw [black](0,3.6) circle (1.3pt);\draw  (0.24,3.65)  node { $13'$};
\filldraw [black] (-1,1.9) circle (1.3pt);\draw  (-1.03,1.71)  node { $5$};
\filldraw [black] (-0.65,1.9) circle (1.3pt);\draw  (-0.60,1.72)  node { $17'$};
\filldraw [black] (-0.3,1.9) circle (1.3pt);\draw  (-0.15,1.9)  node { $9$};
\filldraw [black] (0.05,1.9) circle (1.3pt);\draw  (0.17,1.9)  node { $1$};
 \filldraw [black] (0.35,1.9) circle (1.3pt);\draw  (0.5,1.9)  node { $6$};
\filldraw [black](0.7,1.9) circle (1.3pt);\draw  (0.88,1.9)  node { $10$};
\filldraw [black](1.1,1.9) circle (1.3pt);\draw  (1.13,1.74)  node { $2$};
\filldraw [black] (0,-0.5) circle (1.3pt);\draw  (-0.29,-0.5)  node { $14'$};
\filldraw [black] (0,-3.6) circle (1.3pt);\draw  (0.24, -3.6)  node { $16'$};
\filldraw [black](-0.9, -1.9) circle (1.3pt);\draw  (-1.05,-1.9)  node { $3$};
\filldraw [black](-0.5, -1.9) circle (1.3pt);\draw  (-0.64,-1.9)  node { $4$};
\filldraw [black](-0.1, -1.9) circle (1.3pt);\draw  (-0.23,-1.9)  node { $7$};
\filldraw [black](0.3,-1.9) circle (1.3pt);\draw  (0.12,-1.9)  node { $11$};
\filldraw [black](0.7, -1.9) circle (1.3pt);\draw  (0.52,-1.9)  node { $12$};
\filldraw [black](1.1, -1.9) circle (1.3pt);\draw  (0.95,-1.9)  node { $8$};
\draw (3.6,0)--(0,-3.6)--(-3.6,0)--(0,3.6)--(3.6,0);
\draw (0,-0.5)-- (0.5,0)--(0,0.5)--(-0.5,0)--(0,-0.5);
\draw  (-0.5,0)..controls+(-1.7,1.3)and+(-0.5,-0.6)..(0, 3.6);
\draw  (0,-0.5)..controls+(-1.7,-1.3)and+(0.5,-0.4)..(-3.6,0);
\draw (0, 3.6)--(1.1,1.9)--(0,0.5);\draw (0, 3.6)--(0.7,1.9)--(0,0.5);
\draw (0, 3.6)--(0.35,1.9)--(0,0.5);\draw (0, 3.6)--(0.05,1.9)--(0,0.5);
\draw (0, 3.6)--(-0.3,1.9)--(0,0.5);\draw (0, 3.6)--(-1,1.9)--(0,0.5);
\draw (-0.3,1.9)--(-1,1.9);
\draw (0, -3.6)--(-0.9,-1.9)--(0,-0.5);\draw (0, -3.6)--(-0.5,-1.9)--(0,-0.5);
\draw (0, -3.6)--(-0.1,-1.9)--(0,-0.5);\draw (0,- 3.6)--(0.3,-1.9)--(0,-0.5);
\draw (0, -3.6)--(0.7,-1.9)--(0,-0.5);\draw (0, -3.6)--(1.1,-1.9)--(0,-0.5);
\draw (3.6 ,0)--(1.9,-0.9)--(0.5,0);\draw (3.6 ,0 )--(1.9 ,-0.5)--(0.5, 0);
\draw (3.6 ,0)--(1.9,-0.1)--(0.5,0);\draw (3.6 ,0 )--(1.9 ,0.3)--(0.5, 0);
\draw (3.6 ,0)--(1.9,0.7)--(0.5,0);\draw (3.6 ,0 )--(1.9 ,1.1)--(0.5, 0);
\draw (-3.6 ,0 )--(-1.9 ,-0.9)--(-0.5, 0);\draw (-3.6 ,0 )--(-1.9 ,-0.5)--(-0.5, 0);
\draw (-3.6 ,0)--(-1.9,-0.1)--(-0.5,0);\draw (-3.6 ,0)--(-1.9,0.3)--(-0.5,0);
\draw (-3.6 ,0 )--(-1.9 ,0.7)--(-0.5, 0);\draw (-3.6 ,0 )--(-1.9 ,1.1)--(-0.5, 0);
\filldraw [black] (1.7,-1.5) circle (1.3pt);
\draw  (1.45,-1.5)  node { $x_1$};
\draw  (1.7,-1.5)--(1.9,-0.9);\draw  (1.7,-1.5)--(0 ,-0.5);
\draw   (1.7,-1.5)--(0.5, 0 );\draw   (1.7,-1.5)--(1.1,-1.9 );
\draw   (1.7,-1.5)--(0,-3.6);
\draw   (1.7,-1.5)--(3.6,0);
\filldraw [black] (1.7,1.5) circle (1.3pt);
\draw  (1.45,1.52)  node { $x_2$};
\draw  (1.7,1.5)--(1.9,1.1);\draw  (1.7,1.5)--(0 ,0.5);
\draw   (1.7,1.5)--(0.5, 0 );\draw   (1.7,1.5)--(1.1,1.9 );
\draw   (1.7,1.5)--(0,3.6);
\draw   (1.7,1.5)--(3.6,0);
\filldraw [black] (0,0) circle (1.3pt);
\draw  (-0.18,0)  node { $17$};
\draw  (0,-0.5)-- (0,0)--(0 ,0.5);
\end{scope}
\end{tikzpicture}
\vskip-0.05cm    (c)  The graph $\widehat{G}_3$\  \ \ \ \ \ \  \  \ \ \ \ \  \ \ \ \ \ \ \ \ \ \  \  \ \ \ \  \ \ \ \ \  \ \ \     (d)  The graph $\widehat{G}_4$
\end{center}
\end{figure}

\vskip-0.6cm
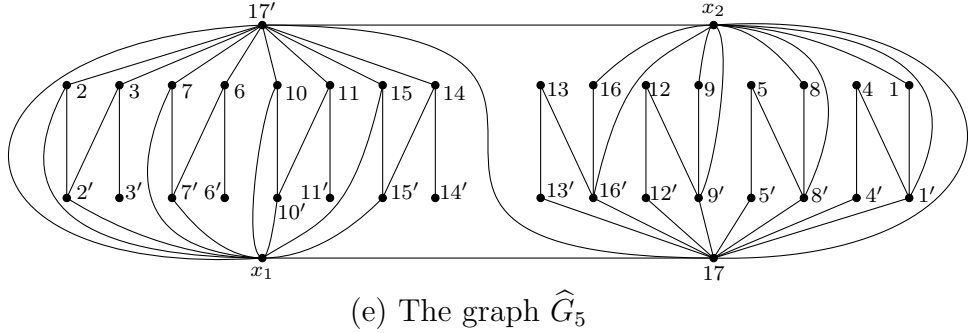
\begin{figure}[H]
\begin{center}
\begin{tikzpicture}
\tiny
[inner sep=0pt]
\filldraw [black] (3,1.8) circle (1.4pt)
                  (3,-1.3) circle (1.4pt)
                  (-3,1.8) circle (1.4pt)
                  (-3,-1.3) circle (1.4pt);
\draw (3,1.8)-- (-3,1.8) ;\draw (3,-1.3)-- (-3,-1.3) ;
\draw  (3,2)  node { $x_2$};\draw  (-3,2) node { $17'$};
\draw  (3,-1.5)  node { $17$};\draw  (-3,-1.5) node { $x_1$};
\draw[-] (-3,1.8)..controls+(-4.5,-0.2)and+(-4.5,-0.3)..(-3,-1.3);
\draw[-] (3,1.8)..controls+(4.5,0.3)and+(4.5,-0.1)..(3,-1.3);
\draw[-] (-3,1.8)..controls+(6,-0.02)and+(-6,0.02)..(3,-1.3);
\draw (-5.6,-0.5) parabola bend(-3,-1.3) (-1.4,-0.5);
\draw (-4.2,-0.5) parabola bend(-3,-1.3) (-2.8,-0.5);
\draw[-] (-3,-1.3)..controls+(-2,0.01)and+(-1,-1.8)..(-5.6,1);
\draw[-] (-3,-1.3)..controls+(-2,0.1)and+(-0.4,-0.8)..(-4.2,1);
\draw[-] (-3,-1.3)..controls+(-0.3,0.1)and+(-0.2,-0.2)..(-2.8,1);
\draw[-] (-3,-1.3)..controls+(0.8,0.5)and+(-0.4,-1.8)..(-1.4,1);
\draw (5.6,1) parabola bend(3,1.8) (1.4,1);
\draw (4.2,1) parabola bend(3,1.8) (2.8,1);
\draw[-] (3,1.8)..controls+(2,-0.01)and+(1,1.8)..(5.6,-0.5);
\draw[-] (3,1.8)..controls+(2,-0.1)and+(0.4,0.8)..(4.2,-0.5);
\draw[-] (3,1.8)..controls+(0.3,-0.1)and+(0.2,0.2)..(2.8,-0.5);
\draw[-] (3,1.8)..controls+(-0.8,-0.5)and+(0.4,1.8)..(1.4,-0.5);
\draw (-3,1.8)--  (-5.6,1);\draw (-3,1.8)--(-4.9,1) ;
\draw (-3,1.8)-- (-4.2,1)  ;\draw (-3,1.8)-- (-3.5,1) ;
\draw (-3,1.8)-- (-2.8,1)  ;\draw (-3,1.8)-- (-2.1,1) ;
\draw (-3,1.8)-- (-1.4,1)  ;\draw (-3,1.8)-- (-0.7,1) ;
\draw (3,-1.3)--  (4.9,-0.5);\draw (3,-1.3)--(4.2,-0.5)  ;
\draw (3,-1.3)-- (3.5,-0.5)  ;\draw (3,-1.3)-- (2.8,-0.5)  ;
\draw (3,-1.3)-- (2.1,-0.5)  ;\draw (3,-1.3)-- (1.4,-0.5)  ;
\draw (3,-1.3)-- (0.7,-0.5)  ;\draw (3,-1.3)-- (5.6,-0.5)  ;
\filldraw [black] (-0.7,1) circle (1.4pt)
                  (-0.7,-0.5) circle (1.4pt)
                  (-1.4,1) circle (1.4pt)
                  (-1.4,-0.5) circle (1.4pt);
\filldraw [black] (-2.1,1) circle (1.4pt)
                  (-2.1,-0.5) circle (1.4pt)
                  (-2.8,1) circle (1.4pt)
                  (-2.8,-0.5) circle (1.4pt);
\filldraw [black] (-3.5,1) circle (1.4pt)
                  (-3.5,-0.5) circle (1.4pt)
                  (-4.2,1) circle (1.4pt)
                  (-4.2,-0.5) circle (1.4pt);
\filldraw [black] (-4.9,1) circle (1.4pt)
                  (-4.9,-0.5) circle (1.4pt)
                  (-5.6,1) circle (1.4pt)
                  (-5.6,-0.5) circle (1.4pt);
\draw  (-5.6,1)--(-5.6,-0.5) ;\draw  (-4.9,1) --(-4.9,-0.5);
\draw (-5.6,-0.5)--(-4.9,1);
\draw  (-5.35,-0.40)  node { $2'$};\draw  (-5.4,0.92) node { $2$};
\draw (-4.7,-0.40)  node { $3'$};\draw   (-4.7,0.92)node { $3$};
\draw  (-4.2,1)--(-4.2,-0.5) ;\draw  (-3.5,1) --(-3.5,-0.5);
\draw (-4.2,-0.5)--(-3.5,1);
\draw  (-3.96,-0.40)  node { $7'$};\draw  (-4,0.92) node { $7$};
\draw (-3.65,-0.40)  node { $6'$};\draw   (-3.3,0.92)node { $6$};
\draw  (-2.8,1)--(-2.8,-0.5) ;\draw  (-2.1,1) --(-2.1,-0.5);
\draw (-2.8,-0.5)--(-2.1,1);
\draw  (-2.6,-0.67)  node { $10'$};\draw  (-2.55,0.90) node { $10$};
\draw (-2.3,-0.40)  node { $11'$};\draw   (-1.85,0.90)node { $11$};
\draw  (-1.4,1)--(-1.4,-0.5) ;\draw  (-0.7,1) --(-0.7,-0.5);
\draw (-1.4,-0.5)--(-0.7,1);
\draw  (-1.10,-0.40)  node { $15'$};\draw  (-1.15,0.90) node { $15$};
\draw (-0.45,-0.40)  node { $14'$};\draw   (-0.45,0.90)node { $14$};
\filldraw [black] (0.7,1) circle (1.4pt)
                  (0.7,-0.5) circle (1.4pt)
                  (1.4,1) circle (1.4pt)
                  (1.4,-0.5) circle (1.4pt);
\filldraw [black] (2.1,1) circle (1.4pt)
                  (2.1,-0.5) circle (1.4pt)
                  (2.8,1) circle (1.4pt)
                  (2.8,-0.5) circle (1.4pt);
\filldraw [black] (3.5,1) circle (1.4pt)
                  (3.5,-0.5) circle (1.4pt)
                  (4.2,1) circle (1.4pt)
                  (4.2,-0.5) circle (1.4pt);
\filldraw [black] (4.9,1) circle (1.4pt)
                  (4.9,-0.5) circle (1.4pt)
                  (5.6,1) circle (1.4pt)
                  (5.6,-0.5) circle (1.4pt);
\draw  (5.6,1)--(5.6,-0.5) ;\draw  (4.9,1) --(4.9,-0.5);
\draw (5.6,-0.5)--(4.9,1);
\draw  (5.4,0.95)  node { $1$};\draw  (5.84,-0.45) node { $1'$};
\draw (5.05,0.95)  node { $4$};\draw   (5.14,-0.45)node { $4'$};
\draw  (4.2,1)--(4.2,-0.5) ;\draw  (3.5,1) --(3.5,-0.5);
\draw (4.2,-0.5)--(3.5,1);
\draw  (4.35,0.95)  node { $8$};\draw  (4.42,-0.45) node { $8'$};
\draw (3.65,0.95)  node { $5$};\draw   (3.72,-0.45)node { $5'$};
\draw  (2.8,1)--(2.8,-0.5) ;\draw  (2.1,1) --(2.1,-0.5);
\draw (2.8,-0.5)--(2.1,1);
\draw  (2.95,0.95)  node { $9$};\draw  (3.04,-0.45) node { $9'$};
\draw (2.28,0.95)  node { $12$};\draw   (2.34,-0.42)node { $12'$};
\draw  (1.4,1)--(1.4,-0.5) ;\draw  (0.7,1) --(0.7,-0.5);
\draw (1.4,-0.5)--(0.7,1);
\draw  (1.63,0.95)  node { $16$};\draw  (1.67,-0.38) node { $16'$};
\draw (0.93,0.95)  node { $13$};\draw   (0.97,-0.38)node { $13'$};
\end{tikzpicture}
\vskip-0.05cm    (e)  The graph $\widehat{G}_5$
\caption{A planar decomposition of $K_{2,17,17}$}
\label{figure 13}
\end{center}
\end{figure}

{\bf Case 2.}~ When $p$ is odd and $p > 3$.

\noindent{\bf (a)}~ The construction for $\widehat{G}_r$ , $1\leq r\leq p$.

\noindent {\bf Step 1:}~~ When $r$ is odd, place the vertex $x_1$,$x_2$ and $u_{4p+1}$ in the face $1$,$2$ and $5$ of $G_{r}$ respectively. Delete edges $v_{4r-3}u_{4r}$ and $u_{4r-1}v_{4r-2}$ from $G_r$.

When $r$ is even, place the vertex $x_1$,$x_2$ and $u_{4p+1}$ in the face $3$,$4$ and $5$ of $G_{r}$, respectively. Delete edge $v_{4r}u_{4r-3}$ and $u_{4r-2}v_{4r-1}$ from $G_r$.

\noindent {\bf Step 2:}~~ Do parallel paths modifications, then join $x_1$, $x_2$, $u_{4p+1}$ and $v_{4p+1}$ to some $u_{j}$ and $v_{j}$, as shown in Table \ref{tab 7}.

\begin{table}[H]
\centering\caption {The edges we add to $G_{r} (1\leq r\leq p)$ in Case 2}\label{tab 7}\vskip-0.2cm\tiny
\renewcommand{\arraystretch}{1.1}
\begin{tabular}{|*{10}{c|}}
\hline
\multicolumn{2}{|c|}{\diagbox { \bahao edge}{\bahao subscript}{\bahao case} } &  \multicolumn{4}{|c|}{$r$ is odd}  & \multicolumn{4}{|c|}{$r$ is even} \\
\hline
\multicolumn{2}{|c|}{$x_1u_j$}    & \multicolumn{3}{|c|}{$4r-1,4r$}
 & \tabincell{c}{$4r+5,r\neq p~(U_1^r)$ \\  $1,r=p~(U_1^r)$}
 & \multicolumn{3}{|c|}{$4r-3,4r-2$}
 &  \tabincell{c}{ $4r+8,r\neq p-1~(U_2^r)$ \\ $8,r=p-1~(U_2^r)$}    \\
\hline
\multicolumn{2}{|c|}{$\Gape[8pt]x_1v_j$}    & \multicolumn{3}{|c|}{$4r-3,4r-1$} & $4r+1,r\neq p~(V_1^r)$ & \multicolumn{3}{|c|}{$4r-2,4r$}
& \tabincell{c}{ $4r+4~(V_2^r)$}\\
\hline
\multicolumn{2}{|c|}{$x_2u_j$}      & \multicolumn{3}{|c|}{$4r-1,4r$}
 & \tabincell{c}{$4r+3,r\neq p~(U_2^r)$ \\  $8,r=p~(U_2^r)$}
 & \multicolumn{3}{|c|}{$4r-3,4r-2$}
 &  \tabincell{c}{ $4r+2~(U_1^r)$}    \\
\hline
\multicolumn{2}{|c|}{$x_2v_j$}     & \multicolumn{3}{|c|}{$4r-2,4r$}
& \tabincell{c}{ $4r+7,r\neq p~(V_1^r)$\\ $3,r=p~(V_1^r)$}
 & \multicolumn{3}{|c|}{$4r-3,4r-1$}
& \tabincell{c}{ $4r+6,r\neq p-1~(V_2^r)$\\ $6,r=p-1~(V_2^r)$}\\
\hline
\multicolumn{2}{|c|}{$\Gape[8pt]u_{4p+1}v_j$}
& \multicolumn{8}{|c|}{$4r-2,4r-1$} \\
\cline{1-1}\cline{2-2}\cline{3-10}
\multicolumn{2}{|c|}{$\Gape[8pt]v_{4p+1}u_j$}
& \multicolumn{4}{|c|}{   \tabincell{c}{$4r+4,4r+8,r\neq p~(U_2^r)$\\$4,r=p~(U_2^r)$}}
&\multicolumn{4}{|c|}{ $4r-11,4r-7~(U_1^r)$}\\
\hline
\end{tabular}
\end{table}

\noindent{\bf (b)}~ The construction for $\widehat{G}_{p+1}$.

We list the edges that belong to $K_{2,4p+1,4p+1}$ but not to any $\widehat{G}_{r}$, $1\leq r\leq p$, as shown in Table \ref{tab 8}.
Then $\widehat{G}_{p+1}$ is the graph that consists of the edges in Table \ref{tab 8}, Figure \ref{figure 8} shows $\widehat{G}_{p+1}$ is a planar graph.

\vskip-0.2cm
\begin{table}[H]
\centering\caption {The edges of $\widehat{G}_{p+1}$ in Case 2
} \label{tab 8}\vskip-0.2cm
\tiny
\renewcommand{\arraystretch}{1.6}
\begin{tabular}{|*{2}{c|}}
\hline
\Gape[8pt]edges & subscript\\
\hline
$x_1u_j$ & {$j=2,4r+3,4r+6,4p+1.$ ($r=1,3,\dots,p-2.$)} \\
\hline
$x_1v_j$ & {$j=2,4,4r+3,4r+6,4p+1.$ ($r=1,3,\dots,p-2.$)} \\
\hline
$x_2u_j$ & {$j=1,2,9,4r,4r+1,4p+1.$ ($r=4,\dots,p-1.$)} \\
\hline
$x_2v_j$ &{$j=1,8,9,4r,4r+1,4p+1.$ ($r=4,\dots,p-1.$)} \\
\hline
$u_{4p+1}v_j$ &  $j=4r-3,4r$.($r=1,2,\dots,p.$) \\
\hline
$v_{4p+1}u_j$ &  $j=4r-2,4r-1,4p-7$.($r=1,2,\dots,p.$) \\
\hline
$v_{4r-3}u_{4r},v_{4r-2}u_{4r-1}$ &  $r=1,3,\dots,p.$ \\
\hline
$u_{4r-3}v_{4r},u_{4r-2}v_{4r-1}$ &  $r=2,4,\dots,p-1.$  \\
\hline
$u_{j}v_{j}$ &  $j=1,\dots,4p+1.$  \\
\hline
\end{tabular}
\end{table}

\vskip-1cm
\begin{figure}[H]
\begin{center}
\begin{tikzpicture}
[yscale=1]
\tikzstyle{every node}=[font=\tiny,scale=1]
\tiny
[inner sep=0pt]
\filldraw [black] (1.6,1) circle (1.4pt)
                  (1.6,-0.5) circle (1.4pt)
                  (2.44,1) circle (1.4pt)
                  (2.4,-0.5) circle (1.4pt)
                   (0.3,1) circle (1.4pt)
                  (0.3,-0.5) circle (1.4pt)
                  (-0.4,1) circle (1.4pt)
                  (-0.4,-0.5) circle (1.4pt);
\draw [white] (-0.43,-0.67) node [left]{}-- (-0.42,-0.65)node [black,pos=0.02,above,sloped]{ $v_{4p-2}$}-- (-0.42,-0.65) node [right]{};
\draw [white] (-0.9,0.87) node [left]{}-- (-0.83,0.83)node [black,pos=0.02,above,sloped]{ $ u_{4p-2}$}-- (-0.83,0.83) node [right]{};
\draw [white] (-0.2,0.86) node [left]{}-- (-0.13,0.83)node [black,pos=0.02,above,sloped]{ $u_{4p-1}$}-- (-0.13,0.83) node [right]{};
\draw [white] (0.22,-0.67) node [left]{}-- (0.23,-0.65)node [black,pos=0.02,above,sloped]{ $v_{4p-1}$}-- (0.23,-0.65) node [right]{};
\draw [white] (1.55,-0.9) node [left]{}-- (1.58,-0.91)node [black,pos=0.02,above,sloped]{ $v_{4p-7}$}-- (1.58,-0.91) node [right]{};
\draw [white] (1.15,0.8) node [left]{}-- (1.17,0.79)node [black,pos=0.02,above,sloped]{ $ u_{4p-7}$}-- (1.17,0.79) node [right]{};
\draw [white]  (1.94,0.8) node [left]{}--  (1.96,0.79)node [black,pos=0.02,above,sloped]{ $u_{4p-4}$}--  (1.96,0.79) node [right]{};
\draw [white](2.25,-0.86) node [left]{}-- (2.28,-0.87)node [black,pos=0.02,above,sloped]{ $v_{4p-4}$}--(2.28,-0.87) node [right]{};
\draw  (1.6,1)-- (1.6,-0.5) ;\draw  (2.44,1)-- (2.4,-0.5) ;\draw  (0.3,1)-- (0.3,-0.5) ;\draw   (-0.4,1)-- (-0.4,-0.5) ;
\draw  (1.6,1)--  (2.4,-0.5) ;\draw  (0.3,1)--  (-0.4,-0.5);
\filldraw [black] (4.6,1.8) circle (1.4pt)
                  (4.6,-1.3) circle (1.4pt)
                  (-2.6,1.8) circle (1.4pt)
                  (-2.6,-1.3) circle (1.4pt);
\draw (4.6,1.8)-- (-2.6,1.8) ;\draw (4.6,-1.3)-- (-2.6,-1.3) ;
\draw[-] (4.6,1.8)..controls+(0.7,-0.8)and+(0.7,0.8).. (4.6,-1.3);
\draw[-] (-2.6,1.8)..controls+(-0.7,-0.8)and+(-0.7,0.8).. (-2.6,-1.3);
\draw[-] (-2.6,1.8)..controls+(7.5,-0.25)and+(-7.5,0.25).. (4.6,-1.3);
\draw  (4.6,2)  node { $x_2$};\draw  (-2.3,1.92) node { $v_{4p+1}$};
\draw  (4.6,-1.45)  node { $u_{4p+1}$};\draw  (-2.6,-1.5) node { $x_1$};
\draw(-4.4,-0.5) parabola bend(-2.6,-1.3) (-2,-0.5) ;
\draw(-2.6,-1.3)..controls+(-2,0.1)and+(-1,-1.8)..(-4.4,1);
\draw (-2.6,-1.3)..controls+(-0.08,0.1)and+(-0.6,-1.5)..(-2,1);
\draw (-2.6,-1.3)..controls+(1,0.1)and+(0,-0.8).. (-0.4,-0.5);
 \draw (-2.6,-1.3)..controls+(1.9,0.3)and+(-0.4,-2).. (-0.4,1) ;
 \draw (-2.6,-1.3)..controls+(0.6,-0.1)and+(0.3,-2.1)..(6.4,-0.5);
\draw[-] (4.6,1.8)..controls+(-0.3,-0.3)and+(0.2,0.3).. (4,1);
\draw[-] (4.6,1.8)..controls+(0.6,-0.1)and+(-0.2,0.5)..(5.4,1);
\draw[-] (4.6,1.8)..controls+(-0.4,-0.15)and+(-0.4,2)..(4,-0.5);
\draw[-] (4.6,1.8)..controls+(0.4,-0.1)and+(-0.3,1.8)..(5.4,-0.5);
\draw[-] (4.6,1.8)..controls+(-0.7,-0.1)and+(0.4,0.7)..(2.44,1) ;
\draw[-] (4.6,1.8)..controls+(-1.9,-0.5)and+(0.4,2).. (2.4,-0.5);
\draw[-] (4.6,1.8)..controls+(-1.6,0.1)and+(0.4,2)..(-4.4,1);
\draw (-2.6,1.8)..controls+(-1,-0.1)and+(0.5,0.5).. (-4.4,1);
\draw (-2.6,1.8)..controls+(-0.5,-0.2)and+(0.05,0.05).. (-3.4,1);
\draw (-2.6,1.8)..controls+(0.5,-0.3)and+(-0.05,0.05).. (-2,1);
\draw (-2.6,1.8)..controls+(0.6,-0.2)and+(-0.3,0.4).. (-1,1);
\draw (-2.6,1.8)..controls+(0.9,-0.2)and+(-0.3,0.4).. (-0.4,1);
\draw (-2.6,1.8)..controls+(1.2,-0.15)and+(-0.3,0.4).. (0.3,1);
\draw (-2.6,1.8)..controls+(3,-0.05)and+(-0.1,0.7).. (1.6,1);
\draw (4.6,-1.3)..controls+(0.3,0.1)and+(-0.1,-0.2).. (5.4,-0.5);
\draw (4.6,-1.3)..controls+(-0.3,0.1)and+(0.05,-0.2).. (4,-0.5);
\draw (4.6,-1.3)..controls+(-1,0.3)and+(0.2,-0.4).. (3,-0.5);
\draw (4.6,-1.3)..controls+(0.6,0.1)and+(-0.4,-0.6).. (6.4,-0.5);
\draw (4.6,-1.3)..controls+(-2,0.2)and+(-0.1,-0.01)..  (1.6,-0.5);
\draw (4.6,-1.3)..controls+(-1.5,0.3)and+(0.6,-0.4).. (2.44,-0.5);
\filldraw [black] (-1,1) circle (1.4pt)
                  (-1,-0.5) circle (1.4pt)
                  (-2,1) circle (1.4pt)
                  (-2,-0.5) circle (1.4pt);
\filldraw [black] (-3.4,1) circle (1.4pt)
                  (-3.4,-0.5) circle (1.4pt)
                  (-4.4,1) circle (1.4pt)
                  (-4.4,-0.5) circle (1.4pt);
\draw  (-3.4,1)--(-3.4,-0.5) ;\draw  (-4.4,1) --(-4.4,-0.5);
\draw (-4.4,-0.5)--(-3.4,1);
\draw  (-4.6,-0.5)  node { $v_{2}$};\draw  (-4.6,1) node { $u_{2}$};
\draw (-3.6,-0.5)  node { $v_{3}$};\draw   (-3.6,1)node { $u_{3}$};
\draw  (-1,1)--(-1,-0.5) ;\draw  (-2,1) --(-2,-0.5);
\draw (-2,-0.5)--(-1,1);
\draw  (-2.45,1) node { $u_{4p-5}$};
\draw [white]  (-2.11,-0.76) node [left]{}-- (-2.05,-0.66)node [black,pos=0.02,above,sloped]{ $v_{4p-5}$}--  (-2.05,-0.66) node [right]{};
\draw (-1.45,-0.45)  node { $v_{4p-6}$};\draw   (-1.45,0.95)node { $u_{4p-6}$};
\filldraw [black] (3,1) circle (1.4pt)
                  (3,-0.5) circle (1.4pt)
                  (4,1) circle (1.4pt)
                  (4,-0.5) circle (1.4pt);
\filldraw [black] (5.4,1) circle (1.4pt)
                  (5.4,-0.5) circle (1.4pt)
                  (6.4,1) circle (1.4pt)
                  (6.4,-0.5) circle (1.4pt);
\draw  (5.4,1)--(5.4,-0.5) ;\draw  (6.4,1) --(6.4,-0.5);
\draw (6.4,1)--(5.4,-0.5);
\draw  (6.6,-0.5)  node { $v_{4}$};\draw  (6.64,1) node { $u_{4}$};
\draw (5.6,-0.5)  node { $v_{1}$};\draw   (5.6,1)node { $u_{1}$};
\draw  (3,1)--(3,-0.5) ;\draw  (4,1) --(4,-0.5);
\draw (4,-0.5)-- (3,1);
\draw  (4.44,-0.5)  node { $v_{4p-3}$};\draw  (4.47,0.9) node { $u_{4p-3}$};
\draw (3.31,-0.5)  node { $v_{4p}$};\draw   (3.31,0.9)node { $u_{4p}$};
\filldraw [black] (4.5,0.25) circle (1pt)
                  (4.7,0.25) circle (1pt)
                  (4.9,0.25) circle (1pt);
\filldraw [black] (-2.5,0.25) circle (1pt)
                  (-2.7,0.25) circle (1pt)
                  (-2.9,0.25) circle (1pt);
\end{tikzpicture}
\vskip-0.7cm
\caption{The graph $\widehat{G}_{p+1}$ in Case 2}
\label{figure 8}
\end{center}
\end{figure}
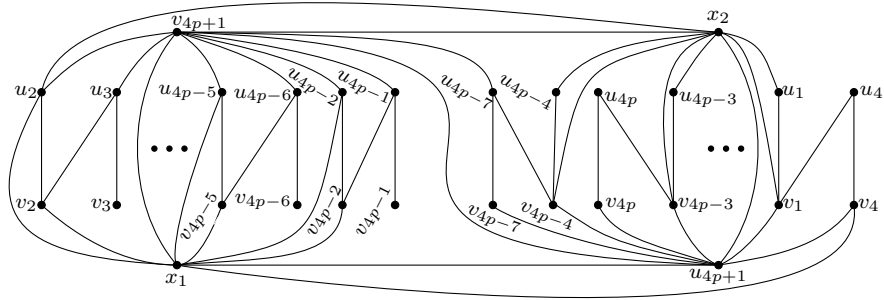
Therefore, $\{\widehat{G}_1,\dots,\widehat{G}_{p+1}\}$ is a planar decomposition of $K_{2,4p+1,4p+1}$ in this case.

{\bf Case 3.}~ When $p\leq 3$.

When $p=0$, $K_{2,1,1}$ is a planar graph. When $p=1,2,3$, we give a planar decomposition for $K_{2,5,5}$, $K_{2,9,9}$ and $K_{2,13,13}$ with $2$, $3$ and $4$ subgraphs respectively, as shown in Figure \ref{figure 9}, Figure \ref{figure 10} and Figure \ref{figure 11}.

\vskip-0.7cm
\begin{figure}[htp]
\begin{center}
\begin{tikzpicture}
[scale=1.3]
\tikzstyle{every node}=[font=\tiny,scale=1]
\tiny
[inner sep=0pt]
\filldraw [black] (0,0) circle (1pt)
                  (-0.5,0) circle (1pt)
                  (-1,0) circle (1pt)
                  (-1.5,0) circle (1pt)
                  (-2,0) circle (1pt)
                  (-2.5,0) circle (1pt)
                   (0.5,0) circle (1pt)
                  (1,0) circle (1pt)
                  (1.5,0) circle (1pt)
                  (2,0) circle (1pt)
                  (0,1.3) circle (1pt)
                  (-1,-0.5) circle (1pt);
 \draw (0,1.43) node {\tiny $x_2$};\draw (-0.78,-0.5) node {\tiny $x_1$};
 \draw (-2.5,-0.17) node {\tiny $v_1$};\draw (-2,-0.12) node {\tiny $u_1$};
 \draw (-1.5,-0.12) node {\tiny $v_5$};\draw (-0.85,-0.10) node {\tiny $u_2$};
\draw (-0.6,0.11) node {\tiny $v_2$};\draw (0,-0.12) node {\tiny $u_5$};
\draw (0.5,-0.12) node {\tiny $v_4$};\draw (1.1,-0.12) node {\tiny $u_4$};
\draw (1.55,-0.12) node {\tiny $v_3$};\draw (2.11,-0.12) node {\tiny $u_3$};
 \draw (0,1.3)--(0,0); \draw (0,1.3)-- (-0.5,0); \draw (0,1.3)--(-1,0);
 \draw (0,1.3)--(-1.5,0); \draw (0,1.3)--(-2,0); \draw (0,1.3)--(-2.5,0);
\draw (0,1.3)-- (0.5,0); \draw (0,1.3)--(1,0);
\draw (0,1.3)--(1.5,0); \draw (0,1.3)--(2,0);
\draw (-2.5,0)-- (2,0);\draw (-1,0)-- (-1,-0.5);
\draw  (-2.5,0)..controls+(0.5,-0.4)and+(-0.5,-0.4)..(-1,0);
\draw  (-2.5,0)..controls+(0.5,-0.9)and+(-0.5,-0.9)..(2,0);
\draw  (-0.5,0)..controls+(0.5,-0.4)and+(-0.5,-0.4)..(1,0);
\draw  (-0.5,0)..controls+(0.5,-0.6)and+(-0.5,-0.6)..(2,0);
\begin{scope}[xshift=5cm]
\filldraw [black] (0,0) circle (1pt)
                  (1.2,0) circle (1pt)
                (-1.2,0) circle (1pt)
                   (0,1.2) circle (1pt)
                  (0,-1.2) circle (1pt)
                  (1.2,1.2) circle (1pt)
                  (1.2,-1.2) circle (1pt)
                   (-1.2,1.2) circle (1pt)
                   (-1.2,-1.2) circle (1pt)
                    (0.8,-0.4) circle (1pt)
                     (-0.8,0.4) circle (1pt)
                    (1.6,0) circle (1pt);
\draw  (-1.2,-1.2)..controls+(-0.5,0.4)and+(-2.2,1.2)..(0,1.2);
\draw  (-1.2,1.2)--(0,1.2)-- (1.2,1.2)--(1.2,0)--(1.2,-1.2)-- (0,-1.2)--(-1.2,-1.2)--(-1.2,0)-- (-1.2,1.2) ;
\draw  (-1.3,1.05) node { $v_1$};
\draw   (0,1.31)node { $u_5$};
\draw   (1.2,1.31)node { $v_3$};
\draw  (1.09,0.11)node { $u_1$};
\draw   (1.3,-1.32)node { $v_4$};
\draw  (0,-1.32) node { $u_3$};
\draw  (-1.3,-1.3)node { $v_5$};
\draw   (-1.38,0)node { $u_4$};
\draw  (0.83,-0.52)node { $v_2$};
\draw  (-0.83,0.55)node { $x_2$};
\draw  (0.36,0.1)node { $x_1$};
\draw  (1.79,0)node { $u_2$};
\draw  (0,0)--(0,1.2);\draw  (0,0)--(-1.2,1.2);\draw  (0,0)--(1.2,1.2);
\draw  (0,0)--(1.2,0);\draw  (0,0)--(1.2,-1.2);\draw  (0,0)--(0,-1.2);
\draw  (0,0)--(-1.2,-1.2);\draw  (0,0)--(-1.2,0);\draw  (0,0)--(-1.2,1.2);
\draw  (0,0)--(0.8,-0.4)--(1.2,0);
\draw  (1.2,1.2)--(1.6,0)--(1.2,-1.2);
\end{scope}
\end{tikzpicture}
\caption{ A planar decomposition $K_{2,5,5}$}
\label{figure 9}
\end{center}
\end{figure}
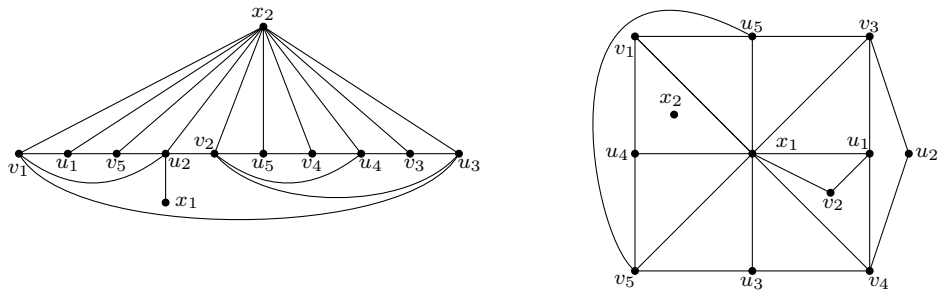

\begin{figure}[H]
\begin{center}
\begin{tikzpicture}
[scale=1.3]
\tikzstyle{every node}=[font=\tiny,scale=1]
\tiny
[inner sep=0pt]
\filldraw [black] (0.5,0) circle (1pt)
                  (1.9,0) circle (1pt)
                (1.2,-0.16) circle (1pt)
                   (1.2,0.06) circle (1pt)
                   (-0.5,0) circle (1pt)
                  (-1.2,0.16) circle (1pt)
                  (-1.9,0) circle (1pt)
                  (-1.2,-0.16) circle (1pt)
                   (-1.2,-0.5) circle (1pt)
                   (0,0.5) circle (1pt)
                  (0.06,1.2) circle (1pt)
                   (0,1.9) circle (1pt)
                    (-0.16, 1.2) circle (1pt)
                     (0.5, 1.2) circle (1pt)
                  (0,-0.5) circle (1pt)
                  (0.16,-1.2) circle (1pt)
                   (0,-1.9) circle (1pt)
                    (-0.26,-1.2) circle (1pt);
\draw  (0,-0.5)--(-0.26,-1.2)-- (0,-1.9) ;\draw  (0,-0.5)--(0.16,-1.2)-- (0,-1.9) ;
\draw(0,0.5)--(-0.16,1.2)-- (0,1.9) ;\draw(0,0.5)--(0.06,1.2)-- (0,1.9) ;
\draw(0,0.5)--(0.5,1.2)-- (0,1.9) ;
\draw  (-0.5,0)-- (-1.2,0.16)--  (-1.9,0) ;\draw  (-0.5,0)-- (-1.2,-0.16)--  (-1.9,0) ;\draw  (-0.5,0)-- (-1.2,-0.5)--  (-1.9,0) ;
\draw   (0.5,0)-- (1.2,0.06)--  (1.9,0);\draw   (0.5,0)-- (1.2,-0.16)--  (1.9,0);
\draw  (0,-1.9) --  (-1.9,0);\draw (0,-1.9)-- (1.9,0)-- (0,1.9);
\draw  (0,-0.5)--(0.5,0);\draw(-0.5,0)--(0,0.5);\draw(0,0.5)-- (0.5,0);
\draw[-]  (0,-1.9)..controls+(0.5,0.5)and+(0.2,-0.5)..  (0.5,0);
\draw[-]  (1.9,0)..controls+(-0.4,0.5)and+(0.5,0.1)..  (0,0.5);
\draw  (0.07,0.29) node { $v_3$};
\draw   (0.61,1.08)node { $u_6$};
\draw   (0.22,1.2)node { $u_5$};
\draw   (0.16,1.9)node { $v_1$};
\draw   (-0.28,1.3)node { $u_9$};
\draw  (0.15,-0.58) node { $v_2$};
\draw   (0.27,-1.1)node { $u_8$};
\draw   (-0.1,-1.2)node { $u_7$};
\draw   (0.16,-1.92)node { $v_4$};
\draw   (-0.45,0.24) node { $u_4$};
\draw  (-1.31,0.24) node { $v_7$};
\draw  (-1.2,-0.04) node { $v_5$};
\draw   (-1.95,-0.15)node { $u_3$};
\draw  (-1.2,-0.38) node { $v_9$};
\draw   (0.58,0.11) node { $u_1$};
\draw  (1.2,0.18) node { $v_6$};
\draw  (1.2,-0.30) node { $v_8$};
\draw   (2,0.11)node { $u_2$};
\filldraw [black] (-1,1) circle (1pt);
\draw  (-1.15,1.11) node { $x_1$};
\draw   (-1,1)  --  (-0.5,0);\draw   (-1,1)  --  (0,0.5);
\draw   (-1,1)  --  (-1.2,0.16);\draw   (-1,1)  --  (-0.16,1.2);
\draw   (-1,1)  --  (-1.9,0);\draw   (-1,1)  --  (0,1.9);
\draw[-]   (-1,1)..controls+(-0.5,-0.1)and+(0.1,0.3)..  (-2.2,0);
\draw[-]   (0,-1.9)..controls+(-0.5,0.1)and+(-0.01, -0.8)..  (-2.2,0);
\draw[-]   (-1,1)..controls+(0.1,0.1)and+(-0.5,-0.1)..  (0,2.2);
\draw[-]   (1.9,0)..controls+(-0.1,0.1)and+(0.8,-0.01)..  (0,2.2);
\filldraw [black] (-0.6,-0.6) circle (1pt);
\draw   (-0.72,-0.48) node { $x_2$};
\draw   (-0.6,-0.6)  --  (-0.5,0);\draw   (-0.6,-0.6)  --  (0,-0.5);
\draw   (-0.6,-0.6)  --  (0.5,0);\draw   (-0.6,-0.6)  --  (0,0.5);
\draw    (-0.6,-0.6)  --  (-1.2,-0.5);\draw    (-0.6,-0.6)  --  (-0.26,-1.2);
\draw[-]    (-0.6,-0.6)..controls+(0.1,-0.4)and+(-0.2,0.4)..  (0,-1.9);
\draw[-]   (-0.6,-0.6)..controls+(-0.1,-0.1)and+(1.2,-1.1)..  (-1.9,0);
\begin{scope}[xshift=5cm]
\filldraw [black] (0.5,0) circle (1pt)
                  (1.9,0) circle (1pt)
                  (1.2,-0.56) circle (1pt)
                  (1.2,0.26) circle (1pt)
                  (1.2,-0.15) circle (1pt)
                   (-0.5,0) circle (1pt)
                  (-1.2,0.2) circle (1pt)
                  (-1.9,0) circle (1pt)
                  (-1.2,-0.26) circle (1pt)
                   (0,0.5) circle (1pt)
                  (0.16,1.2) circle (1pt)
                   (0,1.9) circle (1pt)
                    (-0.22, 1.2) circle (1pt)
                  (0,-0.5) circle (1pt)
                  (0.41,-1.2) circle (1pt)
                   (0,-1.9) circle (1pt)
                    (-0.46,-1.2) circle (1pt)
                     (0,-1.2) circle (1pt);
\draw  (0,-0.5)--(-0.46,-1.2)-- (0,-1.9) ;\draw  (0,-0.5)--(0.41,-1.2)-- (0,-1.9) ; \draw  (0,-0.5)--(0,-1.2)-- (0,-1.9) ;
\draw(0,0.5)--(-0.22,1.2)-- (0,1.9) ;\draw(0,0.5)--(0.16,1.2)-- (0,1.9) ;
\draw  (-0.5,0)-- (-1.2,0.2)--  (-1.9,0) ;\draw  (-0.5,0)-- (-1.2,-0.26)--  (-1.9,0) ;
\draw   (0.5,0)-- (1.2,0.26)--  (1.9,0);\draw   (0.5,0)-- (1.2,-0.56)--  (1.9,0);\draw   (0.5,0)-- (1.2,-0.15)--  (1.9,0);
\draw  (0,1.9) --  (-1.9,0)-- (0,-1.9);\draw(0,1.9) -- (1.9,0);
\draw  (0.5,0)--(0,-0.5) --  (-0.5,0)--(0,0.5);
\draw[-]  (-1.9,0)..controls+(1.2,-1.2)and+(-0.5,-0.2)..  (0,-0.5);
\draw[-]  (0,1.9)..controls+(-1.2, -1.2)and+(-0.2,0.5)..  (-0.5 ,0);
\draw  (0.17,0.62) node { $v_7$};
\draw   (0,1.18)node { $u_2$};
\draw   (-0.36,1.2)node { $u_1$};
\draw   (0.18,1.95)node { $v_5$};
\draw  (-0.21,-0.47) node { $v_6$};
\draw   (0.57,-1.2)node { $u_9$};
\draw   (-0.58,-1.1)node { $u_3$};
\draw   (0.18,-1.92)node { $v_8$};
\draw   (-0.16,-1.2)node { $u_4$};
\draw   (-0.47,0.23) node { $u_8$};
\draw  (-1.1,0.30) node { $v_1$};
\draw  (-1.05,-0.35) node { $v_3$};
\draw   (-2,0.15)node { $u_7$};
\draw   (0.48,-0.24) node { $u_5$};
\draw  (1.22,0.37) node { $v_9$};
\draw  (1.2,-0.03) node { $v_4$};
\draw  (1.2,-0.4) node { $v_2$};
\draw   (2,0.12)node { $u_6$};
\filldraw [black] (1,-1) circle (1pt);
\draw  (1.10,-1.15) node { $x_1$};
\draw   (1,-1)  --  (0.5,0);
\draw   (1,-1)  --  (1.9,0);\draw   (1,-1)  --  (0,-1.9);
\draw   (1,-1)  --  (1.2,-0.56);\draw   (1,-1)  --  (0,-0.5);
\draw[-]   (1,-1)..controls+(0.5,0.1)and+(-0.1,-0.4)..  (2.3,0);
\draw[-]   (0,1.9)..controls+(0.5,-0.1)and+(0.01, 0.8)..  (2.3,0);
\draw[-]   (1,-1)..controls+(-0.1,-0.1)and+(0.5,0.1)..  (0,-2.2);
\draw[-]   (-1.9,0)..controls+(0.1,-0.1)and+(-0.8,0.01)..  (0,-2.2);
\filldraw [black] (0.6,0.6) circle (1pt);
\draw  (0.72,0.75) node { $x_2$};
\draw   (0.6,0.6)  --  (-0.5,0);\draw   (0.6,0.6)  --  (0,-0.5);
\draw   (0.6,0.6)  --  (0.5,0);\draw   (0.6,0.6)  --  (0,0.5);
\draw   (0.6,0.6)  --  (0.16,1.2);
\draw[-]    (0.6,0.6)..controls+(-0.1,0.4)and+(0.3,-0.5)..  (0,1.9);
\draw[-]    (0.6,0.6)..controls+(0.3,0.2)and+(-0.5,0.4).. (1.9,0);
\end{scope}
\end{tikzpicture}
\end{center}
\end{figure}
\vskip-1cm

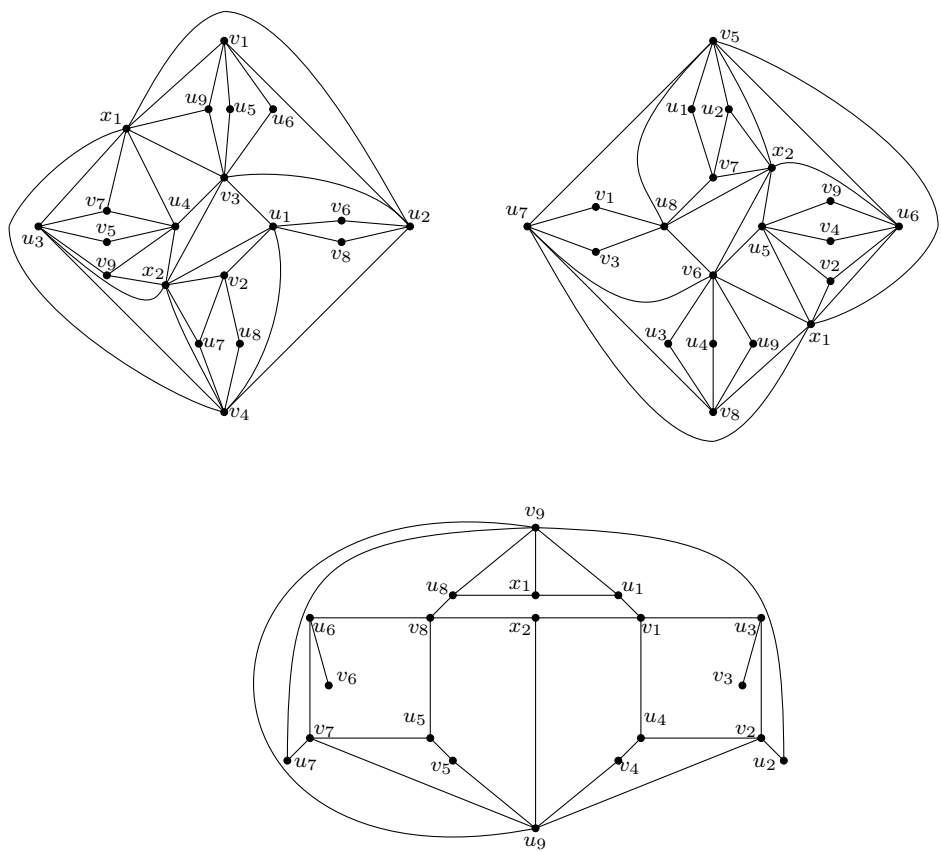
\begin{figure}[H]
\begin{center}
\begin{tikzpicture}
[inner sep=0pt]
\filldraw [black]     (-1,2) circle (1.3pt);
\draw (-1,2.2) node {\tiny $v_{9}$};
\filldraw [black]     (-1,-2) circle (1.3pt);
\draw (-1,-2.2) node {\tiny $u_{9}$};
\draw[-] (-1,2)..controls+(-5,0.8)and+(-5,-1).. (-1,-2);
\draw[-](-1,2)--(-1,1.1);\draw[-](-1,2)--(-2.1,1.1);\draw[-](-1,2)--(0.1,1.1);
\draw[-](2.3,-1.1)..controls+(-0.01,2.5)and+(3,-0.1)..(-1,2);
\draw[-](-4.3,-1.1)..controls+(0.01,2.5)and+(-3,-0.1)..(-1,2);
\draw[-](-1,-2)--(-1,0.8);\draw[-](-1,-2)--(-2.1,-1.1);\draw[-](-1,-2)--(0.1,-1.1);
\draw[-](-1,-2)--(-4,-0.8);\draw[-](-1,-2)--(2,-0.8);
\filldraw [black]     (-1,1.1) circle (1.3pt);
\draw (-1.2,1.23) node {\tiny $x_1$};
\filldraw [black]     (-1,0.8) circle (1.3pt);
\draw (-1.2,0.65) node {\tiny $x_2$};
\draw[-](0.1,1.1)--(-2.1,1.1);\draw[-](0.4,0.8)--(-2.4,0.8);
\filldraw [black]     (-4,0.8) circle (1.3pt);
\draw (-3.8,0.65) node {\tiny $u_{6}$};
\filldraw [black]     (-4,-0.8) circle (1.3pt);
\draw (-3.8,-0.7) node {\tiny $v_{7}$};
\filldraw [black]     (-2.4,0.8) circle (1.3pt);
\draw (-2.55,0.645) node {\tiny $v_{8}$};
\filldraw [black]     (-2.4,-0.8) circle (1.3pt);
\draw (-2.6,-0.55) node {\tiny $u_{5}$};
\draw[-](-4,0.8)--(-4,-0.8)--(-2.4,-0.8)--(-2.4,0.8)--(-4,0.8);
\filldraw [black]     (-3.75,-0.1) circle (1.3pt);\draw[-](-3.75,-0.1)--(-4,0.8);
\draw (-3.5,0.0) node {\tiny $v_{6}$};
\filldraw [black]     (-2.1,1.1) circle (1.3pt);\draw[-](-2.1,1.1)--(-2.4,0.8);
\draw (-2.3,1.2) node {\tiny $u_{8}$};
\filldraw [black]     (-4.3,-1.1) circle (1.3pt);\draw[-](-4,-0.8)--(-4.3,-1.1);
\draw (-4.05,-1.15) node {\tiny $u_{7}$};
\filldraw [black]     (-2.1,-1.1) circle (1.3pt);\draw[-](-2.1,-1.1) --(-2.4,-0.8);
\draw (-2.25,-1.2) node {\tiny $v_{5}$};
\filldraw [black]     (2,0.8) circle (1.3pt);
\draw (1.8,0.65) node {\tiny $u_{3}$};
\filldraw [black]     (2,-0.8) circle (1.3pt);
\draw (1.8,-0.7) node {\tiny $v_{2}$};
\filldraw [black]     (0.4,0.8) circle (1.3pt);
\draw (0.55,0.645) node {\tiny $v_{1}$};
\filldraw [black]     (0.4,-0.8) circle (1.3pt);
\draw (0.6,-0.55) node {\tiny $u_{4}$};
\draw[-](2,0.8)--(2,-0.8)--(0.4,-0.8)--(0.4,0.8)--(2,0.8);
\filldraw [black]     (1.75,-0.1) circle (1.3pt);\draw[-](1.75,-0.1)--(2,0.8);
\draw (1.5,0.0) node {\tiny $v_{3}$};
\filldraw [black]     (0.1,1.1) circle (1.3pt);\draw[-](0.1,1.1)--(0.4,0.8);
\draw (0.3,1.2) node {\tiny $u_{1}$};
\filldraw [black]     (2.3,-1.1) circle (1.3pt);\draw[-](2,-0.8)--(2.3,-1.1);
\draw (2.05,-1.15) node {\tiny $u_{2}$};
\filldraw [black]     (0.1,-1.1) circle (1.3pt);\draw[-](0.1,-1.1) --(0.4,-0.8);
\draw (0.25,-1.2) node {\tiny $v_{4}$};
\end{tikzpicture}
\caption{ A planar decomposition $K_{2,9,9}$}
\label{figure 10}
\end{center}
\end{figure}

\vskip 1cm
\begin{figure}[H]
\begin{center}
\begin{tikzpicture}
[scale=0.8]
\tikzstyle{every node}=[font=\tiny,scale=0.8]
\tiny
[inner sep=0pt]
\filldraw [black] (0.5,0) circle (1.3pt);\draw  (0.5,0.2)  node { $u_{1}$};
\filldraw [black](3.3,0) circle (1.3pt);\draw  (3.3,0.24)  node { $u_{2}$};
\filldraw [black](1.9,0.9) circle (1.3pt);\draw  (1.9,0.69)  node { $v_{6}$};
\filldraw [black](1.9,0.3) circle (1.3pt);\draw  (1.9,0.46)  node { $v_{10}$};
\filldraw [black] (1.9,0) circle (1.3pt);\draw  (2.2,0)  node { $u_{13}$};
\filldraw [black] (1.9,-0.3) circle (1.3pt);\draw  (1.9,-0.444)  node { $v_{12}$};
\filldraw [black](1.9,-0.9) circle (1.3pt);\draw  (1.9,-0.73)  node { $v_{8}$};
\filldraw [black] (-0.5,0) circle (1.3pt);\draw  (-0.5,0.23)  node { $u_{4}$};
\filldraw [black](-1.9,0.3) circle (1.3pt);\draw  (-1.9,0.1)  node { $v_{9}$};
\filldraw [black](-3.3,0) circle (1.3pt);\draw  (-3.34,0.2)  node { $u_{3}$};
\filldraw [black] (-1.9,-0.3) circle (1.3pt);\draw  (-1.9,-0.5)  node { $v_{11}$};
\filldraw [black] (-1.9,0.9) circle (1.3pt);\draw  (-1.9,0.7)  node { $v_{5}$};
 \filldraw [black](-1.9,-0.9) circle (1.3pt);\draw  (-1.7,-1)  node { $v_{7}$};
\filldraw [black](0,0.5) circle (1.3pt);\draw  (0.24,0.46)  node { $v_{3}$};
\filldraw [black](0,3.3) circle (1.3pt);\draw  (0.24,3.3)  node { $v_{1}$};
\filldraw [black] (-0.9,1.9) circle (1.3pt);\draw  (-0.65,1.9)  node { $u_{5}$};
\filldraw [black] (-0.3,1.9) circle (1.3pt);\draw  (-0.05,1.9)  node { $u_{6}$};
 \filldraw [black] (0.3,1.9) circle (1.3pt);\draw  (0.55,1.9)  node { $u_{9}$};
\filldraw [black](0.9,1.9) circle (1.3pt);\draw  (1.2,1.86)  node { $u_{10}$};
\filldraw [black] (0,-0.5) circle (1.3pt);\draw  (-0.26,-0.46)  node { $v_{2}$};
\filldraw [black] (0,-3.3) circle (1.3pt);\draw  (0.24, -3.35)  node { $v_{4}$};
\filldraw [black](-1.2, -1.9) circle (1.3pt);\draw  (-1.27,-1.71)  node { $u_{8}$};
\filldraw [black](-0.65, -1.9) circle (1.3pt);\draw  (-0.89,-1.9)  node { $u_{11}$};
\filldraw [black](-0.1,-1.9) circle (1.3pt);\draw  (-0.35,-1.9)  node { $u_{12}$};
\filldraw [black](0.25, -1.9) circle (1.3pt);\draw  (0.27,-2.07)  node { $v_{13}$};
\filldraw [black](0.6, -1.9) circle (1.3pt);\draw  (0.83,-1.94)  node { $u_{7}$};
\draw (-3.3,0)--(0,3.3)--(3.3,0)--(0,-3.3);
\draw (0,0.5)-- (-0.5,0)--(0,-0.5);\draw (0.5,0)--(0,0.5);
\draw  (0,0.5)..controls+(1.7,1.3)and+(-0.5,0.5)..(3.3,0);
\draw  (-0.5,0)..controls+(-1.7,1.3)and+(-0.5,-0.6)..(0, 3.3);
\draw (0, 3.3)--(0.3,1.9)--(0,0.5);\draw (0, 3.3)--(0.9,1.9)--(0,0.5);
\draw (0, 3.3)--(-0.3,1.9)--(0,0.5);\draw (0, 3.3)--(-0.9,1.9)--(0,0.5);
\draw (0, -3.3)--(-1.2,-1.9)--(0,-0.5);\draw (0, -3.3)--(-0.65,-1.9)--(0,-0.5);
\draw (0.3,-1.9)--(-0.1,-1.9)--(0.6,-1.9);
\draw (0, -3.3)--(-0.1,-1.9)--(0,-0.5);\draw (0, -3.3)--(0.6,-1.9)--(0,-0.5);
\draw (-3.3 ,0)--(-1.9,0.3)--(-0.5,0);\draw (-3.3 ,0 )--(-1.9 ,0.9)--(-0.5, 0);
\draw (-3.3 ,0)--(-1.9,-0.3)--(-0.5,0);\draw (-3.3 ,0 )--(-1.9 ,-0.9)--(-0.5, 0);
\draw (3.3 ,0)--(1.9,0.3)--(0.5,0);\draw (3.3 ,0 )--(1.9 ,0.9)--(0.5, 0);
\draw (3.3 ,0)--(1.9,-0.3)--(0.5,0);\draw (3.3 ,0 )--(1.9 ,-0.9)--(0.5, 0);
\draw (1.9,-0.3)--(1.9,0)--(1.9,0.3);
\filldraw [black] (1.2,-1.1) circle (1.3pt);
\draw  (1. 28,-0.9)  node { $x_2$};
\draw  (1.2,-1.1)--(1.9,-0.9);\draw  (1.2,-1.1)--(0 ,-0.5);
\draw  (1.2,-1.1)--(0.5, 0 );\draw  (1.2,-1.1)--(0.6, -1.9 );
\draw  (1.2,-1.1)..controls+(-0.3,0.3)and+(0.1,-0.5)..(0,0.5);
\draw  (1.2,-1.1)..controls+(-0.4,0.3)and+(0.5,-0.1)..(-0.5,0);
\draw  (1.2,-1.1)..controls+(0.39,-0.3)and+(0.01,0.01)..(0,-3.3);
\draw  (1.2,-1.1)..controls+(0.9,-0.2)and+(-0.3,-0.2)..(3.3,0);
\filldraw [black] (-1.8,-1.7) circle (1.3pt);
\draw  (-1.95,-1.8)  node { $x_1$};
\draw  (-3.3,0)--(-1.8,-1.7)--(0,-3.3);
\draw  (-1.9,-0.9)--(-1.8,-1.7)--(-1.2,-1.9);
\draw  (-0.5,0)--(-1.8,-1.7)--(0,-0.5);
\draw   (-1.8,-1.7)..controls+(-0.5,0.2)and+(0.4,-1)..(-3.7,0);
\draw  (-3.7,0)..controls+(-0.4,0.9)and+(-1,-1)..(0,3.3);
\draw   (-1.8,-1.7)..controls+(0.5,-0.5)and+(-0.4,0.01)..(0,-3.7);
\draw (0,-3.7)..controls+(1, 0.01)and+(-1,-1.2)..(3.3, 0);
\begin{scope}[xshift=8.4cm]
\filldraw [black] (-0.5,0) circle (1.3pt);\draw  (-0.52,0.24)  node { $u_{8}$};
\filldraw [black](-3.3,0) circle (1.3pt);\draw  (-3.3,0.2)  node { $u_{7}$};
\filldraw [black](-1.9,0.9) circle (1.3pt);\draw  (-1.66,0.9)  node { $v_{3}$};
\filldraw [black](-1.9,0.45) circle (1.3pt);\draw  (-1.6,0.45)  node { $u_{13}$};
\filldraw [black] (-1.9,0) circle (1.3pt);\draw  (-1.66,0.12)  node { $v_{11}$};
\filldraw [black] (-1.9,-0.45) circle (1.3pt);\draw  (-1.9,-0.3)  node { $v_{1}$};
\filldraw [black](-1.9,-0.9) circle (1.3pt);\draw  (-1.9,-0.73)  node { $v_{9}$};
\filldraw [black] (0.5,0) circle (1.3pt);\draw  (0.5,-0.29)  node { $u_{5}$};
\filldraw [black](1.9,0.3) circle (1.3pt);\draw  (1.9,0.1)  node { $v_{2}$};
\filldraw [black](3.3,0) circle (1.3pt);\draw  (3.39,-0.2)  node { $u_{6}$};
\filldraw [black] (1.9,-0.3) circle (1.3pt);\draw  (1.9,-0.5)  node { $v_{4}$};
\filldraw [black] (1.9,0.9) circle (1.3pt);\draw  (1.9,0.65)  node { $v_{12}$};
 \filldraw [black](1.9,-0.9) circle (1.3pt);\draw  (1.65,-1)  node { $v_{10}$};
\filldraw [black](0,-0.5) circle (1.3pt);\draw  (0.25,-0.52)  node { $v_{6}$};
\filldraw [black](0,-3.3) circle (1.3pt);\draw  (0.25,-3.33)  node { $v_{8}$};
\filldraw [black] (-0.9,-1.9) circle (1.3pt);\draw  (-0.65,-1.9)  node { $u_{3}$};
\filldraw [black] (-0.3,-1.9) circle (1.3pt);\draw  (-0.05,-1.9)  node { $u_{4}$};
 \filldraw [black] (0.3,-1.9) circle (1.3pt);\draw  (0.58,-1.9)  node { $u_{11}$};
\filldraw [black](0.9,-1.9) circle (1.3pt);\draw  (1.02,-1.71)  node { $u_{12}$};
\filldraw [black] (0,0.5) circle (1.3pt);\draw  (-0.22,0.5)  node { $v_{7}$};
\filldraw [black] (0,3.3) circle (1.3pt);\draw  (0.2, 3.4)  node { $v_{5}$};
\filldraw [black](-0.9, 1.9) circle (1.3pt);\draw  (-1.03,1.75)  node { $u_{1}$};
\filldraw [black](0.35, 1.9) circle (1.3pt);\draw  (-0.53,1.75)  node { $v_{13}$};
\filldraw [black](-0.2,1.9) circle (1.3pt);\draw  (-0.35,2.07)  node { $u_{2}$};
\filldraw [black](-0.55, 1.9) circle (1.3pt);\draw  (0.14,1.9)  node { $u_{9}$};
\filldraw [black](0.9, 1.9) circle (1.3pt);\draw  (0.63,1.9)  node { $u_{10}$};
\draw (-3.3,0)--(0,-3.3);\draw (3.3,0)--(0,3.3);
\draw (0,-0.5)-- (-0.5,0)--(0,0.5);\draw (0.5,0)--(0,-0.5);
\draw  (0,-0.5)..controls+(-1.7,-1.3)and+(0.6,-0.5)..(-3.3,0);
\draw  (-0.5,0)..controls+(-1.7,1.3)and+(-0.8,-0.6)..(0, 3.3);
\draw (0, -3.3)--(0.3,-1.9)--(0,-0.5);\draw (0, -3.3)--(0.9,-1.9)--(0,-0.5);
\draw (0, -3.3)--(-0.3,-1.9)--(0,-0.5);\draw (0, -3.3)--(-0.9,-1.9)--(0,-0.5);
\draw (0, 3.3)--(-0.9,1.9)--(0,0.5);\draw (0,3.3)--(0.35,1.9)--(0,0.5);
\draw (-0.2,1.9)--(-0.55,1.9)--(-0.9,1.9);
\draw (0, 3.3)--(-0.2,1.9)--(0,0.5);\draw (0, 3.3)--(0.9,1.9)--(0,0.5);
\draw (3.3 ,0)--(1.9,0.3)--(0.5,0);\draw (3.3 ,0 )--(1.9 ,0.9)--(0.5, 0);
\draw (3.3 ,0)--(1.9,-0.3)--(0.5,0);\draw (3.3 ,0 )--(1.9 ,-0.9)--(0.5, 0);
\draw (-3.3 ,0)--(-1.9,0)--(-0.5,0);\draw (-3.3 ,0 )--(-1.9 ,0.9)--(-0.5, 0);
\draw (-3.3 ,0)--(-1.9,-0.45)--(-0.5,0);\draw (-3.3 ,0 )--(-1.9 ,-0.9)--(-0.5, 0);
\draw(-1.9,0)--(-1.9,0.45);
\filldraw [black] (1.2,1.1) circle (1.3pt);
\draw  (1.31,0.9)  node { $x_2$};
\draw  (1.2,1.1)--(1.9,0.9);\draw  (1.2,1.1)--(0 ,0.5);
\draw  (1.2,1.1)--(0.5, 0 );\draw  (1.2,1.1)--(0.9, 1.9 );
\draw  (1.2,1.1)..controls+(-0.3,-0.33)and+(0.1,0.4)..(0,-0.5);
\draw  (1.2,1.1)..controls+(-0.2,-0.2)and+(0.3,0.1)..(-0.5,0);
\draw  (1.2,1.1)..controls+(0.3,0.5)and+(0.09,-0.12)..(0,3.3);
\draw  (1.2,1.1)..controls+(0.5,0.4)and+(-0.1,0.05)..(3.3,0);
\filldraw [black] (1.8,-1.7) circle (1.3pt);
\draw  (1.95,-1.85)  node { $x_1$};
\draw  (3.3,0)--(1.8,-1.7)--(0,-3.3);
\draw  (1.9,-0.9)--(1.8,-1.7)--(0.9,-1.9);
\draw  (0.5,0)--(1.8,-1.7)--(0,-0.5);
\draw   (1.8,-1.7)..controls+(0.5,0.2)and+(-0.01,-0.9)..(3.7,0);
\draw  (3.7,0)..controls+(0.01,1)and+(1,-0.8)..(0,3.3);
\draw   (1.8,-1.7)..controls+(-0.5,-0.5)and+(0.4,0.01)..(0,-3.7);
\draw (0,-3.7)..controls+(-1, 0.01)and+(1,-1.3)..(-3.3, 0);
\draw   (1.8,-1.7)..controls+(-2,-4)and+(2.3,-3.3)..(-3.55, 0);
\draw (-3.55, 0)..controls+(-0.2, 0.5)and+(-0.4,-0.1)..(-1.9, 0.9);
\end{scope}
\end{tikzpicture}
\end{center}
\end{figure}

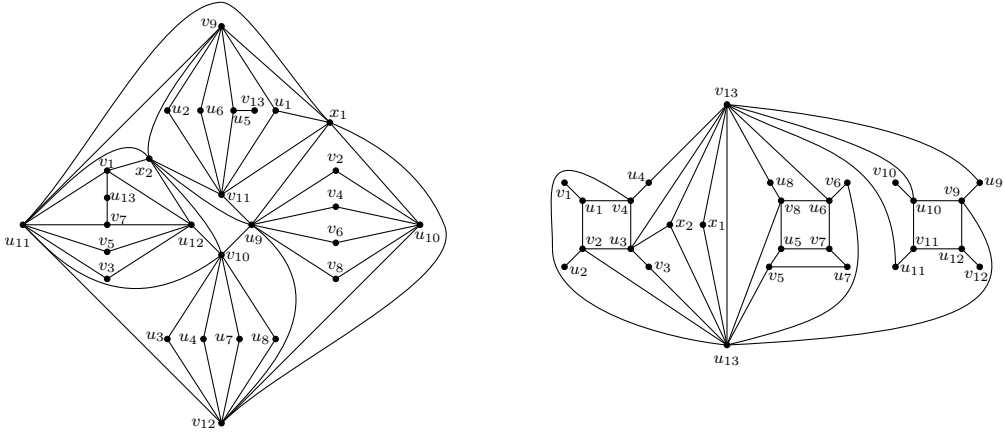
\begin{figure}[htp]
\begin{center}
\begin{tikzpicture}
[scale=0.8]
\tikzstyle{every node}=[font=\tiny,scale=0.8]
\tiny
[inner sep=0pt]
\filldraw [black] (-0.5,0) circle (1.3pt);\draw  (-0.5,-0.29)  node { $u_{12}$};
\filldraw [black](-3.3,0) circle (1.3pt);\draw  (-3.38,-0.29)  node { $u_{11}$};
\filldraw [black](-1.9,-0.9) circle (1.3pt);\draw  (-1.9,-0.71)  node { $v_{3}$};
\filldraw [black](-1.9,-0.45) circle (1.3pt);\draw  (-1.9,-0.3)  node { $v_{5}$};
\filldraw [black] (-1.9,0) circle (1.3pt);\draw  (-1.63,0.45)  node { $u_{13}$};
\filldraw [black] (-1.9,0.45) circle (1.3pt);\draw  (-1.7,0.1)  node { $v_{7}$};
\filldraw [black](-1.9,0.9) circle (1.3pt);\draw  (-1.9,1.02)  node { $v_{1}$};
\filldraw [black] (0.5,0) circle (1.3pt);\draw  (0.54,-0.24)  node { $u_{9}$};
\filldraw [black](1.9,-0.3) circle (1.3pt);\draw  (1.9,-0.1)  node { $v_{6}$};
\filldraw [black](3.3,0) circle (1.3pt);\draw  (3.43,-0.2)  node { $u_{10}$};
\filldraw [black] (1.9,0.3) circle (1.3pt);\draw  (1.9,0.5)  node { $v_{4}$};
\filldraw [black] (1.9,-0.9) circle (1.3pt);\draw  (1.9,-0.7)  node { $v_{8}$};
 \filldraw [black](1.9,0.9) circle (1.3pt);\draw  (1.9,1.1)  node { $v_{2}$};
\filldraw [black](0,-0.5) circle (1.3pt);\draw  (0.28,-0.53)  node { $v_{10}$};
\filldraw [black](0,-3.3) circle (1.3pt);\draw  (-0.29,-3.3)  node { $v_{12}$};
\filldraw [black] (0.9,-1.9) circle (1.3pt);\draw  (0.65,-1.9)  node { $u_{8}$};
\filldraw [black] (0.3,-1.9) circle (1.3pt);\draw  (0.05,-1.9)  node { $u_{7}$};
 \filldraw [black] (-0.3,-1.9) circle (1.3pt);\draw  (-0.55,-1.9)  node { $u_{4}$};
\filldraw [black](-0.9,-1.9) circle (1.3pt);\draw  (-1.1,-1.86)  node { $u_{3}$};
\filldraw [black] (0,0.5) circle (1.3pt);\draw  (0.3,0.46)  node { $v_{11}$};
\filldraw [black] (0,3.3) circle (1.3pt);\draw  (-0.2, 3.35)  node { $v_{9}$};
\filldraw [black](0.9, 1.9) circle (1.3pt);\draw  (1.03,2.01)  node { $u_{1}$};
\filldraw [black](-0.35, 1.9) circle (1.3pt);\draw  (0.53,2.06)  node { $v_{13}$};
\filldraw [black](0.2,1.9) circle (1.3pt);\draw  (0.35,1.74)  node { $u_{5}$};
\filldraw [black](0.55, 1.9) circle (1.3pt);\draw  (-0.1,1.9)  node { $u_{6}$};
\filldraw [black](-0.9, 1.9) circle (1.3pt);\draw  (-0.67,1.9)  node { $u_{2}$};
\draw (3.3,0)--(0,-3.3)--(-3.3,0)--(0,3.3);
\draw (0,-0.5)-- (0.5,0)--(0,0.5);\draw (-0.5,0)--(0,-0.5);
\draw  (0,-0.5)..controls+(-1.7,-1.3)and+(0.5,-0.5)..(-3.3,0);
\draw  (0.5,0)..controls+(1.7,-1.3)and+(0.5,0.6)..(0, -3.3);
\draw (0, -3.3)--(-0.3,-1.9)--(0,-0.5);\draw (0, -3.3)--(-0.9,-1.9)--(0,-0.5);
\draw (0, -3.3)--(0.3,-1.9)--(0,-0.5);\draw (0, -3.3)--(0.9,-1.9)--(0,-0.5);
\draw (0, 3.3)--(0.9,1.9)--(0,0.5);\draw (0, 3.3)--(-0.35,1.9)--(0,0.5);
\draw (0.2,1.9)--(0.55,1.9);
\draw (0, 3.3)--(0.2,1.9)--(0,0.5);\draw (0, 3.3)--(-0.9,1.9)--(0,0.5);
\draw (3.3 ,0)--(1.9,-0.3)--(0.5,0);\draw (3.3 ,0 )--(1.9 ,-0.9)--(0.5, 0);
\draw (3.3 ,0)--(1.9,0.3)--(0.5,0);\draw (3.3 ,0 )--(1.9 ,0.9)--(0.5, 0);
\draw (-3.3 ,0)--(-1.9,-0.45)--(-0.5,0);\draw (-3.3 ,0 )--(-1.9 ,-0.9)--(-0.5, 0);
\draw (-3.3 ,0)--(-1.9,0)--(-0.5,0);\draw (-3.3 ,0 )--(-1.9 ,0.9)--(-0.5, 0);
\draw (-1.9,0.9)--(-1.9,0.45)--(-1.9,0);
\filldraw [black] (-1.2,1.1) circle (1.3pt);
\draw  (-1.28,0.87)  node { $x_2$};
\draw  (-1.2,1.1)--(-1.9,0.9);\draw  (-1.2,1.1)--(0 ,0.5);
\draw  (-1.2,1.1)--(-0.5, 0 );
\draw  (-1.2,1.1)..controls+(0.25,-0.288)and+(-0.1,0.5)..(0,-0.5);
\draw  (-1.2,1.1)..controls+(0.2,-0.15)and+(-0.5,0.1)..(0.5,0);
\draw  (-1.2,1.1)..controls+(-0.2,0.7)and+(-0.01,-0.01)..(0,3.3);
\draw  (-1.2,1.1)..controls+(-0.5,0.7)and+(0.4,0.3)..(-3.3,0);
\filldraw [black] (1.8,1.7) circle (1.3pt);
\draw  (1.95,1.83)  node { $x_1$};
\draw  (3.3,0)--(1.8,1.7)--(0,3.3);
\draw  (1.8,1.7)--(0.9,1.9);
\draw  (0.5,0)--(1.8,1.7)--(0,0.5);
\draw   (1.8,1.7)..controls+(0.5,-0.2)and+(-0.4,1)..(3.7,0);
\draw  (3.7,0)..controls+(0.4,-0.9)and+(1,1)..(0,-3.3);
\draw   (1.8,1.7)..controls+(-0.5,0.5)and+(0.4,-0.01)..(0,3.7);
\draw (0,3.7)..controls+(-1, -0.01)and+(1,1)..(-3.3, 0);
\begin{scope}[xshift=9.5cm]
\filldraw [black]     (-2.05,0) circle(1.3pt);\draw[-](-2.05,0)--(-2.7,-0.4);
\draw (-1.8,0) node {\tiny $x_2$};
\filldraw [black]     (-1.5,0) circle (1.3pt);
\draw (-1.26,0) node {\tiny $x_1$};
\filldraw [black]     (-3.5,0.4) circle (1.3pt);
\draw (-3.3,0.25) node {\tiny $u_{1}$};
\filldraw [black]     (-3.5,-0.4) circle (1.3pt);
\draw (-3.3,-0.28) node {\tiny $v_{2}$};
\filldraw [black]     (-2.7,0.4) circle (1.3pt);
\draw (-2.87,0.25) node {\tiny $v_{4}$};
\filldraw [black]     (-2.7,-0.4) circle (1.3pt);
\draw (-2.9,-0.28) node {\tiny $u_{3}$};
\draw[-](-3.5,0.4)--(-3.5,-0.4)--(-2.7,-0.4)--(-2.7,0.4)--(-3.5,0.4);
\filldraw [black]     (-3.8,0.7) circle (1.3pt);
\draw[-](-3.8,0.7)--(-3.5,0.4);
\draw (-3.8,0.5) node {\tiny $v_{1}$};
\filldraw [black]     (-2.4,0.7) circle (1.3pt);
\draw[-](-2.4,0.7)--(-2.7,0.4);
\draw (-2.6,0.8) node {\tiny $u_{4}$};
\filldraw [black]     (-3.8,-0.7) circle (1.3pt);
\draw[-](-3.5,-0.4)--(-3.8,-0.7);
\draw (-3.56,-0.77) node {\tiny $u_{2}$};
\filldraw [black]     (-2.4,-0.7) circle (1.3pt);
\draw[-](-2.7,-0.4) --(-2.4,-0.7);
\draw (-2.16,-0.7) node {\tiny $v_{3}$};

\filldraw [black]     (-0.2,0.4) circle (1.3pt);
\draw (0,0.25) node {\tiny $v_{8}$};
\filldraw [black]     (-0.2,-0.4) circle (1.3pt);
\draw (0,-0.28) node {\tiny $u_{5}$};
\filldraw [black]     (0.6,0.4) circle (1.3pt);
\draw (0.42,0.25) node {\tiny $u_{6}$};
\filldraw [black]     (0.6,-0.4) circle (1.3pt);
\draw (0.42,-0.28) node {\tiny $v_{7}$};
\draw[-](-0.2,0.4)--(-0.2,-0.4)--(0.6,-0.4)--(0.6,0.4)--(-0.2,0.4);
\filldraw [black]     (-0.38,0.7) circle (1.3pt);
\draw[-](-0.38,0.7)--(-0.2,0.4);
\draw (-0.13,0.7) node {\tiny $u_{8}$};
\filldraw [black]     (0.9,0.7) circle (1.3pt);
\draw[-](0.9,0.7)--(0.6,0.4);
\draw (0.67,0.7) node {\tiny $v_{6}$};
\filldraw [black]     (-0.4,-0.7) circle (1.3pt);
\draw[-](-0.4,-0.7)--(-0.2,-0.4);
\draw (-0.27,-0.88) node {\tiny $v_{5}$};
\filldraw [black]     (0.9,-0.7) circle (1.3pt);
\draw[-](0.9,-0.7) --(0.6,-0.4);
\draw (0.83,-0.88) node {\tiny $u_{7}$};
\draw[-](0.9,-0.7) --(-0.4,-0.7);
\filldraw [black]     (2,0.4) circle (1.3pt);
\draw (2.26,0.25) node {\tiny $u_{10}$};
\filldraw [black]     (2,-0.4) circle (1.3pt);
\draw (2.24,-0.28) node {\tiny $v_{11}$};
\filldraw [black]     (2.8,0.4) circle (1.3pt);
\draw (2.65,0.55) node {\tiny $v_{9}$};
\filldraw [black]     (2.8,-0.4) circle (1.3pt);
\draw (2.64,-0.55) node {\tiny $u_{12}$};
\draw[-](2,0.4)--(2,-0.4)--(2.8,-0.4)--(2.8,0.4)--(2,0.4);
\filldraw [black]     (1.7,0.7) circle (1.3pt);\draw[-](1.7,0.7)--(2,0.4);
\draw (1.53,0.87) node {\tiny $v_{10}$};
\filldraw [black]     (3.1,0.7) circle (1.3pt);\draw[-](3.1,0.7)--(2.8,0.4);
\draw (3.34,0.7) node {\tiny $u_{9}$};
\filldraw [black]     (1.7,-0.7) circle (1.3pt);\draw[-](2,-0.4)--(1.7,-0.7);
\draw (2,-0.75) node {\tiny $u_{11}$};
\filldraw [black]     (3.1,-0.7) circle (1.3pt);\draw[-](3.1,-0.7) --(2.8,-0.4);
\draw (3.05,-0.85) node {\tiny $v_{12}$};
\filldraw [black]     (-1.1,-2) circle (1.3pt);\draw[-](-1.1,-2)--(-1.1,2);
\draw (-1.1,-2.25) node {\tiny $u_{13}$};
\filldraw [black]     (-1.1,2) circle (1.3pt);
\draw (-1.1,2.2) node {\tiny $v_{13}$};
\draw[-](-1.1,-2)..controls+(2,0.5)and+(0.5,-2)..(0.9,0.7);
\draw[-](-1.1,-2)--(-3.5,-0.4);
\draw[-](-1.1,-2)..controls+(4,0.2)and+(1.3,-1.7)..(2.8,0.4);
\draw[-](-4,-0.4)..controls+(0,1)and+(-1.6,1.1)..(-2.7,0.4);
\draw[-](-1.1,-2)..controls+(-1,0.1)and+(0,-1)..(-4,-0.4);
\draw[-](-1.1,-2)--(-2.4,-0.7);
\draw[-](-1.1,-2)--(-2.05,0);
\draw[-](-1.1,-2)--(-1.5,0);
\draw[-](-1.1,-2)--(-0.4,-0.7);
\draw[-](-1.1,-2)--(-0.2,0.4);
\draw[-](-1.1,2)--(-2.7,-0.4);
\draw[-](-1.1,2)--(-1.5,0);
\draw[-](-1.1,2)-- (-2.4,0.7);
\draw[-](-1.1,2)--(-2.05,0);
\draw[-](-1.1,2)--(-0.38,0.7);
\draw[-](-1.1,2)--(0.6,0.4);
\draw[-](-1.1,2)..controls+(2.5,-0.7)and+(-0.01,0.6)..(2,0.4);
\draw[-](-1.1,2)..controls+(3,-0.3)and+(-0.4,0.5)..(3.1,0.7);
\draw[-](-1.1,2)..controls+(1.2,-0.7)and+(0,2)..(1.7,-0.7);
\end{scope}
\end{tikzpicture}
\caption{A planar decomposition of $K_{2,13,13}$}
\label{figure 11}
\end{center}
\end{figure}

Summarizing Cases 1,2 and 3, the lemma follows.
\end{proof}

\begin{theorem}\label{main2} The thickness of the complete 3-partite graph $K_{2,n,n}$ is
$$\theta(K_{2,n,n})=\big\lceil\frac{n+3}{4}\big\rceil.$$ \end{theorem}

\begin{proof}
When $n=4p,4p+3$, from Lemma \ref{2.2}, the theorem holds.

When $n=4p+1$, from Lemma \ref{2.4}, we have $\theta(K_{2,4p+1,4p+1})\leq p+1$. Since $\theta(K_{4p,4p})=p+1$ and $K_{4p,4p}\subset K_{2,4p+1,4p+1}$, we have
$$p+1=\theta(K_{4p,4p})\leq \theta(K_{2,4p+1,4p+1})\leq p+1.$$
Therefore, $\theta(K_{2,4p+1,4p+1})=p+1$.

When $n=4p+2$, since $K_{4p+3,4p+3}\subset K_{2,4p+2,4p+2}$, from Lemma \ref{2.1}, we have $p+2=\theta(K_{4p+3,4p+3})\leq \theta(K_{2,4p+2,4p+2})$. On the other hand, it is easy to see $\theta(K_{2,4p+2,4p+2})\leq \theta(K_{2,4p+1,4p+1})+1=p+2$, so we have $\theta(K_{2,4p+2,4p+2})=p+2$.

Summarizing the above, the theorem is obtained.\end{proof}

\section{The thickness of $K_{1,1,n,n}$ }

\begin{theorem}\label{main3} The thickness of the complete 4-partite graph $K_{1,1,n,n}$ is
$$\theta(K_{1,1,n,n})=\big\lceil\frac{n+3}{4}\big\rceil.$$ \end{theorem}

\begin{proof} When $n=4p+1$, we can get a planar decomposition for $K_{1,1,4p+1,4p+1}$ from that of $K_{2,4p+1,4p+1}$ as follows.

(1) When $p=0$, $K_{1,1,1,1}$ is a planar graph, $\theta(K_{1,1,1,1})=1$. When $p=1,2$ and $3$, we join the vertex $x_{1}$ to $x_{2}$ in the last planar subgraph in the planar decomposition for $K_{2,5,5}$, $K_{2,9,9}$ and $K_{2,13,13}$ which was shown in Figure \ref{figure 9}, \ref{figure 10} and \ref{figure 11}. Then we get the planar decomposition for $K_{1,1,5,5}$, $K_{1,1,9,9}$ and $K_{1,1,13,13}$ with 2, 3 and 4 planar subgraphs respectively.

(2) When $p\geq 4$,  we join the vertex $x_{1}$ to $x_{2}$ in $\widehat{G}_{p+1}$ in the planar decomposition for $K_{2,4p+1,4p+1}$ which was constructed in Lemma \ref{2.4}. The $\widehat{G}_{p+1}$ is shown in Figure \ref{figure 7} or \ref{figure 8} according to $p$ is even or odd.  Because $x_1$ and $x_2$ lie on the boundary of the same face, we will get a planar graph by adding edge $x_{1}x_{2}$ to $\widehat{G}_{p+1}$. Then a planar decomposition for
$K_{1,1,4p+1,4p+1}$ with $p+1$ planar subgraphs can be obtained.

\noindent Summarizing (1) and (2), we have $K_{1,1,4p+1,4p+1}\leq p+1$.

On the other hand, from Lemma \ref{2.1}, we have $\theta(K_{4p+1,4p+1})=p+1$. Due to $K_{4p+1,4p+1}\subset K_{1,1,4p,4p}\subset K_{1,1,4p+1,4p+1}$, we get $p+1\leq\theta(K_{1,1,4p,4p}) \leq \theta(K_{1,1,4p+1,4p+1})$.  So we have $$\theta(K_{1,1,4p,4p})=\theta(K_{1,1,4p+1,4p+1})=p+1.$$

When $n=4p+3$, from  Theorem \ref{main2} , we have $\theta(K_{2,4p+2,4p+2})=p+2$. Since $K_{2,4p+2,4p+2}\subset K_{1,1,4p+2,4p+2}\subset K_{1,1,4p+3,4p+3}\subset K_{1,1,4(p+1),4(p+1)}$, and the ideas from the previous case establish, we have
$p+2\leq \theta(K_{1,1,4p+2,4p+2})\leq \theta(K_{1,1,4p+3,4p+3})\leq\theta(K_{1,1,4(p+1),4(p+1)})=p+2, $
which shows $$\theta(K_{1,1,4p+2,4p+2})= \theta(K_{1,1,4p+3,4p+3})=p+2.$$

Summarizing the above, the theorem follows.
\end{proof}

\bibliography{bibfile}
\end{document}